\documentclass[11pt]{amsart}
\usepackage{amssymb}
\usepackage{amsmath}
\usepackage{amsthm}
\usepackage[textwidth=16cm, textheight=23cm,marginratio=1:1]{geometry}

\usepackage[english]{babel}
\usepackage[utf8]{inputenc}
\usepackage[T1]{fontenc}
\usepackage[textwidth=0.9in]{todonotes}
\usepackage{xcolor}
\usepackage{tikz}
\usepackage{enumitem}
\usepackage[normalem]{ulem}
\usetikzlibrary{matrix,arrows, cd, calc}  

\newcommand{\rred}[1]{#1}%
\usepackage[pdfborder={0 0 0}]{hyperref}
\hypersetup{
    colorlinks,
    linkcolor={red!80!black},
    citecolor={blue!80!black},
    urlcolor={blue!80!black}
}
\usepackage{setspace}
\usepackage{booktabs, longtable}
\numberwithin{equation}{section}
\newtheorem{theorem}[equation]{Theorem}
\newtheorem{maintheorem}{Theorem}

\newtheorem{corollary}[equation]{Corollary}
\newtheorem{lemma}[equation]{Lemma}

\newtheorem{proposition}[equation]{Proposition}
\theoremstyle{definition}
\newtheorem{definition}[equation]{Definition}
\newtheorem{remark}[equation]{Remark}
\newtheorem{example}[equation]{Example}
\newtheorem{assumption}{Assumption}

\DeclareMathOperator{\coker}{coker}
\DeclareMathOperator{\id}{id}
\DeclareMathOperator{\im}{im}
\DeclareMathOperator{\Spec}{Spec}

\DeclareMathOperator{\Supp}{Supp}

\DeclareMathOperator{\Hom}{Hom}
\DeclareMathOperator{\Ext}{Ext}
\DeclareMathOperator{\Endd}{End}
\DeclareMathOperator{\Ann}{Ann}
\DeclareMathOperator{\diag}{diag}
\newcommand{\tensor}{\otimes}
\newcommand{\onto}{\twoheadrightarrow}
\newcommand{\into}{\hookrightarrow}

\newcommand{\kk}{\Bbbk}%
\newcommand{\spann}[1]{\left\langle #1 \right\rangle}
\newcommand{\BBname}{Bia{\l{}}ynicki-Birula}%

\DeclareMathOperator{\OpQuot}{Quot}%
\DeclareMathOperator{\OpHilb}{Hilb}%
\newcommand{\Quot}[2]{\OpQuot_{#1}^{#2}}%

\newcommand{\Gmult}{\mathbb{G}_{\mathrm{m}}}
\newcommand{\Gbar}{\overline{\mathbb{G}}_{\mathrm{m}}}%
\newcommand{\compo}{\mathcal{Z}}%
\DeclareMathOperator{\Gr}{Gr}
\DeclareMathOperator{\Sym}{Sym}
\DeclareMathOperator{\chr}{char}
\DeclareMathOperator{\GL}{GL}
\DeclareMathOperator{\Mod}{Mod}

\newcommand{\squareZeroComp}[3]{\mathcal{Z}_{#1,#2}^{#3}}
\newcommand{\squareZeroCompUpper}[3]{\mathcal{Z}_{#1,#2}^{#3, \mathrm{upper}}}

\newcommand{\mm}{\mathfrak{m}}%
\newcommand{\OO}{\mathcal{O}}%
\newcommand{\xx}{\mathbf{x}}%

\begin{document}

\title{Components and singularities of Quot schemes and varieties of commuting matrices}
\author{Joachim Jelisiejew and Klemen \v{S}ivic}
\thanks{JJ was supported by Polish National Science Center, project
    2017/26/D/ST1/00755 and by the START
fellowship of the Foundation for Polish Science. K\v{S} is partially supported by Slovenian Research Agency (ARRS), grant numbers N1-0103 and P1-0222.}
\date{\today{}}

\begin{abstract}
    We investigate the variety of commuting matrices.
    We classify its components for any number of matrices of size at most $7$.
    We prove that starting from quadruples of of $8\times 8$ matrices,
    this scheme has generically nonreduced components, while up to degree $7$
    it is generically reduced. Our approach is to recast the problem as
    deformations of modules and generalize an array of methods: apolarity,
    duality and \BBname{} decompositions to this setup.
    We include a thorough review of our methods to make the paper
    self-contained and accessible to both algebraic and linear-algebraic
    communities.  Our results give the corresponding statements for the Quot
    schemes of points, in particular we classify the components of
    $\OpQuot_d(\OO_{\mathbb{A}^n}^{\oplus r})$ for $d\leq 7$ and all $r$, $n$.
\end{abstract}

\maketitle

\tableofcontents

\newcommand{\degM}{d}%
\newcommand{\genM}{r}%
\newcommand{\ambM}{n}%
\newcommand{\Quotmain}{\Quot{\genM}{\degM}}%
\newcommand{\Quotplus}{\Quot{\genM}{\degM,+}}%
\newcommand{\Quotminus}{\Quot{\genM}{\degM,-}}%
\newcommand{\MM}{\mathbb{M}}%
\newcommand{\Matrices}{\MM_{\degM}}%
\newcommand{\CommMatParam}[1]{C_{\ambM}(\MM_{#1})}%
\newcommand{\CommMat}{C_{\ambM}(\Matrices)}%
\newcommand{\Modfin}{\Mod^{\degM}(\mathbb{A}^{\ambM})}%
\newcommand{\prodfamily}{\mathcal{U}}%
\newcommand{\prodfamilystable}{\mathcal{U}^{\mathrm{st}}}%

\section{Introduction}

    The variety $\CommMat$ of $n$-tuples of commuting $d \times d$
    matrices is an object of great interest for linear algebraists, in representation
	theory~\cite{CrawleyBoevey} and in deformation theory. It has applications
    in complexity theory, see below. However, its geometry
	is complicated and surprisingly almost nothing about its components is present in the
    literature, especially when $n\geq 4$.
	The aim of the current work is to exhibit new components and clarify the
    general picture by building a robust toolbox.

    The variety $\CommMat$ is a cone over the
    zero tuple, so it is connected. It has a distinguished \emph{principal component} that
    is the closure of the locus of tuples of diagonalizable matrices. Asking
    whether $\CommMat$ is irreducible can be rephrased as asking whether every
    tuple is a limit of diagonalizable ones.
    The variety $C_2(\Matrices)$ is irreducible for each
    $d$ \cite{MT}.
	In contrast, for
    $\ambM\ge 4$ the variety $C_{\ambM}(\Matrices)$ is irreducible if and only if
    $\degM\le 3$ \cite{Ger, Gur}, and the question of irreducibility of
    $C_3(\Matrices)$ is not solved completely yet: it is reducible for
    $\degM\ge 29$, see \cite{HO}, \cite[p. 238]{NS}, and irreducible for
    $\degM\le 10$, see \cite{Sivic__Varieties_of_commuting_matrices}. In fact, it is known to be irreducible also for
    $\degM=11$, although this result has not been published yet, see \cite[p.
    271]{HO'M}. These irreducibility results hold in characteristic zero.

    To our knowledge literally no components of $\CommMat$ other than the
    principal one are described in the literature. In this article we
    describe all components of
    $\CommMat$ for arbitrary $n$ and for small $d$.
    \begin{maintheorem}\label{ref:numberOfComponentsIntro:thm}
        Let $\kk$ be an algebraically closed field of characteristic zero.
        The number of irreducible components of $\CommMat$ for $d\leq 7$ is as shown in
        Table~\ref{tab:componentTable}. All components have smooth points.
    \begin{table}[h!]
        \centering
        \begin{tabular}{c c c c c c c c c}
            &           $d \leq 2$ & $d=3$ & $d=4$ & $d=5$  & $d=6$ & $d=7$ & $d\gg 0$\\\midrule
            $n\leq 2$ & $1$        & $1$    & $1$  & $1$    & $1$   & $1$   & $1$\\
            $n = 3$ & $1$          & $1$    & $1$  & $1$    & $1$   & $1$   & $\gg 0$\\
            $n = 4$ & $1$          & $1$    & $2$  & $2$    & $2$   & $2$   & $\gg 0$\\
            $n = 5$ & $1$          & $1$    & $2$  & $4$    & $4$   & $8$   & $\gg 0$\\
            $n = 6$ & $1$          & $1$    & $2$  & $4$    & $7$   & $11$   & $\gg 0$\\
            $n \geq 7$ & $1$       & $1$    & $2$  & $4$    & $7$   & $13$   & $\gg 0$\\
        \end{tabular}
        \caption{Number of components of $\CommMat$}
        \label{tab:componentTable}
    \end{table}
    \end{maintheorem}
    The components themselves are also described: the elementary components
    are presented in Section~\ref{ssec:examples} and others are obtained
    using Proposition~\ref{ref:productOfComponents:prop}.

    The study of $\CommMat$ is a classical subject on its own, but our work is also inspired
    by complexity theory, namely by deciding whether a concise ternary tensor has
    minimal border rank. Let $A, B, C$ be $d$-dimensional $\kk$-vector spaces and
    let $T\in A\tensor B\tensor C$.
    Assume that $T$ is \emph{concise}, which means that the inclusion $T(B^{\vee}\tensor C^{\vee})
    \subset A$ is an equality and the same holds for two other contractions.
    Assume moreover that $T$ is
    \emph{$1_A$-generic}, which means that there exists a functional $\alpha\in
    A^{\vee}$ such that $T(\alpha)\in B\tensor C$ has full rank. Using
    $T(\alpha)$ we see that $T$ is isomorphic to a tensor in $A\tensor B\tensor
    B^{\vee}$ so it can be identified with a tuple $\mathcal{E}(T)$ of matrices parameterized by
        $A$, one of them equal to the identity
        matrix~\cite[\S2.1]{Landsberg_Michalek__Abelian_Tensors}. The tensor $T$ satisfies
    Strassen's equations for the minimal border
    rank if and only if the matrices in $\mathcal{E}(T)$
    commute~\cite[Lemma 2.6]{Landsberg_Michalek__Abelian_Tensors}.
    The fundamental
    observation~\cite[Proposition 2.8]{Landsberg_Michalek__Abelian_Tensors} is that $T$
    has minimal border rank if and only if $\mathcal{E}(T)\in C_{d}(\MM_d)$
    lies in the principal component. The identity matrix may be
    discarded from the tuple to obtain $\mathcal{E}'(T)\in C_{d-1}(\MM_d)$ and
    $\mathcal{E}'(T)$ lies in the principal component if and only if
    $\mathcal{E}(T)$ lies in the principal component. From this point of view,
    the nonprincipal components
    of $C_{d-1}(\MM_d)$ yield classes of tensors satisfying Strassen's equations yet having higher
    border rank. This is useful for investigation of equations to higher secant
    varieties of the Segre variety, a major open
    problem~\cite{Landsberg__tensors}.

    We shift focus to the Quot scheme $\Quotmain$ of zero-dimensional, degree $d$
    quotient modules of $S^{\oplus r}$, where $S=\kk[y_1, \ldots ,y_n]$; see
    Section~\ref{ssec:quot_and_commuting} for details. A classification of such
    modules is possible only in very small degrees~\cite{Moschetti_Ricolfi}.
    By the classical ADHM
    construction~\cite[Chapter~2]{nakajima_lectures_on_Hilbert_schemes}
    the
    geometry of $\Quotmain$ is equivalent to the geometry of an open subset of
    $\CommMat$; this subset is the whole $\CommMat$ for $r\geq d$.
    An analysis of components from
    Theorem~\ref{ref:numberOfComponentsIntro:thm} shows that we have the
    following number of components for $\Quotmain = \Quotmain(\mathbb{A}^n)$.
    \begin{table}[h!]
        \centering
        \begin{tabular}{c l l l l l l l l}
            &           $d \leq 2$ & $d=3$ & $d=4$ & $d=5$  & $d=6$ & $d=7$ & $d\gg 0$\\\midrule
            $n\leq 2$ & $1, \ldots $        & $1, \ldots $    & $1, \ldots $  & $1, \ldots $    & $1, \ldots $   & $1, \ldots $   & $1, \ldots $\\
            $n = 3$ & $1, \ldots $          & $1, \ldots $    & $1, \ldots $
            & $1, \ldots $    & $1, \ldots $   & $1, \ldots $   & $\gg
            0$\\
            $n = 4$ & $1, \ldots $          & $1, \ldots $    & $1,2, \ldots $
            & $1,2, \ldots $    & $1,2, \ldots $   & $1,2, \ldots $   & $
            \gg 0$\\
            $n = 5$ & $1, \ldots $          & $1, \ldots $    & $1,2, \ldots $
            & $1,3,4, \ldots $    & $1,3, 4, \ldots $   & $1,4,7,8, \ldots $   & $
            \gg 0$\\
            $n = 6$ & $1, \ldots $          & $1, \ldots $    & $1,2, \ldots $
            & $1,3,4, \ldots $    & $1,4,6,7, \ldots $   & $1,5,9,11, \ldots $   & $
            \gg 0$\\
            $n \geq 7$ & $1, \ldots $       & $1, \ldots $    & $1,2, \ldots $
            & $1,3,4, \ldots $    & $1,4,6,7, \ldots $   & $1,6,10,12,13,
            \ldots $   & $
            \gg 0$\\
        \end{tabular}
        \caption{Number of components of $\Quotmain$. In each entry, consecutive numbers
        correspond to the number of components for $r=1,2, \ldots$ and ``\ldots'' means that the numbers
        stabilize at the value of the last entry. In particular, we see that for $r\geq
        5$ we already have all the components (for $d\leq 7$).}
        \label{tab:componentTableForQuot}
    \end{table}

    We review the basics of ADHM constructions in Section~\ref{ssec:quot_and_commuting}.
    This connection is frequently used to analyse $\Quotmain$, for example~\cite{dosSantos}
    deduced irreducibility of $\OpHilb_{10}(\mathbb{A}^3)$ from the
    aforementioned irreducibility of $C_{3}(\MM_{10})$; see also~\cite{HG}. In this paper we use it
    backwards: we take advantage of the sophisticated commutative algebra tools such as duality for finite free
    resolutions (see Section~\ref{ssec:transposes_and_resolution}) and Green's Linear Syzygy
    Theorem (see proof of Theorem~\ref{cube_zero}) to understand $\Quotmain$
    and then use this knowledge to understand $\CommMat$.

    The Quot scheme $\Quotmain$ is a natural generalization of the Hilbert
    scheme of $d$ points on $\mathbb{A}^n$, the latter just corresponds to
    $r=1$. There
    are two tools of great importance in the analysis of this Hilbert scheme,
    namely
    \begin{itemize}
        \item Macaulay's inverse systems, also known as apolarity, which are used to classify or produce
            explicit examples~\cite{emsalem_iarrobino_small_tangent_space,
                iarrobino_kanev_book_Gorenstein_algebras,
            cartwright_erman_velasco_viray_Hilb8, Landsberg__tensors},
        \item \BBname{} decomposition, which is used to
            find components without the need of describing
            them~\cite{Jelisiejew__Elementary} and understand
            singularities~\cite{Jelisiejew__Pathologies}.
    \end{itemize}
    We generalize both tools to the setting of $\Quotmain$: we
    introduce \emph{apolarity for modules} (Section~\ref{ssec:apolarity}) and
    study the \BBname{}
    decomposition (Section~\ref{sec:BBdecomposition}) for $\Quotmain$.
    The \BBname{} decomposition is quite technical, so in this introduction we discuss only
    apolarity.

    There are at least three ways of explicitly writing down our objects of
    interest:
    \begin{enumerate}
        \item Commuting matrices. In this form the degree $d$ is explicitly
            given as the size. The relations are implicit in the
            commutativity relations. The presentation takes much space.
            Describing explicit deformations is easy, but proving
            irreducibility of loci is tedious.
        \item Modules given by generators and relations. Here the relations
            are explicit and numerous, while the degree is nontrivial to
            compute. Describing explicit deformations or proving
            irreducibility of loci are both tedious.
        \item Modules given by inverse systems, i.e., using apolarity. Here the relations and $d$ are
            both implicit but relatively easy to compute. This presentation
            is compact. Producing explicit deformations is neither
            easy nor hard, but proving irreducibility is easy.
    \end{enumerate}
    Let us contrast the three ways on an explicit example.
    Consider a quadruple of $4\times
    4$ matrices
    \[
        x_1 = \begin{bmatrix}
            0 & 0 & 1 & 0\\
            0 & 0 & 0 & 0\\
            0 & 0 & 0 & 0\\
            0 & 0 & 0 & 0
        \end{bmatrix},\quad
        x_2 = \begin{bmatrix}
            0 & 0 & 0 & 1\\
            0 & 0 & 0 & 0\\
            0 & 0 & 0 & 0\\
            0 & 0 & 0 & 0
        \end{bmatrix},\quad
        x_3 = \begin{bmatrix}
            0 & 0 & 0 & 0\\
            0 & 0 & 1 & 0\\
            0 & 0 & 0 & 0\\
            0 & 0 & 0 & 0
        \end{bmatrix},\quad
        x_4 = \begin{bmatrix}
            0 & 0 & 0 & 0\\
            0 & 0 & 0 & 1\\
            0 & 0 & 0 & 0\\
            0 & 0 & 0 & 0
        \end{bmatrix}.
    \]
    The matrices pairwise commute. We equip $\kk^{4}$ with a $\kk[y_1,
    \ldots ,y_4]$-module structure, where $y_i\cdot v = x_i(v)$ for every
    $i=1,2,3,4$ and $v\in \kk^4$. Denote the resulting module by $M$. It is
    generated by elements $e_3 := (0, 0, 1, 0)^T, e_4 := (0, 0, 0, 1)^T\in \kk^4$ so we can
    write it as a quotient of $F := Se_3 \oplus Se_4$. In fact, we have
    \[
        M \simeq \frac{F}{(y_1 e_4,\  y_2 e_3,\  y_2 e_4 - y_1 e_3,\  y_3
        e_4,\  y_4 e_3,\  y_4 e_4 - y_3 e_3)S  }.
    \]
    \newcommand{\Fdual}{F^{*}}%
    This is the presentation by generators and relations.
    The inverse system is obtained as follows. Consider the
    graded dual module $\Fdual := \bigoplus_i \Hom(F_i, \kk)$. The monomial
    basis of $F$ gives a natural dual basis of $\Fdual$; we denote by
    $z_1z_2^2 e_3^*$ the element of $\Fdual$ dual to $y_1y_2^2 e_3$. One can
    verify that $(y_1 e_4, y_2 e_3, y_2 e_4 - y_1 e_3, y_3
    e_4, y_4 e_3, y_4 e_4 - y_3 e_3)S \subset F$ is equal to the set $(N\cdot S)^{\perp}$, where
    $N \subset \Fdual$ is a linear span of $z_1 e_3^* + z_2e_4^*, z_3 e_3^* + z_4
    e_4^*$ and $(-)^{\perp}$ denotes the orthogonal space is the above duality. We write this as
    \[
        M  \simeq \frac{F}{\left((z_1 e_3^* + z_2e_4^*, z_3 e_3^* + z_4
            e_4^*)S\right)^{\perp}}.
    \]
    We refer the reader to Section~\ref{ssec:apolarity} for further details.

    Having discussed apolarity, we return to considering components.
    The usual method of producing components of $\CommMat$, which we apply
    successfully
    in Theorem~\ref{ref:numberOfComponentsIntro:thm}, is to produce
    an irreducible subvariety $\mathcal{Z}\subset \CommMat$ of dimension $D$ and find a point $\xx\in
    \mathcal{Z}$ such that $\dim T_{\xx}\CommMat = D$. This method requires
    the resulting component $\mathcal{Z}$ to have a smooth point $\xx$. Below, we show that
    $C_4(\MM_8)$ has a generically nonreduced component which thus violates this
    condition.
    While it follows from earlier results that $\CommMat$ is very
    singular when $n, d\gg 0$, see Remark~\ref{rem:MurphysLaw}, the importance
    of Theorem~\ref{ref:mainthm:genericallynonreduced:intro} lies in that the
    nonreducedness appears for small $d$ and $n$ and that we obtain a
    \emph{generically} nonreduced component.
    \begin{maintheorem}\label{ref:mainthm:genericallynonreduced:intro}
        Let $\kk$ be an algebraically closed field of characteristic zero.
        Consider the locus $\mathcal{L}$ of $4$-tuples of $8\times 8$ matrices with nonzero
        entries only in the top right $4\times 4$ corner. Then $(\kk I_4)^4+\overline{\GL_8\cdot
        \mathcal{L}}$
        is a component of $C_4(\MM_8)$ that is generically nonreduced (so it
        has no smooth points). Consequently, the variety $\CommMat$ has
        generically nonreduced components for all $n\geq 4$ and $d\geq 8$.
        In contrast, the variety $\CommMat$ is generically reduced for $d\leq 7$ and all
        $n$. Similarly, the Quot scheme $\Quotmain$ has a generically nonreduced
        component for $n,r\geq4$ and $d\geq 8$ while it is generically reduced
    for $d \leq 7$ and all $r$, $n$.
    \end{maintheorem}
    Exhibiting a generically nonreduced component of a moduli space is subtle,
    as seen in Mumford's celebrated example~\cite[\S13]{HarDeform} see
    also~\cite{Vakil_MurphyLaw, Kass, Ricolfi_nonreduced}. In our
    setup, we use the machinery of differential graded Lie algebras (DGLA) to obtain an explicit
    description of the \emph{primary} obstruction for $\Quotmain$. This allows us to
    obtain the quadratic part of equations of the complete local ring at a point of
    $\mathcal{L}$. The locus $\mathcal{L}$ is a sink for an appropriate torus
    action, so this ring is a completion of a graded ring, hence the obtained part
    accounts for all quadratic relations. We prove that they cut out the
    ring of dimension equal to the dimension of $(\kk I)^4+\overline{\GL_8\cdot
        \mathcal{L}}$, which concludes the proof for $\Quot{4}{8}$.
    The result for $C_4(\MM_8)$ follows from ADHM construction. The whole
    approach seems to be new and useful also in other cases.

    The geometry of commuting matrices is still largely unexplored. One open
    question is to find the smallest $d$ such that $C_3(\MM_d)$ is reducible. Another
    question, inspired by Table~\ref{tab:componentTable} is: there is a
    natural ``add zero
    matrices'' map from the set of components of $C_n(\MM_d)$ to the set of
    components of $C_{n+1}(\MM_{d})$. Is this map a bijection for all $n\geq
    d$?
    It is possible to see that this map is a bijection for all $n$ larger
    than the maximal dimension of a commutative subalgebra of $\MM_d$, in
    particular for $n\geq \lfloor(d/2)^2\rfloor+1$, see~\cite{Schur}. For
    related ideas, see~\cite[Introduction]{Levy_Ngo_Sivic}.
    Finally, despite the challenges rising from Theorem~\ref{ref:mainthm:genericallynonreduced:intro} it
    is likely that one can classify components of $\CommMat$ for $d=8$ and perhaps
    further.
    The general tools summarized and developed in this article would
    be very useful in investigating these questions.

    Let us briefly summarize the contents of this article. In
    Section~\ref{sec:prelims} we summarize the basics of $\CommMat$,
    $\Quotmain$ and their connection. Most of these results are folklore but
    frequently references are missing or are available only in special
    cases or in language
    inaccessible to the linear-algebra community. In
    Sections~\ref{sec:strutural}-\ref{sec:BBdecomposition} we present advanced
    general tools, which are new: apolarity, primary obstructions for Quot,
    concatenation maps, \BBname{} decompositions etc. This is the core part in terms of general theory.
    Finally, in Section~\ref{sec:degeight} we apply these results to obtain
    our main theorems. An appendix contains a brief introduction to functors
    of points.

\section{Notation}
\begin{center}
    \begin{tabular}{@{}l l l @{}}
    Symbol && Explanation \\ \midrule
    $\degM$ && degree of the module = size of matrices\\
    $\ambM$ && number of matrices = number variables in polynomial ring $S$\\
    $\genM$ && number of generators of the module\\
    $x_1, \ldots , x_{\ambM}$ && commuting matrices\\
    $y_1, \ldots , y_{\ambM}$ && variables of polynomial ring $S$\\
    $V$ && a fixed $d$-dimensional vector space over $\kk$
\end{tabular}
\end{center}

\section*{Acknowledgements}

    We very much thank Nathan Ilten for coding and sharing an experimental
    version of his \texttt{VersalDeformations} package for \emph{Macaulay2}
    which allows one to compute deformations of modules. We thank Joseph
    Landsberg, Maciej Ga{\l{}}{\k{a}}zka, Hang Huang\rred{, and the referee} for suggesting several
    improvements to the text.

    \section{Preliminaries}\label{sec:prelims}

    Let $\kk$ be an algebraically closed field.
    So far we do not put any restrictions on its characteristic (these will be
    explicitly included in Sections~\ref{sec:BBdecomposition}-\ref{sec:degeight}), although in examples we assume
    characteristic zero.
    Let $S = \kk[y_1, \ldots , y_n]$.
    We will now define several spaces that parameterize $S$-modules.
    In the table below, \emph{modules} means zero-dimensional $S$-modules of degree $d$.
    Here and elsewhere the \emph{degree} of a module $M$ is just its
    dimension as a $\kk$-linear space.
        The spaces are linked by the forgetful functors and form the diagram,
        see Table~\ref{table:objectsinvolved}.
        For more details on the maps, see
        Section~\ref{ssec:quot_and_commuting}.

        \begin{table}[h!]
                \begin{tabular}{@{} l l l @{}}
                \mbox{space} && \mbox{objects}\\ \midrule
                $\Modfin$ && \mbox{modules}\\
                $\Quotmain$ && \mbox{modules with fixed $\genM$ generators}\\
                $\CommMat$ && \mbox{modules with fixed basis}\\
                $\prodfamilystable$ && \mbox{modules with fixed basis and
                fixed sequence of $\genM$ generators}\\
                &&
            \end{tabular}
            \[
                \begin{tikzcd}
                    \prodfamilystable \arrow[rrr, "\mbox{\tiny smooth fib.dim. }
                    rd"]\arrow[d, "/\GL_d"] &&& \CommMat\arrow[d,
                    "/\GL_d"]\\
                    \Quotmain \arrow[rrr] &&& \Modfin
                \end{tikzcd}
            \]
            \caption{Moduli spaces}\label{table:objectsinvolved}
        \end{table}

        \subsection{Commuting matrices}\label{ssec:commMat} Let $\Matrices$ be the affine space $\Hom(V,
V)$ for a fixed $\degM$-dimensional vector space $V$.
By $I_d$ we denote the $d\times d$ identity matrix.
Let
    \[
        C_{\ambM}(\MM_d) = \left\{ (x_1, \ldots ,x_{\ambM})\in
        \MM_d^{\ambM}\ |\ \forall{i,j}:  x_ix_j = x_jx_i \right\}
    \]
    be the \emph{variety of $\ambM$-tuples of $d\times d$ commuting matrices}.
    More precisely,
    we define $C_{\ambM}(\MM_d)$ as the subscheme cut out by the quadratic equations coming
    from $x_ix_j -x_jx_i = 0$. We show that the word \emph{variety} even
    though accepted in the literature is a slight abuse:
    it is known that $\CommMat$ is in general reducible and
    in Section~\ref{ssec:nonreducedComponent} we prove that it has
    generically nonreduced components.
    \begin{lemma}\label{ref:tangentspace:lem}
        The tangent space to the scheme $C_{\ambM}(\Matrices)$ defined above
        is
        \[
            T_{(x_1,\ldots ,x_{\ambM})}C_{\ambM}(\Matrices)=\{(z_1,\ldots ,z_{\ambM})\in
        \Matrices^n;\ \forall i,j: [x_i,z_j]+[z_i,x_j]=0\}.
        \]
        For the variety $(C_n(\Matrices))_{\mathrm{red}}$ the
        tangent space to it at $(x_1,\ldots ,x_n)$ is contained in the space
        above.
    \end{lemma}
    \begin{proof}
        We use the description of the tangent space via maps
        from $\Spec(\kk[\varepsilon]/\varepsilon^2)$, see for
        example~\cite[VI.1.3]{eisenbud_harris}.
        The tangent space to $C_{\ambM}(\Matrices)$ at the point $(x_1, \ldots ,x_{\ambM})$ is
        the vector space of tuples $(z_1, \ldots ,z_{\ambM})\in \Matrices^{\ambM}$ such that $(X_1,
        \ldots ,X_{\ambM}) = (x_1 + \varepsilon
        z_1, x_2 + \varepsilon z_2, \ldots ,x_{\ambM} + \varepsilon
        z_{\ambM})$ satisfies the equations $X_iX_j = X_jX_i$ modulo
        $\varepsilon^2$. But
        \begin{align*}
            X_iX_j - X_jX_i &= (x_i + \varepsilon z_i)(x_j + \varepsilon z_j) -
            (x_j + \varepsilon z_j)(x_i + \varepsilon z_i) \equiv\\ &\equiv x_ix_j - x_jx_i +
            \varepsilon(z_ix_j + x_iz_j - x_jz_i - z_jx_i) =
            \varepsilon([x_i, z_j] + [z_i, x_j])\pmod{\varepsilon^2},
        \end{align*}
        so the equations are satisfied precisely when $[x_i, z_j] + [z_i,
        x_j] = 0$. Finally, the tangent space to the underlying variety is
        always contained in the tangent space to the scheme.
    \end{proof}

    We now show that $C_n(\Matrices)$ indeed corresponds to $S$-modules with a
    basis.
    \begin{lemma}\label{ref:matricesaremoduleswithbasis}
        The points of $C_n(\Matrices)$ are in bijection with $S$-modules with
        a fixed $\kk$-linear basis.
    \end{lemma}
    \begin{proof}
        Let $V = \kk e_1 \oplus \ldots \oplus \kk e_d$. Having a point
        $(x_1, \ldots ,x_n)\in C_n(\Matrices)$, we define an $S$-module
        structure on $V$ by $y_i\cdot v = x_i(v)$. The equations $x_ix_j =
        x_jx_i$ imply that $y_i\cdot (y_j\cdot v) = y_j\cdot (y_i\cdot v)$ for
        all $i,j$, so indeed we get an $S$-module.
        Conversely, for an $S$-module $M$ with a basis $\mathcal{B}$ and
        $i=1,2, \ldots ,n$
        consider the multiplication by $y_i$ as an endomorphism of $M$ and
        let $x_i$
        be its matrix written in basis
        $\mathcal{B}$. Since $S$ is commutative, the matrices $x_1,
        \ldots ,x_n$ commute.
    \end{proof}

    \subsection{Quot and commuting matrices}\label{ssec:quot_and_commuting}
    Let $S = \kk[y_1, \ldots ,y_n]$ be a polynomial ring.
    An $S$-module $M$ has \emph{finite degree} if the $\kk$-vector space
    $M$ has finite dimension. We say that $M$ has \emph{degree} $d$ if the
    $\kk$-vector space $M$ has dimension $d$. Fix a free $S$-module $F = Se_1 \oplus Se_2 \oplus
    \ldots \oplus Se_{\genM}$ of rank $\genM$.
    The \emph{Quot scheme $\rred{\Quotmain}$ of points on
    $\mathbb{A}^{\ambM} = \Spec(S)$} parameterizes degree $d$ quotient modules of the
    $S$-module $F$.
    In other words, a $\kk$-point of
    $\rred{\Quotmain}$ is a quotient $F/K$ of $S$-modules that is a $\degM$-dimensional
    vector space.  For a quotient $F/K$ we denote by $[F/K]$ the corresponding
    point in  $\rred{\Quotmain}$. To give such a quotient is the same as to give an $S$-module
    together with a fixed set of $\genM$ generators.

    To give the Quot scheme the topological space and scheme structures we
    need to define it using functorial language, see appendix for details.
    However, the willing reader can take this for granted.

    \begin{example}\label{ex:rankOneGivesQuotEqHilb}
        When $\genM = 1$, we are speaking about quotients of $S^{\oplus 1}$, so
        about modules of the form $S/I$. Therefore $\rred{\Quot{1}{\degM}} =
        \OpHilb_{\degM}(\mathbb{A}^n)$.
    \end{example}

    The schemes $\CommMat$ and $\Quotmain$ are tightly connected by the ADHM
    construction
    \cite{nakajima_lectures_on_Hilbert_schemes, HG,
    Baranovsky__construction_of_Quot}, named after~\cite{ADHM}, which we now recall. In terms of
    Table~\ref{table:objectsinvolved}, we just take a module, fix its basis
    and $\genM$ generators and consider the two projections: forgetting about the
    generating set and forgetting about the basis.
    We now explain this more carefully.
    Throughout the article, a key fact to keep in mind is that the spaces in
    Table~\ref{table:objectsinvolved} are very singular, but, as we prove
    below, the maps are smooth, so intuitively these objects are ``singular in the same
    way''.

    Fix a number $\genM$ and consider the product
    \[
        \prodfamily := \CommMat \times V^{\genM}.
    \]
    For every tuple $T = (x_1, \ldots ,x_{\ambM}, v_1, \ldots , v_{\genM})$ in this
    product we may
    take the smallest subspace $W \subset V$ containing $v_1$, \ldots ,$
    v_{\genM}$ and preserved by operators $x_1$, \ldots , $x_{\ambM}$. We say
    that $T$ is \emph{stable} if $W = V$. In other words, $T$ is stable iff
    $v_1, \ldots , v_{\genM}$ generate $V$ as a $\kk[x_1, \ldots ,x_{\ambM}]$-module.
    For every $i$ the power $x_i^{\degM}$ is a linear combination of smaller
    powers of $x_i$, so $T$ is stable if and only if the elements
    \[
        \left\{x_1^{a_1}\circ x_2^{a_2}\circ \ldots \circ x_n^{a_n}(v_j)\ \ |\ a_i <
        \degM,\
    j=1,.., \genM\right\}
    \]
    span $V$. This is an open condition, so the locus $\prodfamilystable$ of
    stable tuples is open.
    To a stable tuple $T$ we associate a point of $\Quotmain$ as
    follows.
    \newcommand{\pibarM}{\overline{\pi_{M}}}%
    \newcommand{\pibarMprim}{\overline{\pi_{M'}}}%
    \begin{enumerate}
        \item the tuple $(x_1, \ldots ,x_{\ambM})$ induces an $S$-module
            structure on $V$, where $y_i(v) = x_i(v)$ for
            all $v\in V$; here we use commutativity of $x_i$ and $x_j$. We
            denote the obtained $S$-module by $M$ and call it the
            \emph{module corresponding to $(x_1, \ldots ,x_{\ambM})$}.
        \item the stability of a tuple $(x_1, \ldots ,x_{\ambM}, v_1, \ldots ,
            v_{\genM})$ means that the $S$-module $M$ is generated by $v_1,
            \ldots ,v_{\genM}$. There is a unique epimorphism of $S$-modules
            $\pi_M\colon F\to M$ that sends $e_i$ to $v_i$ for $i=1, \ldots
            ,\genM$. The map $\pi_M$ restricts to an isomorphism of
            $S$-modules $\pibarM\colon \frac{F}{\ker \pi_M}\to M$.
    \end{enumerate}
    \begin{lemma}\label{ref:descriptionOfCommMat:lem}
            The map $(x_1, \ldots ,x_n, v_1, \ldots , v_{\genM})\mapsto
            (\frac{F}{\ker \pi_M}, \pibarM)$ is a bijection between the $\kk$-points of
            $\prodfamilystable$ and the set
            \[
                \left\{ \left(\frac{F}{K}, \varphi\right)\ \Big|\
                [F/K]\in\Quotmain,\ \varphi\colon F/K\to V\mbox{ is a $\kk$-linear
                isomorphism} \right\}.
            \]
    \end{lemma}
    \begin{proof}
        It remains to construct an inverse.
        The multiplication by $y_i$ is a linear map $\mu_{i}\colon F/K\to
        F/K$, so $\varphi\circ\mu_i\circ\varphi^{-1}\colon V\to V$ is a
        $\kk$-linear endomorphism of $V$, so an element of $\Matrices$ which
        we denote by $x_i$.
        Let $v_j = \varphi(\overline{e_j})$ for $j=1,2, \ldots , \genM$.
        Since the module $F/K$ is generated by images of $e_1,
        \ldots ,e_{\genM}$, the tuple $(x_1, \ldots ,x_{\ambM}, v_1, \ldots
        ,v_{\genM})$ is stable.
    \end{proof}
    Lemma~\ref{ref:descriptionOfCommMat:lem} gives us a map
    $\prodfamilystable\to \Quotmain$ that first transforms $(x_1, \ldots ,x_n,
    v_1, \ldots , v_r)$  to $(F/K, \varphi)$ and then forgets about the
    isomorphism $\varphi$. The fiber of $\prodfamilystable\to \Quotmain$
    consists of all possible $\varphi$, whence it is in bijection with
    $\GL_d$.

    \begin{example}
        Consider the case $\degM = \ambM = \genM = 1$. The closed points of
        $\Quotmain$ are modules of the form $\kk[y_1]/(y_1 - \alpha)$ for
        $\alpha\in \kk$.
        The space $V$ is one-dimensional.
        The closed points of $\prodfamily$ are pairs $(x_1, v_1)$, where $v_1\in
        V$ and $x_1 \in \MM_1  \simeq \kk$.
        The pair $(x_1, v_1)$ is stable if and only if $v_1\neq 0$.
        In this case $x_1v_1 = \alpha v_1$ for some $\alpha$.
        The projection $\pi_M\colon \kk[y_1]\to M$ sends $y_1 -
        \alpha$ to zero and so
        \[
            \ker \pi_M = (y_1 - \alpha)
        \]
        as expected. For every $\lambda\in \kk^*$, under the identification
        from Lemma~\ref{ref:descriptionOfCommMat:lem} the pairs $(x_1, v_1)$ and
        $(x_1, \lambda v_1)$ are sent to the same module $\kk[y_1]/(y_1 - \alpha)$, though with
        different isomorphisms $\varphi$. The fiber of the above map
        $\prodfamilystable\to \Quotmain$ over $\kk[y_1]/(y_1 - \alpha)$
        is equal to $\{(x_1, \lambda v_1)\ |\ \lambda\in \kk^*\}$, so it
        is isomorphic to $\kk^*$.
    \end{example}

        We have a natural action of $\GL(V)$ on $\CommMat$, $\prodfamily$ and
        $\prodfamilystable$. Namely, $\GL(V)$ acts on $\Matrices$ by
        conjugation, so it also acts on $\CommMat$ by simultaneous
        conjugation:
        \[
            g\circ(x_1, \ldots ,x_{\ambM}) = (gx_1g^{-1}, \ldots ,
            gx_{\ambM}g^{-1})
        \]
        for $g\in \GL(V)$. Similarly, $\GL(V)$ acts on $V^{\times \genM}$ by
        $g(v_1, \ldots ,v_{\genM}) = (gv_1, \ldots , gv_{\genM})$, so we
        obtain an action on $\prodfamily$ by
        \[
            g\circ(x_1, \ldots ,x_{\ambM}, v_1, \ldots , v_{\genM}) =
            (gx_1g^{-1}, \ldots , gx_{\ambM}g^{-1}, gv_1, \ldots ,
            gv_{\genM})
        \]
        for $g\in \GL(V)$. If the tuple $T = (x_1, \ldots ,x_{\ambM}, v_1, \ldots
        , v_{\genM})$ was stable, then also the tuple $g\circ T$ is stable, so
        the action on $\prodfamily$ restricts to an action on
        $\prodfamilystable$.

        \begin{lemma}\label{ref:GLVaction:lem}
            In the description from Lemma~\ref{ref:descriptionOfCommMat:lem},
            the $\GL(V)$-action on $\prodfamilystable$ is given by
            \[
                g\circ\left( \frac{F}{K}, \varphi \right) = \left(
                \frac{F}{K}, g\circ\varphi
                \right).
            \]
        \end{lemma}
        \begin{proof}
            Let $T = (x_1, \ldots ,x_{\ambM}, v_1, \ldots , v_{\genM})$ be a
            stable tuple and $M$ be the associated $S$-module and epimorphism
            $\pi_M\colon F\to M$ so that $\pi_{M}(e_j) = v_j$. Then $y_i\cdot
            m = x_i(m)$ for every $m\in M$.
            Let $T' = g\circ T$ and $M'$ be the associated $S$-module with
            epimorphism $\pi_{M'}\colon F\to M'$. Then
            $y_i \circ m' = gx_ig^{-1}(m')$ and $\pi_{M'}(e_j) = gv_j$, so the following diagram is
            commutative
            \[
                \begin{tikzcd}
                    F\arrow[r, two heads, "\pi_M"]\arrow[rd, two heads, "\pi_{M'}"'] &
                    M\arrow[d, "g\cdot (-)"]\arrow[r, "y_i"] & M\arrow[d, "g\cdot (-)"]\\
                    & M'\arrow[r, "y_i"] & M'.
                \end{tikzcd}
            \]
            Since $g\cdot (-)$ is an isomorphism, we have $\ker \pi_{M} = \ker
            \pi_{M'}$, which exactly means that the action of $g$ on the quotients of $F$ is trivial.
            Moreover, the isomorphisms $\pibarM$ and $\pibarMprim$ satisfy
            $\pibarMprim = g\circ \pibarM$.
        \end{proof}
    The set-theoretic map defined after
    Lemma~\ref{ref:descriptionOfCommMat:lem} is a shade of a richer
    structure: a morphism of schemes. To construct it is a matter of
    introducing the language of functors; we defer this to the appendix but state
    the result here.
    \begin{proposition}\label{ref:prodfamilybundleOverQuotmain:prop}
        There exists a map of schemes $p\colon \prodfamilystable\to \Quotmain$
        such that
        \[
            p(x_1, \ldots ,x_n, v_1, \ldots ,v_r) = [F/K]
        \]
        in the
        notation above.
        The $\GL(V)$-action defined on points of $\prodfamilystable$ in
        Lemma~\ref{ref:GLVaction:lem} gives rise to a morphism of schemes
        $\GL(V) \times \prodfamilystable\to \prodfamilystable$, i.e., to an
        algebraic $\GL(V)$-action.
        This $\GL(V)$-action and $p$ make $\prodfamilystable$ a principal $\GL(V)$-bundle over $\Quotmain$.
        Explicitly, this means that there exists an open cover $U_i$ of
        $\Quotmain$ and $\GL(V)$-equivariant isomorphisms of schemes $p^{-1}(U_i) \simeq
        \GL(V)\times U_i$ over $U_i$.
    \end{proposition}
    \begin{proof}
        See Corollary~\ref{ref:prodfamilybundleOverQuotmain:cor}.
    \end{proof}
    \begin{remark}\label{rem:MurphysLaw}
        The smooth maps $\Quotmain \leftarrow \prodfamilystable \rightarrow
        \CommMat$ can be used to show a Murphy's Law type statement for
        $\CommMat$. Namely by~\cite{Jelisiejew__Pathologies} the Hilbert scheme
        of points has almost arbitrary pathologies. The Hilbert scheme is just
        $\Quotmain$ for $r=1$ and so using the smooth maps
        one can deduce (similarly as in~\cite[\S5]{Jelisiejew__Pathologies})
        that $\CommMat$ has bad singularities, for example it is nonreduced.
        However, these singularities appear for $n\geq 16$ and very high
        $d$.
    \end{remark}
    \begin{lemma}\label{ref:tangentSpaceToQuot:lem}
        Let $F/K$ be a quotient module.
        The tangent space to $\Quotmain$ at $[F/K]$ is isomorphic to $\Hom_S(K,
        F/K)$. Let $(x_1, \ldots ,x_n)\in \CommMat$ be a corresponding
        element. Then we have
        \begin{equation}\label{eq:tangentSpaces}
            \dim T_{[F/K]} \Quotmain = rd - d^2 + \dim T_{(x_1, \ldots ,x_n)}
            \CommMat.
        \end{equation}
    \end{lemma}
    \begin{proof}
        The bijection $T_{[F/K]}\Quotmain  \simeq \Hom(K, F/K)$ is classical,
        see~\cite[Proposition~4.4.4]{Sernesi__Deformations}
        or~\cite[Theorem~6.4.5(b)]{fantechi_et_al_fundamental_ag} or adapt the
        proof of~\cite[Theorem~10.1]{Stromme__Hilbert_schemes}.
        To compute the tangent space dimension, consider the maps
        from $\prodfamilystable$ to $\Quotmain$, $\CommMat$ as in
        Table~\ref{table:objectsinvolved}.
        Take any point $u = (x_1, \ldots ,x_n,
        v_1, \ldots , v_r)$ mapping to $[F/K]\in \Quotmain$.
        By construction the map
        $\prodfamilystable\to \CommMat$ is a composition of the open embedding
        of $\prodfamilystable$ into $\prodfamily = \CommMat \times V^{\genM}$ with
        the projection of $\CommMat \times V^{\genM}$ onto the first coordinate.
        By Proposition~\ref{ref:prodfamilybundleOverQuotmain:prop}, locally near
        $u$ the map $\prodfamilystable\to \Quotmain$ is a projection
        $\GL_d\times U\to U$. The equality~\eqref{eq:tangentSpaces} follows from
        comparing the dimensions of $T_u\prodfamilystable, T_{[F/K]}\Quotmain$
        and $T_{(x_1, \ldots ,x_n)} \CommMat$.
    \end{proof}

    \subsection{Support and eigenvalues}\label{ssec:dictionary}

    For a finite degree module $M$, we define its support as the set of
    maximal ideals $\mm\subset S$ such that $M/\mm M\neq 0$. The support is
    equal to the set of maximal ideals containing the ideal $\Ann(M) =
    \{s\in S\ |\ sM = 0\}$, as we explain below (this also shows that
    our definition of support agrees with the usual one). On the one hand, if
    $\mm$ is in the support then the ideal $\Ann(M) + \mm$ annihilates a
    nonzero module $M/\mm M$ and so
    $\Ann(M)+\mm\neq (1)$, so $\Ann(M)\subset \mm$. On the other hand, if $\mm$
    is not in the support then $M = \mm M$ so by Nakayama's
    lemma~\cite[Tag~00DV]{stacks_project} the intersection $(1-\mm)\cap
    \Ann(M)$ is nonempty, so $\Ann(M)\not\subset \mm$.

    \begin{lemma}\label{ref:supportContainedInEigenvalues:lem}
        Let $(x_1, \ldots ,x_n)\in \CommMat$ and let $M$ be the associated
        $S$-module. Suppose that $(\lambda_1, \ldots ,\lambda_n)\in \kk^n$ are
        such that $(y_1 - \lambda_1, \ldots , y_n - \lambda_n)$ is in the
        support of $M$. Then for every $i$ the element $\lambda_i$ is an
        eigenvalue of the matrix $x_i$.
    \end{lemma}
    \begin{proof}
        If for some $i$ the element $\lambda_i$ is not an eigenvalue of $x_i$,
        then $x_i - \lambda_i I$ is invertible. This translates to the
        multiplication by $y_i - \lambda_i$ on $M$ being an isomorphism. It
        follows that $M = (y_i - \lambda_i)M$ so $M/(y_i - \lambda_i)M = 0$
        and consequently $M/(y_1-\lambda_1, \ldots ,y_n-\lambda_n)M = 0$ so
        $(y_i - \lambda_i)_{i=1,2, \ldots, n}$ is not in the support.
    \end{proof}
    \begin{remark}\label{ref:cexample:rmk}
        The converse of Lemma~\ref{ref:supportContainedInEigenvalues:lem} does
        not hold. For example consider a pair of diagonal $2\times 2$ matrices
        $(x_1, x_2) = (\diag(0, 1), \diag(0, 1))$. The corresponding module is
        $M = \kk[y_1, y_2]/(y_1, y_2) \oplus \kk[y_1, y_2]/(y_1 - 1, y_2 - 1)$ so
        $(y_1-1,y_2)$ is not in the support of $M$ even though $0$ and $1$ are
        eigenvalues of both matrices.
    \end{remark}
    Despite Remark~\ref{ref:cexample:rmk}, we do know that the two extremal
    situations: where the module is supported at single point
    or when the matrices are (up to adding multiple of identity) nilpotent,
    actually coincide.
    \begin{lemma}\label{ref:supportVsEigenvalues:lem}
        Let $(x_1, \ldots ,x_n)\in \CommMat$ and let $M$ be the associated
        $S$-module. Then the support of $M$ consists of one point precisely
        when every $x_i$ has only one eigenvalue.
    \end{lemma}
    \begin{proof}
        \def\mm{\mathfrak{m}}
        Fix $i$. The support of $M$ consists of maximal ideals in $S$. By
        Nullstellensatz such
        ideals have the form $(y_1 - \lambda_1, y_2 - \lambda_2, \ldots , y_n
        - \lambda_n)$ for a point $(\lambda_1, \lambda_2,  \ldots ,
        \lambda_n)\in \kk^n$. We identify each such ideal with the
        corresponding point, so the support of $M$ becomes a subset of
        $\kk^n$.
        For any $\lambda_i\in \kk$ the following are equivalent:
        \begin{enumerate}
            \item $\lambda_i$ is an eigenvalue of $x_i$,
            \item $x_i-\lambda_i I$ is not invertible,
            \item $(y_i-\lambda_i) M\neq M$,
            \item $M/(y_i -\lambda_i)M \neq 0$,
            \item there exists a maximal ideal $\mm$ containing $y_i -
                \lambda_i$ and such
                that $M/\mm M \neq 0$. Indeed $\mm$ can be chosen as any
                element of the support of the nonzero module $M/(y_i -
                \lambda_i)M$,
            \item the support of $M$ intersects the hyperplane $y_i =
                \lambda_i$.
        \end{enumerate}
        The support of $M$ is a point iff for every $i$ the last
        condition is satisfied by at most one $\lambda_i$. This happens iff for
        every $i$ the matrix $x_i$ has only one eigenvalue.
    \end{proof}

    \begin{remark}
        In the arguments below we use Hilbert functions of modules and Jordan
        block types for matrices. While those structures are connected, it is
        not a straightforward connection as it requires Lefschetz-type
        theorems, see for example~\cite[Prop.~3.5,
        Prop.~3.64]{Harima_et_al__Lefschetz_properties},
        \cite{Harima_Migliore__Lefschetz}
        or~\cite[Prop.~2.10, Lemma~2.11]{Iarrobino_Marques_McDaniel__Artinian_algebras_and_Jordan_type}.
    \end{remark}

    \subsection{Components}\label{ssec:components}

    We now discuss that components of $\CommMat$ and $\Quotmain$ are in
    bijection for $r\geq d$. We begin with the following general lemma.
    \begin{lemma}\label{ref:componentBijectionAbstract:lem}
        Let $p\colon U\to X$ be a surjective morphism of finite type
        schemes. Assume that for an irreducible component $W\subset X$ the preimage $p^{-1}(W)$ is
                irreducible.
        Then the maps $W\mapsto p^{-1}(W)$ and $Z\mapsto
        \overline{p(Z)}$ give a bijection between irreducible components of $U$
        and $X$.
    \end{lemma}
    \begin{proof}
        Let $Z_1, \ldots , Z_k$ be the irreducible components of $U$. Their
        images in $X$ are
        irreducible, so each of them lies in some irreducible component of
        $X$. Let $W$ be any irreducible component of $X$ and let $I_W\subset
        \{1, \ldots , k\}$ consists of those $i$ for which
        $\overline{p(Z_i)}\subset W$. Since $p$ is surjective, the
        closed subsets $\overline{p(Z_i)}$ cover $X$ so there exists an
        $i_0\in I_W$ such that $\overline{p(Z_{i_0})} = W$. If $I_W\neq
        \{i_0\}$ then $p^{-1}(W)$ contains at least two components of $U$
        contradicting the assumption. Therefore $I_W = \{i_0\}$ and this
        proves the claim.
    \end{proof}
    \begin{lemma}\label{ref:componentBijectionAssumptions:lem}
        Both maps $p_1\colon \prodfamilystable\to p_1(\prodfamilystable)
        \subset \CommMat$ and $p_2\colon \prodfamilystable\to
        \Quotmain$ in Table~\ref{table:objectsinvolved} satisfy the assumptions of
        Lemma~\ref{ref:componentBijectionAbstract:lem}. These maps also map
        open sets to open sets.
    \end{lemma}
    \begin{proof}
        The map $p_1$ is an open embedding into a product $\CommMat\times
        \mathbb{A}^{rd}$ composed with projection, so it is flat, therefore it
        maps open subsets to open subsets~\cite[Tag~039K]{stacks_project}. In
        particular $p_1(\prodfamilystable)$ is open. Let
        $W\subset p_1(\prodfamilystable)$ be irreducible. Then $p_1^{-1}(W)$
        is a nonempty open
        subset of the irreducible scheme $W\times \mathbb{A}^{rd}$, so it is
        irreducible as well. This concludes the proof for $p_1$.

        Proposition~\ref{ref:prodfamilybundleOverQuotmain:prop} implies that
        there is an open cover $\{V_i\}$ of $\Quotmain$ such that for all $i$ the map
        $p_2^{-1}(V_i)\to V_i$ is isomorphic to the second projection
        $\GL_d\times V_i\to V_i$. In particular $p_2$ is flat, so it maps
        open subsets to open subsets. Let $W\subset \Quotmain$ be irreducible
        and suppose $p_2^{-1}(W)$ is reducible. Then there are two disjoint
        open subsets $U_1, U_2\subset p_2^{-1}(W)$. Pick any $V_1, V_2$ from the
        above open cover (if necessary, refining it) such that $V_i \subset p_2(U_i)$ for
        $i=1,2$. The set $W$ is irreducible and $V_1, V_2 \subset W$ so $V_1\cap V_2$ is
        nonempty and irreducible. It follows that $P := p_2^{-1}(V_1\cap V_2)$ is nonempty,
        isomorphic to $\GL_d \times (V_1\cap V_2)$ so irreducible, and
        $U_1$ and $U_2$ intersect $P$, so $U_1\cap P$, $U_2\cap P$ are
        nonempty open
        inside irreducible $P$, hence intersect, so $U_1$ and $U_2$ intersect, a
        contradiction.
    \end{proof}

    Lemma~\ref{ref:componentBijectionAbstract:lem} and
    Lemma~\ref{ref:componentBijectionAssumptions:lem} imply that there are
    bijections between irreducible components of $\Quotmain$, of $\prodfamilystable$, and of
    $p_1(\prodfamilystable)\subset \CommMat$. The components
    of the open subset $p_1(\prodfamilystable)$ are a subset of components of
    $\CommMat$ so we obtain the following
    \[
         \{\mbox{ irr.comp. of }\Quotmain\}\simeq \{\mbox{ irr.comp. of
         }\prodfamilystable\} \simeq \left\{ \mbox{ irr.comp. of
         $p_1(\prodfamilystable)$}\right\}\into \{\mbox{ irr.comp. of }\CommMat\}.
    \]
    For $r\geq d$ the map $\prodfamilystable\to \CommMat$ is surjective; just
    pick $v_1, \ldots ,v_r$ containing a basis of $V$. In this case we obtain
    a bijection between components of $\Quotmain$ and $\CommMat$.
    \begin{example}
        The assumption $r\geq d$ is sharp.
        Consider the zero tuple $(x_1, \ldots ,x_n) = (0, 0, \ldots ,0)$.
        For $v_1, \ldots ,v_r\in V$
        the tuple $(x_1, \ldots ,x_n, v_1, \ldots ,v_r)$ is stable if and only
        if $v_1, \ldots ,v_r$ span $V$. In particular this forces $r\geq d$.
        It follows that $\prodfamilystable\to \CommMat$ is
        surjective if and only if $r\geq d$.
        If $r<d$ then $\CommMat$ may have more components than $\Quotmain$.
        For $(n,d,r) = (4, 4, 1)$ the Quot scheme is the Hilbert scheme
        (Example~\ref{ex:rankOneGivesQuotEqHilb}) which
        is
        irreducible~\cite{cartwright_erman_velasco_viray_Hilb8}. In contrast,
        the variety of commuting matrices is reducible, see
        Example~\ref{ex:squareZero4x4}.
    \end{example}
    For a component $\mathcal{Z}^{\OpQuot}\subset \Quotmain$ the dimension
    of the corresponding component $\mathcal{Z}^{C}\subset \CommMat$ satisfies
    $\dim \mathcal{Z}^{\OpQuot} = rd - d^2 + \dim \mathcal{Z}^{C}$ because we
    take preimages and images under smooth maps whose fibres have dimension
    $rd$ and $d^2$ respectively (the details of this argument are for example in the proof of
    Lemma~\ref{ref:tangentSpaceToQuot:lem}).
    Lemma~\ref{ref:tangentSpaceToQuot:lem} gives the same equality for
    tangent spaces, so we conclude that a point $[F/K]\in
    \mathcal{Z}^{\OpQuot}$ is smooth it and only if any corresponding point
    $(x_1, \ldots ,x_n)\in \CommMat$ is smooth.

        The \emph{principal component} of $\CommMat$ is the closure of the
        locus of $n$-tuples of diagonalizable matrices.
    It contains an open subset consisting of $n$-tuples where the first matrix
    has $d$ distinct eigenvalues. A matrix that commutes with such a matrix
    $x$, can be uniquely written in the form $p(x)$ where $p$ is a
    polynomial of degree at most $d-1$, see \cite[Theorem
    3.2.4.2]{MatrixAnalysis}. The principal component of $\CommMat$ is
    therefore a closure of a set that is parametrized by the set of
    all matrices with $d$ distinct eigenvalues, which is open in $\Matrices$,
    and $n-1$ copies of the space of uni-variate polynomials of degree at most
    $d-1$, so the principal component has dimension $d^2 + (n-1)d$
    \cite[Proposition 6]{GuralnickSethuraman}.

        A tuple of matrices is diagonalizable if and only if it corresponds to
        an $S$-module abstractly isomorphic to $\bigoplus_{i=1}^d
        S/\mm_i$ for some maximal ideals $\mm_i\subset S$ with $S/\mm_i =
        \kk$. The
        \emph{principal component} of $\Quotmain$ is the closure of the locus
        of quotients $F/K$ that are abstractly isomorphic to such modules.
        In terms of Table~\ref{table:objectsinvolved}, the principal component of
        $\Quotmain$ is the image of
        the preimage in $\prodfamilystable$ of the principal component of
        $\CommMat$. It follows that the dimension of the principal component
        of $\Quotmain$ is $(n+r-1)d$.

        Antipodal to the principal components are the elementary components. A
        component of $\CommMat$ is \emph{elementary} if its points correspond
        to tuples where each matrix has precisely one eigenvalue. A component
        of $\Quotmain$ is \emph{elementary} if it parameterizes modules $M$
        supported at a single point. By
        Lemma~\ref{ref:supportVsEigenvalues:lem} those two notions agree in
        the sense that the injection above restricts to a bijection between
        elementary components of $\Quotmain$ and elementary components of $\CommMat$ that
        intersect the image of $\prodfamilystable$.

        \subsection{Transposes and duality}\label{ssec:transposes}

            The commuting matrices have a natural involution: the
            transposition of each matrix in the tuple. The corresponding
            notion for modules is taking the dual module. We
            now review how the basic invariants behave under this
            operation.

            Recall that an $S$-module has finite degree if it is finite
            dimensional as a $\kk$-vector space. Throughout this section we
            fix a finite degree $S$-module $M$.
            The vector space $M^{\vee} = \Hom_{\kk}(M, \kk)$ has a natural
            structure of an $S$-module by precomposition: for $s\in S$ and
            $\varphi\in\Hom_{\kk}(M, \kk)$ we define $(s\cdot \varphi)(m) =
            \varphi(sm)$ for every $m\in M$.
            This $S$-action is called \emph{contraction}.
            \begin{definition}
                The \emph{dual module} of $M$ is the space $M^{\vee}$ equipped
                with an  $S$-action by contraction. We denote it by
                $M^{\vee}$ if no confusion is likely to occur.
            \end{definition}
            The double dual map $M\to (M^{\vee})^{\vee}$ is an isomorphism of
            $S$-modules and the annihilators of $M$ and $M^{\vee}$ in $S$ are equal.
            For an ideal $I \subset S$, let $(0:I)_M = \left\{
                m\in M\ |\ I m =\{0\} \right\}$ be the annihilator of $I$ in $M$.
                \begin{lemma}\label{ref:dualityPartTwo}
                    For every ideal $I$ the vector spaces $M/IM$ and
                    $(0:I)_{M^{\vee}}$ have the same dimension.  Also the
                    vector spaces $(0:I)_M$ and $M^{\vee}/IM^{\vee}$ have the
                    same dimension.
                \end{lemma}
                \begin{proof}
                    The tautological perfect pairing $M\times M^{\vee}\to \kk$
                    descends to a perfect pairing between $M/IM$ and the
                    perpendicular to $IM$ inside $M^{\vee}$. A functional $\varphi\in
                    M^{\vee}$ is perpendicular to $IM$ iff it vanishes on $IM$
                    iff $I\varphi = \{0\}$. Hence the subspace $(IM)^{\perp}
                    \subset M^{\vee}$ is equal to $(0:I)_{M^{\vee}}$ and the
                    perfect pairing shows that it has dimension equal to
                    $M/IM$. The second claim is proven by interchanging $M$
                    and $M^{\vee}$ using that the double dual is an
                    isomorphism.
                \end{proof}

                Let $d = \dim_{\kk} M$. Let $\mm =
                (y_1,..,y_n) \subset S$.
                Fix a basis on $M$ and the dual basis on $M^{\vee}$.
                Assume that $M$ is supported only at zero. In particular, this
                implies that the annihilator of $M$ contains some power of
                $\mm$, say $\mm^D$, so that $M$ is a module over the local
                ring $S/\mm^D$.
                Let
                $\xx = (x_1, \ldots ,x_n)\in \CommMat$ be the commuting tuple
                associated to $M$, then $\xx^T = (x_1^T, \ldots ,x_n^T)$ is the commuting
                tuple associated to $M^{\vee}$.
                Thus $M$ lies in the principal component if and
                only if $M^{\vee}$ lies in the principal component.
                Similarly,
                $M$ lies in a non-elementary component if and only if
                $M^{\vee}$ lies in a non-elementary component.
                For later use we summarize the
                following data.
                \begin{center}
                    \small
                    \begin{tabular}[<+position+>]{c c c c c}
                        min. no of gen. of $M$ &$=$&
                        $\dim_{\kk}(M/\mm M) = \dim_{\kk}(0:\mm)_{M^{\vee}}$ &
                         $=$ &dim. of common kernel of $\xx^T$\\
                         min. no of gen. of $M^{\vee}$ &$=$&
                         $\dim_{\kk}(M^{\vee}/\mm M^{\vee}) = \dim_{\kk}(0:\mm)_{M}$ &
                         $=$ &dim. of common kernel of $\xx$\\
                         &&
                         $\dim_{\kk}(M^{\vee}/\mm^2 M^{\vee}) =
                         \dim_{\kk}(0:\mm^2)_{M}$ &
                         $=$ &$\dim_{\kk} \bigcap_{i,j=1}^n \ker(x_ix_j)$\\
                         &&
                         $d - \dim_{\kk}(0:\mm)_{M^{\vee}} =
                         \dim_{\kk}\mm M$ &
                         $=$ &$\dim_{\kk} \sum_{i=1}^n \im x_i$\\
                         $d - \dim_{\kk}\big((0:\mm)_{M^{\vee}} \cap
                         \mm M^{\vee}\big)$ &$=$&
                         $\dim_{\kk}(\mm M + (0:\mm)_M)$ &
                         $=$ &$\dim_{\kk} \left(\sum_{i=1}^n \im x_i +
                         \bigcap_{i=1}^n \ker(x_i)\right)$\\
                    \end{tabular}
                \end{center}

        \subsection{Duality and minimal graded
        resolutions}\label{ssec:transposes_and_resolution}

            The duality above is tightly connected with resolutions. Let us
            recall the graded case. Let $S = \kk[y_1, \ldots ,y_n]$ be a
            polynomial ring in $n$ variables graded by $\deg(y_i) = 1$.
            For a graded $S$-module $N$ and $j\in \mathbb{Z}$ we denote by
            $N(j)$ the module $N$ with grading shifted down by $j$.
            Explicitly, this means that $N(j)_i = N_{i+j}$ for every $i\in
            \mathbb{Z}$. For example, $S(j)$ is a free module generated by an
            element of degree $-j$.
            Let $M$
            be a finite degree graded $S$-module. By Hilbert's Syzygy
            Theorem and Auslander-Buchsbaum formula~\cite[Chapter~19]{Eisenbud} the module $M$ has a unique up
            to isomorphism minimal graded free resolution which has length
            $n$:
            \[
                \begin{tikzcd}
                    0 & \arrow[l] M & \arrow[l] F_0 & \arrow[l] F_1 &\arrow[l]
                    \ldots & \arrow[l] F_{n} & \arrow[l] 0
                \end{tikzcd}
            \]
            Each $F_i$ is a finitely generated free $S$-module
            which is graded. We write $F_i = \bigoplus_{j\in \mathbb{Z}} S^{\beta_{ij}}(-j)$
            which explicitly means that $F_i$ has $\beta_{ij}$ generators of degree $j$.
            A subset of elements of $F_i$ is \emph{minimal} if its image in
            $F_i/(y_1, \ldots ,y_n)F_i$ is linearly independent. In
            particular, for $F_0$ we speak of \emph{minimal generators} of $M$, for
            $F_1$ about \emph{minimal relations} and for
            $F_2$ about \emph{minimal syzygies} of $M$.
            \begin{example}
                The module $F = Se_1 \oplus Se_2\oplus Se_3$
                graded by setting $\deg(e_1) = 1$, $\deg(e_2) = \deg(e_3) = 3$ is
                denoted by $F = S(-1) \oplus S^{\oplus 2}(-3)$.
            \end{example}
            \begin{example}\label{ex:resolution}
                For $n = 2$ the module $M = S/(y_1^2, y_1y_2, y_2^2)$ has a minimal
                graded resolution that has length exactly two:
                \[
                    \let\amsamp=&
                    \begin{tikzcd}
                        0 & \arrow[l] M &\arrow[l] S^{\oplus 1} &&&\arrow[lll,
                            "{\begin{bmatrix}y_1^2 \amsamp y_1y_2 \amsamp
                            y_2^2\end{bmatrix}}"']
                        S^{\oplus 3}(-2) &&&\arrow[lll, "{\begin{bmatrix}
                                y_2 \amsamp 0\\
                                -y_1 \amsamp y_2\\
                                0 \amsamp -y_1
                            \end{bmatrix}}"'] S^{\oplus 2}(-3) & \arrow[l] 0.
                    \end{tikzcd}
                \]
            \end{example}
            A beautiful feature is that the resolution of the dual module $M^{\vee}$ is dual to the
            resolution of $M$. This is known to experts, but we were
            unable to locate this statement explicitly, so we provide a proof. The key part to prove this is the following
            result.
            \begin{lemma}[{\cite[Ex~4.4.20]{BrunsHerzog}}]\label{ref:shift:lem}
                Let $R$ be a graded ring and $0\neq y\in R$ a nonzerodivisor
                which is homogeneous of degree $r$.
                Let $M$ and $N$ be graded $R$-modules such that the multiplication by $y$
                is injective on $N$ and $yM = 0$. Then $\Hom_R(M, N) = 0$ and
                for every $i\geq 0$ we have an isomorphism of graded $R$-modules
                \[
                    \Ext^{i+1}_R(M, N)(-r)  \simeq \Ext^{i}_{R/(y)}(M,
                    N/yN).\qedhere
                \]
            \end{lemma}
            \begin{proof}[Sketch of proof]
                For every $\varphi\in \Hom_R(M, N)$ we have $0 = \varphi(yM) =
                y\varphi(M)$ so $\varphi(M) = 0$ and $\varphi = 0$.
                For the $\Ext$-part, repeat the proof
                of~\cite[Lemma~3.1.16]{BrunsHerzog} in the graded setting.
            \end{proof}
            \begin{theorem}[duality]\label{ref:duality_for_resolutions:thm}
                Let $M$ be a graded $S$-module of finite degree and
                $M^{\vee}$ be a dual module. Let $F_{\bullet}$ be minimal
                graded free resolution of $M$. Then $F_{\bullet}$ has length
                exactly $n=\dim(S)$. The complex $\Hom(F_{\bullet}, S)$ is
                a minimal graded free resolution of $M^{\vee}(n)$.
            \end{theorem}
            This is well known to experts, but we provide a proof that does not employ the full machinery of canonical
            modules.
            \begin{proof}
                Since $M$ is graded $S$-module of finite degree, there exists $N$ large enough such that $y_i^{N}M=0$ for
                $i=1, \ldots ,n$.
                By Lemma~\ref{ref:shift:lem} for every $0\leq i \leq n$
                we have $\Ext^{i}_{S}(M, S)(-iN) =
                \Hom_{S/(y_1^N, \ldots ,y_i^{N})}(M, S/(y_1^N, \ldots
                ,y_i^{N}))$. For $i <n$ the same Lemma applied to
                $y_{i+1}^N$ gives that the right hand side is zero. The above $\Ext$ groups
                are homology of the complex $\Hom(F_{\bullet}, S)$, so this
                complex is exact except for the $n$-th position: it is a
                resolution of $\Ext^{n}_{S}(M, S)  \simeq \Hom_{S/(y_1^N,
                \ldots ,y_n^{N})}(M, S/(y_1^N, \ldots ,y_n^{N}))(nN)$. Let
                $R=S/(y_1^N, \ldots ,y_n^{N})$, let $d = n(N-1)$ and let
                $\pi\colon R\to
                \kk(-d)$ be a
                $\kk$-linear projection onto the top degree part: it maps
                the monomial $\prod_{i=1}^n y_i^{N-1}$ to $1$ and all other
                monomials to zero.
                For any finitely generated graded $R$-module $P$ we obtain a graded
                $\kk$-linear map $\Phi_P\colon \Hom_R(P, R)\to \Hom_{\kk}(P,
                \kk(-d)) = P^{\vee}(-d)$ defined by $\Phi_P(\varphi) = \pi\circ \varphi$.
                The map $\Phi_P$ is a homomorphism of $R$-modules, because for
                every $\varphi\colon P\to R$ and $r\in R$, $p\in P$ we have
                \[
                    (\Phi_P(r\varphi))(p) = \pi(r\varphi(p)) =
                    \pi(\varphi(rp)) = (\pi\circ
                    \varphi)(rp) = (r(\pi \circ \varphi))(p) =
                    (r\Phi_P(\varphi))(p).
                \]
                We claim that $\Phi_P$ is an isomorphism for every
                $P$ as above. Taking the claim for granted, we take $P = M$
                and obtain an isomorphism of graded $R$-modules
                \[
                    \Hom_R(M, R)(nN)\to M^{\vee}(-d+nN) = M^{\vee}(n).
                \]
                Together with the isomorphism $\Ext^n_S(M, S) \simeq
                \Hom_R(M, R)(nN)$ this concludes the proof of theorem; it
                remains to prove the claim. Fix any $P$ and its graded
                presentation $G_2\to G_1 \to P\to 0$ where $G_i$ are finitely
                generated graded free
                $R$-modules. We obtain a commutative diagram of graded $R$-modules
                \[
                    \begin{tikzcd}
                        0 \ar[r] & \Hom_R(P, R)\ar[r]\ar[d, "\Phi_P"] &
                        \Hom_R(G_1, R) \ar[r]\ar[d, "\Phi_{G_1}"]
                        & \Hom_R(G_2, R)\ar[d, "\Phi_{G_2}"]\\
                        0 \ar[r] & P^{\vee}(-d)\ar[r] & G_1^{\vee}(-d) \ar[r]
                        & G_2^{\vee}(-d),
                    \end{tikzcd}
                \]
		    where $\Phi_P$ is isomorphism whenever $\Phi_{G_1}$ and $\Phi_{G_2}$ are.
                It is thus enough to prove that $\Phi_P$ is an isomorphism for
                $P$ free. By splitting direct sums, we reduce to the case $P =
                R(-e)$. By shifting  degrees we reduce to the case $P = R$.
                This case asks whether $\Phi_R\colon\Hom_R(R,R)\to
                R^{\vee}(-d)$ is an
                isomorphism. Both sides are vector spaces of the same
                dimension, so it is enough to prove that $\Phi_R$ is
                injective. Take a nonzero homomorphism $\varphi$. Then
                $\im(\varphi)$ is a nonzero homogeneous ideal of $R$. We see
                directly that every such ideal contains the top degree part $R_d$ of $R$, so
                $\varphi$ maps some element of $R$ into $R_d$ hence its
                composition with $\pi$ is nonzero. This concludes the proof
                of the claim.
            \end{proof}

        \begin{example}
            In the setting of Example~\ref{ex:resolution} the dual of the
            resolution of $M$ is
                \[
                    \let\amsamp=&
                    \begin{tikzcd}
                        0 & \arrow[l] M^{\vee}(2) &\arrow[l] S^{\oplus 2}(3) &&&
                        \arrow[lll, "{\begin{bmatrix}y_2 \amsamp -y_1 \amsamp
                        0\\ 0 \amsamp y_2 \amsamp -y_1\end{bmatrix}}"'] S^{\oplus 3}(2) &&& \arrow[lll,
                    "{\begin{bmatrix} y_1^2 \\ y_1y_2 \\ y_2^2 \end{bmatrix}}"'] S^{\oplus 1} & \arrow[l] 0.
                    \end{tikzcd}
                \]
                The surjection $S^{\oplus 2}(3)\onto M^{\vee}(2)$ says
                that $M^{\vee} = \Hom_{\kk}(\kk[y_1, y_2]/(y_1,
                y_2)^2, \kk)$ is generated by two elements of degree minus
                one: the projections onto $\kk y_i$ for $i=1,2$.
        \end{example}

        \section{Structural results on $\CommMat$ and
        $\Quotmain$}\label{sec:strutural}

            In this section we present general tools for investigation of
            $\CommMat$. These will be heavily employed in the proofs of our
            main results, but are of general interest. In contact with
            Section~\ref{sec:prelims} the results presented below are new (to
            our best knowledge).

        \subsection{Apolarity for modules}\label{ssec:apolarity}
        Writing down equations for a finite degree module $F/K$ can be cumbersome,
        since $K$ typically has many generators. This can be resolved by
        introducing apolarity for modules. Apolarity for algebras
        is already an extremely important tool in the
        classification of
        algebras~\cite{iarrobino_kanev_book_Gorenstein_algebras,
        Jel_classifying}, apolarity for modules generalizes it in a natural
        way.
        Recall that an $S$-module has finite degree if it is finite
        dimensional as a $\kk$-vector space. We say that an $S$-submodule
        $K\subset F$ has \emph{cofinite degree} (shortly: \emph{is cofinite}) if the module $F/K$ has finite
        degree.

    Let $F$ be a free $S$-module and $\Hom_{\kk}(F, \kk)$ be the
    space of functionals. To a cofinite submodule $K\subset F$ we associate a
    subspace $K^{\perp} \subset \Hom_{\kk}(F, \kk)$ defined by
    \[
        K^{\perp} := \left\{ \varphi\colon F\to \kk\ |\ \varphi(K) = \{0\}
    \right\}.
    \]
    This subspace is isomorphic to $\Hom_{\kk}(F/K, \kk)$ and thus finite
    dimensional.
    The aim of this subsection is to explore the relation between $K \subset F$ and
    $K^{\perp}$.
    Recall from Section~\ref{ssec:transposes}
    that $\Hom_{\kk}(F, \kk)$ is an $S$-module via contraction action: for $s\in S$ and
    $\varphi\in\Hom_{\kk}(F, \kk)$ we have $(s\cdot \varphi)(f) =
    \varphi(sf)$ for every $f\in F$.
    The following observation is fundamental:
    \begin{lemma}\label{ref:perpOfModuleIsModule:lem}
        For every $K \subset F$ the subspace $K^{\perp}$ is an $S$-submodule.
    \end{lemma}
    \begin{proof}
        Take $\varphi\in K^{\perp}$, so that $\varphi(K) = \{0\}$. For every $s\in
        S$ we have
        \[
            (s\cdot \varphi)(K) = \varphi(s\cdot K) \subset \varphi(K)
        = \{0\}.\qedhere
    \]
    \end{proof}
    By Lemma~\ref{ref:perpOfModuleIsModule:lem}, the map $K\mapsto K^{\perp}$
    sends cofinite $S$-submodules of $F$ to finite degree
    $S$-submodules of $\Hom_{\kk}(F, \kk)$.  Now we construct a map in the
    opposite direction. To a finite degree $S$-submodule $M \subset
    \Hom_{\kk}(F, \kk)$, we
    associate an $S$-submodule
    \[
        M^{\perp} := \left\{ f\in F\ |\ \forall{\varphi\in M} : \varphi(f) = 0
    \right\} \subset F.
    \]
    Applying $\Hom_{\kk}(-, \kk)$ to the inclusion of $M$ into  $\Hom_{\kk}(F, \kk)$ we get
    a surjection
    \[
        \Hom_{\kk}(\Hom_{\kk}(F, \kk),\kk) \onto \Hom_{\kk}(M, \kk).
    \]
    Since
    $M$ has finite degree, the composed map $F\to \Hom_{\kk}(M, \kk)$ is also
    surjective. The module $M^{\perp}$ is by definition its kernel, so we get
    a bijective linear map $F/M^{\perp}\to \Hom_{\kk}(M, \kk)$. In particular, the dimensions of
    vector spaces $F/M^{\perp}$ and $M$ are equal. But even more is true.
    \begin{lemma}\label{ref:duality}
        The map $F/M^{\perp}\to \Hom_{\kk}(M, \kk)$ is an isomorphism of
        $S$-modules, where $S$ acts on $\Hom_{\kk}(M, \kk)$ by contraction.
    \end{lemma}
    \begin{proof}
        We already know that the map is a bijection, it remains to check that
        it is $S$-linear. Unraveling definitions we see that the image of
        $f\in F$ in $\Hom_{\kk}(M, \kk)$ is an element $\mu_{f}$ such that
        $\mu_f(m) = m(f)$ for all $m\in M \subset \Hom_{\kk}(F, \kk)$.
        Therefore
        \[
            \mu_{sf}(m) = m(sf) = (s\cdot m)(f) = \mu_{f}(s \cdot m) =
            (s\cdot \mu_f)(m),
        \]
        so the map $f\mapsto \mu_f$ is $S$-linear as claimed.
    \end{proof}
    Considering both maps $K\mapsto K^{\perp}$ and $M\mapsto M^{\perp}$ we
    obtain the following bijection which is an embedded form of the double
    dual map.
    \begin{proposition}[Apolarity for
        modules]\label{ref:MacaulayGeneral:prop}
        The maps $K\mapsto K^{\perp}, M\mapsto M^{\perp}$ give a
        bijection between cofinite submodules of $F$ and finite
        degree submodules of $\Hom_{\kk}(F, \kk)$.
    \end{proposition}
    \begin{proof}
        It remains to check that the natural maps $M\to (M^{\perp})^{\perp}$
        and $K\to (K^{\perp})^{\perp}$ are identities and we leave it to the
        reader.
    \end{proof}

    The above may rightfully look like too abstract to apply, since the space
    $\Hom_{\kk}(F, \kk)$ is huge, of uncountable dimension over $\kk$. We
    shrink it a bit now and give a down-to-earth presentation.
    Let $F_i\subset F$ be the linear subspace of elements of degree $i$ where
    $F$ is generated by elements of degree zero.
    Define
    \[
        \Fdual  := \bigoplus_i\Hom_{\kk}(F_i, \kk)\subset \Hom_{\kk}(F, \kk).
    \]
    The space $\Fdual$ can be thought of as a restricted dual of $F$. Note
    that in the current
    article we use $(-)^{\vee}$ to denote the dual space. Alternatively, we
    can view $\Fdual$ as the space of all those functionals on $F$ that vanish
    on $(y_1, \ldots ,y_n)^DF$ for some $D$.
    Fix an $S$-basis $e_1, \ldots e_{r}$ of $F$, so that $F =
    \bigoplus_{j=1}^r\kk[y_1, \ldots ,y_n]e_j$. Each module $F_l$ has a ``monomial''
    basis consisting of elements $y_1^{a_1}\cdot  \ldots y_n^{a_n} e_j$ with
    $\sum a_i = l$ and
    the dual basis on $\Fdual$ identifies it with the linear space
    \[
        \bigoplus_{j=1}^r \kk[z_1, \ldots ,z_n]e_j^*.
    \]
    The $S$-module structure coming from contraction becomes ``coefficientless
    derivation'': the element $y_i$ acts of $\Fdual$ by
    \[
        y_i\circ (z_1^{a_1} \ldots z_{n}^{a_n})e_j^* = \begin{cases}
            0 & \mbox{ if } a_i = 0\\
            (z_1^{a_1} \ldots z_{i-1}^{a_{i-1}}z_i^{a_i -
            1}z_{i+1}^{a_{i+1}} \ldots  z_{n}^{a_n})e_j^* & \mbox{otherwise}.
        \end{cases}
    \]
    For example, we have $y_2\circ (z_1z_2 e_1^* + z_3z_1 e_2^* + z_2^2 e_3^*) =
    z_1e_1^* + z_2e_3^*$.
    The grading on $F$ naturally induces a grading on $\Fdual$, where $\deg(z_i)= -1$.
    By slight abuse of notation we forget about the minus and consider
    $\Fdual$ as positively graded, i.e., with $\deg(z_i) = 1$. For example,
    $\deg(z_1^2e_1 + z_2z_3e_2) = 2$.
    In this setup we obtain a restricted version of inverse systems, more resembling the
    classical one.
    \begin{proposition}[Apolarity for modules, local]
        The maps $K\mapsto K^{\perp}, M\mapsto M^{\perp}$ give a
        bijection between cofinite $S$-submodules $K \subset F$ such that
        $F/K$ is supported only at
        the origin and finite degree $S$-submodules of $\Fdual$.
    \end{proposition}
    \begin{proof}
        We will prove that the maps from Proposition~\ref{ref:MacaulayGeneral:prop}
        restrict to the given classes of modules.
        For every element $f\in \Fdual$ there is a $D$ such that $(y_1, \ldots
        ,y_n)^D$ annihilates $f$. Therefore, the same holds for finitely many
        elements of $\Fdual$ and this shows that for every finitely generated submodule $M \subset \Fdual$,
        the module $F/M^{\perp}$ is
        annihilated by $(y_1, \ldots
        ,y_n)^D$ so it is supported only at the origin.
        Conversely, take a cofinite module $K\subset F$ such that $F/K$ is supported only at
        the origin. The support of $F/K$ is equal to the set of prime ideals
        containing $\Ann(F/K)$. Since there is only one such, the radical of
        $\Ann(F/K)$ is equal to $(y_1, \ldots ,y_n)$. Since $S$ is Noetherian,
        some power of $(y_1, \ldots ,y_n)$ is contained in $\Ann(F/K)$, so
        we have $(y_1, \ldots ,y_n)^D F \subset K$ for large
        enough $D$ so that $K^{\perp}$ is contained in $\Fdual_{\leq D-1}$.
    \end{proof}
    \begin{definition}
        For elements $\sigma_1, \ldots ,\sigma_r\in \Fdual$
        generating an $S$-submodule $M \subset \Fdual$ the
        \emph{apolar module of $\sigma_1, \ldots ,\sigma_r$} is $F/M^{\perp}$.
        Conversely, for an $S$-module $F/K$ supported only at the
        origin any set of generators of
        the $S$-module $K^{\perp}\subset
        \Fdual$ is called a set of \emph{dual generators} of $F/K$.
    \end{definition}

    We now give some examples for later use. We will freely use the linear algebra from
    \S\ref{ssec:transposes} coupled with Lemma~\ref{ref:duality} to compute some invariants.
    \begin{example}\label{example:cyclic}
            Let $S$ be any polynomial ring, $F = Se_1$ be a rank one free
            module and $\Fdual = Se_1^*$. Take any elements $\sigma_1, \ldots
            ,\sigma_r\in \Fdual$ and the module $M\subset \Fdual$ generated by
            them. Then $F/M^{\perp}$ has the form $Se_1/Ie_1$ for an
            ideal $I \subset S$ and in fact $I = \Ann(\sigma_1, \ldots ,\sigma_r)$ is
            the usual apolar ideal of $\sigma_1, \ldots ,\sigma_r$ as
            in~\cite{iarrobino_kanev_book_Gorenstein_algebras} for the
            contraction action. This shows
            how the apolarity for modules generalizes the one for algebras.
    \end{example}

    \begin{example}\label{example:moduleForAB}
        Fix any $a$, $b$ such that $a\leq
        \binom{b+1}{2}$ and any $n\geq b$.
        Let $S=\kk[y_1, \ldots ,y_n]$ with $\mm = (y_1, \ldots ,y_n)$ and let
        $F = \bigoplus_i Se_i$ be a free module of rank at least one.
        Let $Q_1 = (\sum_{i=1}^b z_i^2)e_1^*$ and $Q_2, \ldots ,Q_a\in
        \kk[z_1, \ldots ,z_b]
        e_1^*
        \subset \Fdual$ be any quadrics such that $Q_1, \ldots ,Q_a$ are linearly
        independent.
        The $S$-submodule $M \subset \Fdual$
        generated by quadrics $Q_{1}, \ldots ,Q_a$ is an $(a+b+1)$-dimensional vector space with basis
        $\{Q_{i}\}_{i=1, \ldots ,a}\cup \{z_je_1^*\}_{1\leq j\leq b}\cup
        \{e_1^*\}$ and $(0:\mm)_M$ is one-dimensional.
        The module $F/M^{\perp}$ is generated by $e_1$
        and isomorphic to $S^{\oplus 1}/I$, where $S/I$ is a graded algebra with Hilbert function $(1,
        b, a)$.
    \end{example}
    \begin{example}
        To give a very concrete example  of a non-cyclic
        module, let $S = \kk[y_1, \ldots ,y_5]$ with
        $\mm = (y_1, \ldots ,y_5)$, let
        $F= Se_1 \oplus Se_2$ and $Q = (z_1^2 + z_2^2 + z_3^2)e_1^*
        + z_4e_2^*\in \Fdual$. The $S$-submodule $M \subset \Fdual$ generated by $Q$
        is a $6$-dimensional vector space spanned by $Q$, $y_1\circ Q =
        z_1e_1^*$, $y_2\circ Q = z_2e_1^*$, $y_3\circ Q = z_3e_1^*$, $y_4\circ
        Q = e_2^*$ and $y_1^2\circ Q = y_2^2\circ Q=y_3^2\circ Q= e_1^*$.
        We see that $(y_1, \ldots ,y_5)^3$ annihilates $M$, while $(y_1,
        \ldots ,y_5)^2$ does not.
        The module $M^{\perp} \subset F$ is
        minimally generated by
        \[
            (y_je_1)_{j=4,5},\ (y_ie_2)_{i=1,2,3,5},\ (y_iy_je_1)_{1\leq
            i<j\leq 3},\ ((y_i^2 - y_j^2)e_1)_{1\leq i < j\leq 3},\ y_1^2e_1 -
            y_4e_2.
        \]
        Using the table from Section~\ref{ssec:transposes}, we compute that
        $(F/M^{\perp}) / \mm (F/M^{\perp})$ has dimension $\dim_{\kk}
        (0 : \mm)_M = \dim_{\kk} \spann{e_1^*, e_2^*} = 2$, so
        the module $F/M^{\perp}$ is not cyclic.
    \end{example}

    The following two examples will be useful in Section~\ref{sec:degeight}.
    The exact assumptions are tailored to the needs of that section, so are
    by no means ``natural''.
    \begin{example}\label{example:independentQuadrics}
        Fix any $c$ and $a\leq \binom{c+1}{2}$ and additionally fix any $b\geq
        c$. Let $S$ be a polynomial ring of dimension $n\geq
        b$ and let $\mm = (y_1, \ldots ,y_n)$.
        Let $F$ be a free $S$-module of rank $r = \max(a, b-c)$ with basis $e_1,
        e_2, \ldots e_r$.
        Pick $a$ linearly independent quadrics $Q_1,
        \ldots, Q_a$ in $\kk[z_1, \ldots ,z_c]$ with $Q_1 = \sum_{i=1}^c z_i^2$ and consider the element
        $\mathcal{Q} := \sum_{i=1}^a Q_ie_i^*\in \Fdual$. Finally consider the submodule $M$ in
        $\Fdual$ generated by $\mathcal{Q}$ and $\{z_je_1^*\}_{j=c+1}^{b}$.
        Since the quadrics $Q_i$ are linearly independent, we have $(y_1,
        \ldots ,y_c)^2 M = \spann{e_1^*, \ldots ,e_a^*} =
        (0:\mm)_M$.
        For further reference, we note that $\mm M/\mm^2 M$ has dimension $c$
        with a basis
        $\{y_i \circ \mathcal{Q}\}_{1\leq i\leq c}$ while
        $(0:\mm^2)_M / \mm^2 M$ has dimension $b$ with a basis consisting of the $c$ elements
        above and $\{z_je_1^*\}_{j=c+1}^{b}$. Finally, $M$ itself is
        $(a+b+1)$-dimensional.
        Note that in this example we consider $M$ and not $F/M^{\perp}$.
    \end{example}

    \begin{example}\label{example:twoQuadricsThreeLinears}
        Let $S = \kk[y_1, \ldots ,y_n]$ be a polynomial ring, $F$ be a free $S$-module and $Q_1,
        Q_2\in \Fdual$ be two homogeneous quadrics such that the submodule $M$
        in $\Fdual$ generated by them has $\dim_{\kk} M_1 \leq
        3$, where $S_1$, $M_1$ are the linear parts of
        $S$ and $M$ respectively. We claim that
        up to coordinate change we have $y_iM = 0$ for $i\geq 4$.
        Fix $L(Q_i) = \{\ell\in S_1\ |\ \ell Q_i = 0\} \subset S_1$. The space $L(Q_i)$ is the kernel of a
        map $S_1 \to M_1$, so $\dim L(Q_i) \geq n-3$. Assume
        $\dim L(Q_1) \leq \dim L(Q_2)$ and consider two cases:
        \begin{enumerate}
            \item $\dim L(Q_1) = n - 3$. Up to coordinate change, we may
                assume $y_i Q_1 = 0$ for $i\geq 4$ and $y_1Q_1$, $y_2Q_1$,
                $y_3Q_1$ span $M_1$. In particular, the space $M_1$ is
                annihilated by $y_i$ for every $i\geq 4$. Pick any $i\geq 4$. For every $j$ the element
                $y_jQ_2$ lies in $M_1$, so is annihilated by $y_i$. In other
                words, we have $y_j(y_iQ_2) = y_i(y_jQ_2) = 0$
                for every $j$. But this means that every linear form
                annihilates $y_iQ_2$ which can happen only if $y_iQ_2 = 0$. This
                proves that $L(Q_2)$ contains $L(Q_1)$ and concludes this case.
            \item $\dim L(Q_1) \geq n - 2$, so that $\dim L(Q_2) \geq n-2$.
                Then we have $\dim (L(Q_1) \cap L(Q_2)) \geq n-4$. If strict
                equality happens, we are done. If not, up to change of basis
                we may assume $L(Q_1) = \spann{y_3, y_4, \ldots }$ and $L(Q_2)
                = \spann{y_1, y_2, y_5, y_6,  \ldots }$.
                The space $M_1$ is at most three-dimensional and contains $y_1Q_1,
                y_2Q_1, y_3Q_2, y_4Q_2$, so there is a linear dependence: some
                nonzero linear form $\ell$ lies both in the submodule generated by $Q_1$
                and by $Q_2$. But then the annihilator of $\ell$ contains
                $L(Q_1) + L(Q_2) = S_1$, so actually $\ell = 0$, a
                contradiction.
        \end{enumerate}
    \end{example}

    \subsection{Obstruction theories}\label{ssec:obstructions}

        \newcommand{\nn}{\mathfrak{n}}%
        \newcommand{\mms}{\mathfrak{m}_{\hat{S}}}%
        \newcommand{\hatS}{\hat{S}}%
        In the following we will use obstruction theories, so we give a brief
        outline of them (see \cite{Fantechi_Manetti, HarDeform,
        Sernesi__Deformations} for more complete account). Roughly speaking, obstructions are a useful black box
        for proving that a given morphism $X\to Y$ is smooth or \'etale at a
        given $\kk$-point $x$.

        Let $(A, \mm)$ be a noetherian complete local $\kk$-algebra with residue
        field $\kk$, such as
        $\hat{\mathcal{O}}_{X, x}$. By Cohen's structure theorem we can write
        $A = \hatS/I$ where $\hatS = \kk[[y_1, \ldots ,y_d]]$ is a power
        series ring with maximal ideal $\mms = (y_1, \ldots , y_d)$ and $I
        \subset \mms^2$. The prototypical obstruction
        space for $A$ is $Ob := (I/\mms I)^{\vee}$.
        We now construct the prototypical
        obstruction map for $Ob$.
        Suppose we have an Artin
        local $\kk$-algebra $(B,\nn)$ with residue field $\kk$ and a
        surjection $B\onto B_0 = B/J$ such that $J\cdot \nn = 0$.
        \begin{example}
            We could take $B = \kk[\varepsilon]/\varepsilon^3$ and $B_0 =
            \kk[\varepsilon]/\varepsilon^2$ or more generally $B =
            \kk[\varepsilon]/\varepsilon^{r+1}$ and $B_0 =
            \kk[\varepsilon]/\varepsilon^{r}$ for some $r\geq 2$.
        \end{example}
        Suppose further that we have a $\kk$-algebra
        homomorphism $\varphi_0\colon A\to B_0$. We want to answer the
        question:
        \begin{center}
            When does $\varphi_0$ lift to
            $\varphi\colon A\to B$?
        \end{center}
        We can always lift (not uniquely) the
        morphism $\kk[[y_1, \ldots ,y_d]]\to B_0$ to $B$: we first lift it to
        a homomorphism $\kk[y_1,\ldots ,y_d]\to B$ and this homomorphism then extends to $\kk[[y_1,\ldots ,y_d]] \to B$, because some power of the maximal ideal $\mathfrak{n}$ of the Artin local $\kk$-algebra $B$ is zero. So we obtain a commutative
        diagram
        \[
            \begin{tikzcd}
                0\arrow[r] & I\arrow[r]\arrow[d, "\psi|_I"] & \hatS\arrow[r]\arrow[d, "\psi"] &
                A\arrow[r]\arrow[d, "\varphi_0"] & 0\\
                0 \arrow[r] & J\arrow[r] & B\arrow[r] & B_0 \arrow[r] & 0.
            \end{tikzcd}
        \]
        Clearly, if $\psi|_I = 0$ then we obtain a map $A\to B$.
        Moreover, $\psi|_I(\mms I) \subset \nn\cdot J = 0$, so actually
        $\psi|_I = 0$ if and only the induced map $I/\mms I\to J$ is zero.
        Hence, $\psi|_I = 0$ if and only if the induced $\kk$-linear map $ob_{B\to B_0,
        \varphi_0}\colon
        J^{\vee} \to Ob$ is zero. The map $ob_{B\to B_0, \varphi_0}$ is called the obstruction map.
        It does not depend on the choice of lifting $\psi$, thanks to the fact
        that $J\cdot \nn = 0$ and $I \subset \mms^2$.
        An obstruction theory is an abstracted version of
        $Ob$.
        \begin{definition}\label{ref:obstructiontheory:def}
            An \emph{obstruction theory} for a complete local $\kk$-algebra $(A,
            \mm)$ with residue field $\kk$ is a $\kk$-vector space $O$ and for
            every exact sequence $0\to J\to B\to B_0\to 0$ as above and a
            $\kk$-algebra homomorphism
            $\varphi_0\colon A\to B_0$, a $\kk$-linear map
            $ob_{B\to B_0, \varphi_0}\colon J^{\vee} \to O$ such that
            $ob_{B\to B_0, \varphi_0}$ is zero if and only if $\varphi_0$ lifts
            to $\varphi\colon A\to B$.
            Moreover, the maps $ob$ are required to be appropriately
            functorial: if $0\to J'\to B'\to B_0'\to 0$ is
            another sequence as above and $\rho, \rho_0$ are $\kk$-algebra homomorphisms
            which fit into a commutative diagram
            \[
                \begin{tikzcd}
                    0\ar[r] & J\ar[d, "\rho|_J"]\ar[r] & B \ar[d, "\rho"]\ar[r] & B_0
                    \ar[d, "\rho_0"]\ar[r] & 0\\
                    0\ar[r] & J'\ar[r] & B' \ar[r] & B_0' \ar[r] & 0
                \end{tikzcd}
            \]
            then $ob_{B'\to B_0', \rho_0\circ\varphi_0} = (ob_{B\to B_0,
            \varphi_0})\circ (\rho|_J)^{\vee}$ as maps $(J')^{\vee}\to O$,
            see~\cite[Def~1.3, Def~3.1]{Fantechi_Manetti}.
            The space $O$ is called the \emph{obstruction space}.
            If $X$ is a scheme, then an \emph{obstruction theory} of a point $x\in X$ is an
            obstruction theory for the complete local ring $\hat{\mathcal{O}}_{X, x}$.

            Let $(A', \mm')$ be another complete local $\kk$-algebra with residue field
            $\kk$ and with a
            $\kk$-algebra
            homomorphism $f\colon A'\to A$ and let $O_A$, $O_{A'}$ be
            obstruction spaces for some obstruction theories for $A$
            and $A'$. A \emph{map (or morphism) of obstruction theories} is a linear map
            $O_f\colon O_{A}\to O_{A'}$ such that for every exact sequence
            $0\to J\to B\to B_0\to 0$  as
            above and $\varphi_0\colon A\to B_0$ the
            obstruction map for lifting $\varphi_0\circ f\colon A'\to B_0$ to $A'\to B$ is equal
            to $O_f \circ ob_{B\to B_0, \varphi_0}$.
        \end{definition}

        We stress that for a given $(A,\mm)$ many obstruction theories with
        different obstruction spaces exist. For example, given one such
        theory with obstruction space $O$, we can choose any space $O'$ with
        subspace $O \into O'$ and obtain a new theory with obstruction space
        $O'$.
        Geometrically, if $A = \hat{\mathcal{O}}_{X, x}$ for a $\kk$-point $x$
        of a scheme $X$, then a morphism $\varphi_0$ corresponds exactly to
        $\Spec(B_0)\to X$ and $\varphi$ to $\Spec(B)\to X$. So the question of
        lifting becomes the question of lifting a given map.
        \begin{example}
            In the example above, we lift a map from
            $\Spec(\kk[\varepsilon]/\varepsilon^2)\to X$ to
            $\Spec(\kk[\varepsilon]/\varepsilon^3)\to X$. More generally, we could
            try to lift this map to $\Spec(\kk[\varepsilon]/\varepsilon^n)$
            for $n = 3, 4,  \ldots $. If all those lifts exist, they
            glue to a map $\Spec(\kk[[t]])\to X$. The closure of the image of
            this map is either just $x$ or a curve passing through $x$.
        \end{example}
        \begin{example}\label{example:obs}
            If $A$ has an obstruction theory with obstruction space zero, then
            the lifting automatically exists for all homomorphisms. In
            this case $A$ is actually isomorphic to $\hatS$, hence, in the geometric sense, the point $x\in X$ is
            smooth. Indeed, if $\hatS\onto A$ is not an isomorphism then
            $\hatS/\mms^r\onto A/\mm^r$ is not an isomorphism for some $r$;
            take the smallest such $r$.
            Taking $B_0 = A/\mm^r$, $\varphi_0\colon A\to A/\mm^r$ the
            canonical map and $B=  \hatS/\mms^r$ we obtain a lifting
            $\varphi\colon A\to B$. From the
            definition of lifting it follows that $\varphi$ maps no nonzero
            linear form from $\mm$ to $\mms^2/\mms^r$, and consequently it
            maps no nonzero form of degree $k$ from $\mm^k$ to
            $\mms^{k+1}/\mms^r$ for $k<r$. However, this is a contradiction
            with $I\not\subset\mms^r$.
        \end{example}

        We will be interested in the following facts about obstruction
        theories:
        \begin{enumerate}
            \item By~\cite[Theorem~6.4.9]{fantechi_et_al_fundamental_ag} a point $[F/K]\in \Quotmain$ has an obstruction theory with
                obstruction space $\Ext^1(K, F/K)$. For comparison, note that
                the \emph{tangent} space at $[F/K]$ is $\Hom(K, F/K)$, by
                Lemma~\ref{ref:tangentSpaceToQuot:lem}.
            \item Suppose that $x\in X$ has obstruction theory $Ob_x$, $y\in Y$
                has obstruction theory $Ob_y$ and $f\colon X\to Y$ is a map of
                schemes with $f(x)= y$ that induces a map of obstruction
                theories $Ob_x\to Ob_y$. If this map is \emph{injective} and the
                tangent map $df\colon T_{X, x}\to T_{Y, y}$ is surjective then
                $f$ is smooth at $x$. This is called the Fundamental Theorem
                of obstruction calculus,
                see~\cite[Lemma~6.1]{Fantechi_Manetti}.
                If moreover $df$ is bijective, then $f$ is \'etale at $x$.

                \noindent This generalizes Example~\ref{example:obs}; indeed this
                example corresponds to the case $Y = \Spec(\kk)$, $Ob_x = 0$.
        \end{enumerate}

        Finally, we discuss \emph{primary obstructions}, which will be used in
        the proof of generic nonreducedness. For technical reasons we assume
        $\chr \kk \neq 2$. Let $A$ and $\hatS$ be as before.
        Since $I$ is contained in $\mms^2$ the surjection $\hatS/\mms^2\onto
        A/\mm^2$ is an isomorphism and in particular it gives a homomorphism
        $A/\mm^2\to \hatS/\mms^2$. Let $(O, ob)$ be any obstruction theory for
        $(A, \mm)$. Taking $B = \hatS/\mms^3$, $B_0 =
        \hatS/\mms^2$, $J = \mms^2/\mms^3$ and $\varphi_0\colon A\to A/\mm^2\simeq B_0$ the canonical projection in the definition of
        obstruction theory above, we obtain a map
        $ob_0\colon \left(\mms^2/\mms^3\right)^{\vee}\to O.
        $
        Since $\hatS$ is a power series ring, the domain of this map is
        dual to $\Sym_2(\mms/\mms^2)$ and we obtain the \emph{primary
        obstruction map}
        \[
            ob_0\colon (\Sym_2(\mms/\mms^2))^{\vee}\to O.
        \]
        Dualizing,
        we get a map $ob_0^{\vee}\colon O^{\vee}\to \Sym_2(\mms/\mms^2) = \mms^2/\mms^3$.
        Consider now a slightly more general
        situation. Fix a linear subspace $K\subset \mms^2$ and take $B_K =
        \hatS/(\mms^3 + K)$, $B_0 =\hatS/\mms^2$. Again using the definition, we
        get an obstruction map
        \[
            ob_K\colon \left( \frac{\mms^2}{\mms^3+K} \right)^{\vee}\to O.
        \]
        Moreover, by the functoriality from the
        Definition~\ref{ref:obstructiontheory:def},
        the map $B\to B_K$ induces a factorization $ob_K = ob_0 \circ f$ where
        $f$ is the dual to the canonical surjection $\mms^2/\mms^3\onto
        \mms^2/(\mms^3+K)$. Dualizing similarly as above we get a map
        $ob_K^{\vee}\colon O^{\vee}\to \mms^2/(\mms^3+K)$ which is the
        composition of $ob_0^{\vee}$ with the canonical projection. Now we are
        ready to observe that the following are equivalent:
        \begin{enumerate}
            \item the obstruction $ob_K$ is zero,
            \item the map $ob_K^{\vee}$ is zero,
            \item the image of the map $ob_{0}^{\vee}$ lies in $(K+\mms^3)/\mms^3$.
        \end{enumerate}
        We now use this observation to compute the image of dual to the primary
        obstruction map.
        We have $A/\mm^3\simeq
        \hatS/(\mms^3+I)$. We show
        that the canonical projection $\varphi_0\colon A\to B_0$ lifts to a
        $\kk$-algebra homomorphism $\varphi_K\colon A\to B_K$ if and only if
        ${I+\mms^3}\subset K{+\mms^3}$. In one
        direction this is clear: if
        ${I+\mms^3}\subset K{+\mms^3}$ we may
        take for $\varphi_K$ the composition of canonical
        projections $A\to A/\mm^3\to B_K$. Conversely, assume that
        $\varphi_K\colon A\to B_K$ is a $\kk$-algebra homomorphism which is a
        lifting of $\varphi_0$. Then $\varphi_K(\mm^3)=0$, so $\varphi_K$
        factors through $\overline{\varphi_K}\colon A/\mm^3\simeq
        \hatS/(\mms^3+I)\to B_K=\hatS/(\mms^3+K)$. If $\ell\in\hatS$ is any
        linear form, then it is clear that $\overline{\varphi_K}(\bar{\ell})$ is
        of the form $\bar{\ell}+\bar{q}$ for some $q\in\mms^2$. Since
        $\overline{\varphi_K}$ is a $\kk$-algebra homomorphism, it follows that
        ${I+\mms^3}\subset K{+\mms^3}$. Finally,
        the map $ob_K$ is zero if and
        only if the map $\varphi_0$ lifts to a $\kk$-algebra homomorphism $A\to
        B_K$. The above observations imply that $\im ob_0^{\vee}\subset
        (K+\mms^3)/\mms^3$ if and only if ${I+\mms^3}\subset K{+\mms^3}$, so $\im
        ob_0^{\vee}=(I+\mms^3)/\mms^3$.
        Intuitively, the above shows that $\im ob_0^{\vee}$ recovers the quadratic part
        of $I$.

        To make this useful, we need to know how to find the primary
        obstruction map explicitly. Recall that $\mms/\mms^2  \simeq
        \mm/\mm^2$ is called the cotangent space of $(A, \mm)$ and its dual
        $(\mms/\mms^2)^{\vee}  \simeq (\mm/\mm^2)^{\vee}$ is called the
        tangent space of $(A, \mm)$. In the case $A = \hat{\mathcal{O}}_{X, x}$ these spaces
        are the cotangent and tangent space to $X$ at $x$, respectively.
        Assume the characteristic of $\kk$ is not two. There is a $\GL(\mms/\mms^2)$-equivariant isomorphism
        \begin{equation}\label{eq:quadraticForms}
            \Sym_2((\mms/\mms^2)^{\vee})  \simeq (\Sym_2(\mms/\mms^2))^{\vee}
        \end{equation}
        which is unique up to
        multiplication by a scalar. We fix one such isomorphism: it maps
        the product $(y_i^*)\cdot (y_j^*)$ of the elements of the dual basis
        (which is an element of the left hand side)
        to the functional $(y_iy_j)^*$ on the right hand side
        for $i\neq j$ and it maps $(y_i^*)^2$ to $2(y_i^2)^*$. The reader might
        wonder why the constant $2$ is necessary; the reason is
        essentially the same as for $\frac{1}{2}$ constants that appear while
        writing a quadratic form as a symmetric matrix.
        Using~\eqref{eq:quadraticForms} we view $ob_0$ as
        a map from $\Sym_2((\mms/\mms^2)^{\vee})$. Consider $\varphi_1,
        \varphi_2\in (\mms/\mms^2)^{\vee}$. For $a\in A$ let $\bar{a}\in A$ be
        the image of $a$ under the composition $A\to A/\mm = \kk\into A$. The elements $\varphi_i$, $i=1,2$
        of the tangent space induce $\kk$-algebra homomorphisms $f_i\colon A/\mm^2\to
        \kk[\varepsilon_i]/\varepsilon_i^2$ given by $f_i(a) = \bar{a} +
        \varepsilon_i \varphi_i(a-\bar{a})$ which jointly give a
         $\kk$-algebra homomorphism $f_{12}\colon A/\mm^2\to \kk[\varepsilon_1,
        \varepsilon_2]/(\varepsilon_1, \varepsilon_2)^2$ defined by
        $f_{12}(a) = \bar{a} + \sum_{i=1}^2 \varepsilon_i
        \varphi_i(a-\bar{a})$. We have a commutative diagram
        \begin{equation}\label{eq:functorialityOfExtensions}\begin{tikzcd}
                    0\ar[r] & \mms^2/\mms^3\ar[d]\ar[r] &
                    \hatS/\mms^3 \ar[d]\ar[r] & \hatS/\mms^2 \simeq A/\mm^2
                    \ar[d, "f_{12}"]\ar[r] & 0\\
                    0\ar[r] & (\varepsilon_1\varepsilon_2)\ar[r] &
                    \frac{\kk[\varepsilon_1, \varepsilon_2]}{(\varepsilon_1^2,
                    \varepsilon_2^2)} \ar[r] & \frac{\kk[\varepsilon_1,
                    \varepsilon_2]}{(\varepsilon_1, \varepsilon_2)^2} \ar[r] &
                    0.
            \end{tikzcd}
        \end{equation}
        A diagram chase shows that the image of
        $(\varepsilon_1\varepsilon_2)^*$ in $(\mms^2/\mms^3)^{\vee}$ is equal,
        under the isomorphism~\eqref{eq:quadraticForms} to
        $\varphi_1\varphi_2$, this is proven most conveniently by changing
        coordinates on $\hatS$ so that $\varphi_i = y_i^*$ for $i=1,2$ or
        $\varphi_1 = \varphi_2 = y_1^*$.
        Using the functoriality from
        Definition~\ref{ref:obstructiontheory:def} we
        derive that the image of $(\varepsilon_1\varepsilon_2)^*$ under the composition $(\varepsilon_1\varepsilon_2)^{\vee}\to
        (\mms^2/\mms^3)^{\vee}\to O$ is the value of
        the primary obstruction on $\varphi_1\varphi_2\in \Sym_2((\mms/\mms^2)^{\vee})$.

        By Lemma~\ref{ref:tangentSpaceToQuot:lem} for a point $[F/K]$ on the Quot scheme, the tangent space
        is given by $\Hom(K, F/K)$. Fix a free resolution $F_{\bullet}$ of $F/K$, beginning
        with $F_0 = F$. For every $\varphi\colon K\to F/K$ in
        this
        space we can lift it to a chain complex map
        \begin{equation}\label{eq:chainComplex}
            \begin{tikzcd}
                &0  &\arrow[l] K\arrow[ldd, "\varphi"']   &\arrow[l, "d_0"]\arrow[ldd,
                "s_1(\varphi)"'] F_1 &\arrow[l,"d_1"'] F_2\arrow[ldd,
                "s_2(\varphi)"']  &\arrow[l,"d_2"'] F_3\arrow[ldd,
                "s_3(\varphi)"']  &\arrow[l] \ldots\\\\
                0&\arrow[l]F/K &\arrow[l, "\pi"'] F_0 &\arrow[l,"d_0"'] F_1
                &\arrow[l,"d_1"'] F_2  &\arrow[l] \ldots
            \end{tikzcd}
        \end{equation}
        \begin{theorem}\label{ref:primaryForQuot:thm}
            Consider a point $[F/K]$ on the Quot scheme and the associated
            obstruction theory with obstruction space $\Ext^1(K, F/K)$. Then
            for $\varphi_1, \varphi_2\in \Hom(K,
            F/K)$ the primary obstruction map sends
            $\varphi_1\varphi_2\in \Sym_2(\Hom(K, F/K))$ to
            \begin{equation}\label{eq:primaryOb}
                \pi\circ\left(s_1(\varphi_1)\circ s_2(\varphi_2) +
                s_1(\varphi_2)\circ s_2(\varphi_1)\right)
            \end{equation}
            in $\Ext^1(K, F/K)$.
        \end{theorem}
        We remark that
        \[
            \pi\circ s_1(\varphi_1) \circ s_2(\varphi_2)\circ d_2 = \varphi_1
            \circ d_0 \circ d_1 \circ s_3(\varphi_2) = \varphi_1 \circ 0 \circ
            s_3(\varphi_2) = 0,
        \]
        so indeed the expression~\eqref{eq:primaryOb} is a cycle in
        $\Hom(F_{\bullet}, F/K)$, hence it
        makes sense to take its homology class which is by definition the
        $\Ext^1(K, F/K)$ group. Below, we give two proofs. The first one is
        completely abstract and in essence tells that the result is known as a
        consequence of much deeper insights. The
        second one is down to earth, but certain details are left to the
        reader. Regretfully, we do not know a reference which is both complete
        and accessible.
        \begin{proof}[First proof of Theorem~\ref{ref:primaryForQuot:thm}]
            Let $M = F/K$ and let $(A, \mm)$ be the complete local ring of $[F/K]$ in the Quot
            scheme.
            Take $\varphi_1, \varphi_2\in \Hom(K, M)$. As discussed above the
            image of $\varphi_1\varphi_2$ in the primary obstruction is the
            image of $(\varepsilon_1\varepsilon_2)^*$ \rred{by the}
            obstruction map associated to the bottom row of
            Diagram~\ref{eq:functorialityOfExtensions}.
            Since $F$ is a
            free $S$-module, the map $\Ext^1(K, M) \to \Ext^2(M, M)$ is an isomorphism,
            thus we can compute the obstruction after forgetting that $M$ is a
            quotient of $F$. By the main
			result of~\cite{Fiorenza_Iacono_Martinengo} the infinitesimal
            deformation functor of $M$ is isomorphic to the Maurer-Cartan
            functor for the differential graded Lie algebra
            $\Endd(F_{\bullet})$. By~\cite[Section~5]{Manetti__notes}
            or~\cite[Example~2.16]{Manetti__Def_theory_of_DGLA} the primary
            obstruction for a differential graded Lie algebra $L$ is equal to the
            the Lie bracket in its cohomology algebra. The cohomology of
            $\Endd(F_{\bullet})$ is by definition $\Ext(M, M)$ and the bracket
            is just the bracket in the Yoneda pairing, i.e., for classes
            $\varphi_1, \varphi_2\in \Ext^1(M, M)$ it returns the sum of
            Yoneda products $\varphi_1\varphi_2 + \varphi_2\varphi_1$ in
            $\Ext^2(M, M)$, where
            the plus sign is due to the fact that we take brackets of odd
            degree elements, see~\cite[Definition~1.1]{Manetti__notes}. Now, the expression in the theorem is exactly
            this sum under the isomorphism $\Ext^2(M, M) \simeq \Ext^1(K, M)$.
        \end{proof}
        \begin{proof}[Sketch of second proof of Theorem~\ref{ref:primaryForQuot:thm}]
            \def\fonetwo{f_{\mathrm{tan}}}%
            Let $q\in \Sym_2(\Hom(K, F/K))$ be written as $q = \sum_{i\leq j}
            \lambda_{ij}\varphi_i
            \varphi_j$ where $\varphi_1, \ldots
            ,\varphi_d$ is a basis of $\Hom(K, F/K)$.
            Let $\varepsilon_1, \ldots ,\varepsilon_d\in \mms/\mms^2$ be a
            dual basis. We lift it to a generating set of $\mms$.
            Using~\eqref{eq:quadraticForms} we can interpret $q$
            as a functional $\bar{q}\colon \mms^2/\mms^3\to \kk$ and build the
            algebra
            \[
                B = \frac{\hatS}{\mms^3 + \ker(\bar{q})}.
            \]
            Let $B_0 = \hatS/\mms^2 = A/\mm^2$ and let $q^*\in \mms^2/\mms^3$
            satisfy $\bar{q}(q^*) = 1$. Using the explicit formula
            below~\eqref{eq:quadraticForms}, we check that
            in the algebra $B$ we have $\varepsilon_i \varepsilon_j =
            \lambda_{ij}q^*$ for $i\neq j$ while $\varepsilon_i^2 =
            2\lambda_{ii}q^*$.
            As discussed above, there is a natural morphism $\fonetwo\colon \Spec(B_0)\to
            \Quotmain$
            which by the universal property of
            Quot, see~Example~\ref{ex:Quotdef}, corresponds to a
            module $\mathcal{M} := (F\tensor_{\kk}
            B_0)/\mathcal{K}$.
            In this version of the proof we will only show that if the
            morphism $\fonetwo$ extends to a morphism $\Spec(B)\to
            \Quotmain$, then the element of $\Ext^1(K, M)$
            defined by the formula $\sum_{i\leq j} \lambda_{ij}\pi\circ (s_1(\varphi_i)\circ
            s_2(\varphi_j) + s_1(\varphi_j) \circ s_2(\varphi_i))$ from the Theorem is zero. (This is enough
            for the application in the current article.)

            For better clarity, we write subscript $B$ instead of
            $\otimes_{\kk} B$; similarly for $B_0$.
            Recall that $F_1$ is the
            free module with a canonical surjection $d_0\colon F_1\to K$. By Nakayama
            lemma any lift of $d_0$ to a $S_{B_0}$-module homomorphism $(F_1)_{B_0}\to \mathcal{K}$ is
            surjective.
            By the description of the tangent space to Quot scheme (see
            references in Lemma~\ref{ref:tangentSpaceToQuot:lem}) one such
            lift $d'_0\colon (F_1)_{B_0}\to \mathcal{K}$
            is
            \[
                d'_0(m) = d_0(m) + \sum_{i=1}^d \varepsilon_i s_1(\varphi_i)(m).
            \]
            Suppose that $\fonetwo$ extends to a map $\Spec(B)\to
            \Quotmain$. By the definition of Quot, Example~\ref{ex:Quotdef}, it follows that there
            exists an $S$-linear map $h\colon F_1\to F$ such that if we define a
            map $\tilde{d}_0\colon (F_1)_B\to F_{B}$ by $\tilde{d}_0(m) = d'_0(m) +
            q^* h(m)$, then $F_B/\im(\tilde{d}_0)$ is flat over
            $B$. Flatness implies that $\ker(\tilde{d}_0)$ reduces to
            $\ker(d_0)$ modulo $(\varepsilon_1, \ldots ,\varepsilon_d)$. Let $\tilde{d}_1\colon (F_2)_B\to
            \ker(\tilde{d}_0)$ be any $S_B$-module homomorphism that reduces to $d_1$
            modulo $(\varepsilon_1,  \ldots , \varepsilon_d)$. Nakayama lemma implies
            that $\tilde{d}_1$ is surjective. We obtain an exact sequence of
            $S_B$-modules
            \[
                \begin{tikzcd}
                    0 & \arrow[l] F_B/\im(\tilde{d}_0) & \arrow[l, "\tilde{\pi}_0"']
                    F_B & \arrow[l, "\tilde{d}_0"'] (F_1)_B  & \arrow[l,
                    "\tilde{d}_1"'] (F_2)_B
                \end{tikzcd}
            \]
            which $\pmod{(\varepsilon_1,  \ldots , \varepsilon_d)}$ gives the beginning of
            the resolution
            of $F/K$.
            Write $\tilde{d}_1|_{F_2} = d_1 +   \sum_{i=1}^d
            s_{2i}\varepsilon_i + q^* j$ where $s_{2i}, j\colon F_2\to F_1$. The condition $\tilde{d}_0\circ
            \tilde{d}_1 = 0$ decomposes in the basis $\varepsilon_1, \ldots
            ,\varepsilon_d, q^*$ of $B$ into equalities
            \begin{align*}
                \quad d_0 s_{2i} + s_1(\varphi_i)d_1 &= 0 \quad \mbox{for }
                i=1,2, \ldots ,d\\
                d_0 j + \sum_{i < j} \lambda_{ij}(s_1(\varphi_i)s_{2j} +
                s_1(\varphi_j)s_{2i}) + \sum_{i}
                2\lambda_{ii}s_1(\varphi_i)s_{2i} + hd_1 &=
                0.
            \end{align*}
            The first one implies that the maps ${-}s_{2i}\colon F_2\to
            F_1$ are lifts of $s_1(\varphi_i)$, for $i=1,2, \ldots ,d$, as in~\eqref{eq:chainComplex}.
            Since the class in $\Ext^1(K, F/K)$ does
            not depend on the choice of such lifts,
            we may take $s_2(\varphi_i) = {-}s_{2i}$. Then the third equation becomes
            \[
                d_0 j {-} \sum_{i \leq j}
                \lambda_{ij}(s_1(\varphi_i)s_2(\varphi_j) +
                s_1(\varphi_j)s_2(\varphi_i)) + hd_1 = 0.
            \]
            After composition with $\pi$ we obtain
            \[
                {-}\pi \circ \left(\sum_{i\leq j}
                \lambda_{ij}(s_1(\varphi_i)s_2(\varphi_j) +
                s_1(\varphi_j)s_2(\varphi_i))\right) + \pi h d_1 = 0,
            \]
            which shows that our class is equal to the boundary $(\pi h)\circ d_1$
            and it is zero in $\Ext^1(K, F/K)$.
        \end{proof}

    \subsection{Natural endomorphisms of $\CommMat$}

The following lemma is obvious, but we state it as we will frequently use it to assume some open conditions on the matrices.
\begin{lemma}\label{generic_assumption}
If $\mathcal{V}$ is an irreducible subvariety of $\CommMat$ and $\mathcal{U}$
is a nonempty Zariski-open subset of $\mathcal{V}$, then $\mathcal{U}$ and
$\mathcal{V}$ belong to the same irreducible component of $\CommMat$.\qed
\end{lemma}

    \begin{lemma}\label{polynomialMaps}
        Let $\varphi \colon \mathbb{A}^{\ambM}\to \mathbb{A}^{\ambM}$ be a polynomial map,
        and let the map $\CommMat\to \CommMat$ be defined by applying
        $\varphi$ to the $\ambM$-tuples of matrices. By a slight abuse of notation
        we denote it again by $\varphi$. If $\mathcal{C}$ is any irreducible
        component of $\CommMat$, then $\varphi(\mathcal{C})\subseteq
        \mathcal{C}$.
    \end{lemma}
    \begin{proof}
        Let $D\geq 1$ be greater than the degree of any coordinate of $\varphi$.
        The set of all polynomial maps $\mathbb{A}^{\ambM}\to
        \mathbb{A}^{\ambM}$ of degree at most $D$
        in each of the variables forms an affine space, say $\mathbb{A}^N$.
        Since $\mathcal{C}$ and $\mathbb{A}^N$ are irreducible, the image of
        the map $\mathbb{A}^N\times \mathcal{C}\to \CommMat$ defined by
        $(\psi,x_1,\ldots ,x_{\ambM})\mapsto \psi(x_1,\ldots
        ,x_{\ambM})$ is
        irreducible. This image clearly contains $\mathcal{C}$ because $id$ is
        among $\psi$. Since
        $\mathcal{C}$ is a component, the image is the whole $\mathcal{C}$.
        The lemma now follows.
    \end{proof}

    \begin{corollary}\label{algebra}
        Two tuples of commuting matrices that generate the same unital
        $\kk$-algebra lie in the same components of $\CommMat$.
    \end{corollary}
    \begin{proof}
        Fix any component $\mathcal{C}$ of $\CommMat$ containing
        $(x_1,\ldots ,x_{\ambM})$ and let $(y_1,\ldots ,y_{\ambM})\in \CommMat$
        be an $\ambM$-tuple that generates the same algebra as $(x_1,\ldots
        ,x_{\ambM})$. Then there exist polynomial maps $\varphi,\psi\colon
        \mathbb{A}^n\to \mathbb{A}^n$ with $(y_1,\ldots ,y_{\ambM})=\varphi
        (x_1,\ldots ,x_{\ambM})$ and $(x_1,\ldots ,x_{\ambM})=\psi (y_1,\ldots
        ,y_{\ambM})$. Lemma~\ref{polynomialMaps} implies that $(y_1,\ldots
        ,y_{\ambM})\in \varphi (\mathcal{C})\subseteq \mathcal{C}$.
        Symmetrically, if
        $(y_1,\ldots ,y_{\ambM})$ belongs to some component
        $\mathcal{C}'$, then $(x_1,\ldots ,x_{\ambM})\in
        \psi(\mathcal{C}')\subseteq \mathcal{C}'$, again by
        Lemma~\ref{polynomialMaps}.
    \end{proof}

    \begin{corollary}\label{ref:truncation:cor}
        Suppose a tuple $(x_1, \ldots ,x_n)\in \CommMat$ lies on a
        non-elementary component. Then also its sub-tuple $(x_1,x_2,
        \ldots ,x_{n-1})\in C_{n-1}(\MM_d)$ lies on a non-elementary
        component.
    \end{corollary}
    \begin{proof}
        Let the non-elementary component containing $(x_1, \ldots ,x_n)$ be
        $\mathcal{Z}$ and let $\pi\colon \CommMat\to C_{n-1}(\MM_d)$ forget
        the last matrix. Then $\pi(\mathcal{Z})$ is an irreducible locus
        containing $(x_1, \ldots ,x_{n-1})$. By Corollary~\ref{algebra} or
        simply by $\GL_n$-action, for a general tuple in $\mathcal{Z}$ the
        first matrix has at least two eigenvalues. The same is then true for
        $\pi(\mathcal{Z})$ which shows that this locus is contained in a
        non-elementary component.
    \end{proof}
    \newcommand{\shiftmap}[1]{\operatorname{shift}_{#1}}%
    Important classes of maps $\varphi$ above are translations and linear coordinate changes.
    More precisely, for a tuple $\alpha_{\bullet} = (\alpha_1, \ldots , \alpha_{\ambM})\in
    \kk^{\ambM}$, we have a map $\mathbb{A}^{\ambM}\to \mathbb{A}^{\ambM}$ that
    translates by $\alpha_{\bullet}$, so we also have \emph{translation map}
    $\shiftmap{\alpha}\colon \CommMat \to \CommMat$:
    \begin{equation}\label{eq:translationMap}
        \shiftmap{\alpha}(x_1, \ldots ,x_{\ambM}) = (x_1 + \alpha_1 I_d, \ldots
        ,x_{\ambM} + \alpha_{\ambM}I_d).
    \end{equation}
    In fact, the $\shiftmap{-}$ is an action of $(\mathbb{A}^{\ambM},+)$ on
    $\CommMat$.
    Similarly, for a linear coordinate change $\mathbb{A}^{\ambM}\to
    \mathbb{A}^{\ambM}$ we have an induced $\GL_{\ambM}$-action on $\CommMat$,
    where $A = [a_{ij}]\in \GL_{\ambM}$
    acts by
    \[
        A\circ (x_1,\ldots ,x_{\ambM})=A\cdot (x_1, \ldots ,x_{\ambM})^T :=
        \left( \sum_{j} a_{ij}x_j \right)_i.
    \]
    The actions of $\GL_{\ambM}$ and $\mathbb{A}^{\ambM}$ do not commute, rather they form a
    semidirect product; the group of affine coordinate changes. Both
    $\GL_{\ambM}$
    and $\mathbb{A}^{\ambM}$ commute with the $\GL(V)$-action.

    Let $x_1$ be a nilpotent matrix in the Jordan form
    and let $a_1\le a_2\le \cdots \le a_m$ be the sizes of Jordan blocks of $x_1$. Consider a block matrix $y$ in the form
        \[
            y = \left[
                \begin{array}{cccc}
                    y_{a_m,a_m}&y_{a_m,a_{m-1}}&\cdots&y_{a_m,a_1}\\
                    y_{a_{m-1},a_m}&y_{a_{m-1},a_{m-1}}&\cdots&y_{a_{m-1},a_1}\\
                    \vdots&\vdots&\ddots&\vdots\\
                    y_{a_1,a_m}&y_{a_1,a_{m-1}}&\cdots&y_{a_1,a_1}
                \end{array} \right]
        \]
        where $y_{k,l}$ are $k \times l$
        matrices.
        We say that $y$ is an \emph{upper-triangular Toeplitz matrix} if each
        $y_{k,l}$ is an upper-triangular Toeplitz matrix, i.e., it has the form
        \[
            \begin{bmatrix}
                0 & \ldots & 0 &z_0 & z_1 & z_2 & \ldots & z_{k-1}\\
                0 & \ldots & 0 &0 & z_0 & z_1 &  \ldots & z_{k-2}\\
                0 & \ldots & 0 &0 & 0 & z_0 &  \ldots & z_{k-3}\\
                \vdots & & \vdots  & &&\ddots & \ddots & \vdots\\
                0 & \ldots & 0 & & \ldots & & 0 & z_0
             \end{bmatrix}\quad  \mathrm{or}\quad
             \begin{bmatrix}
                z_0 & z_1 & z_2 & \ldots & z_{l-1}\\
                0 & z_0 & z_1 & \ldots & z_{l-2}\\
                \vdots & \ddots & \ddots & \ddots & \vdots\\
                0 & \ldots & 0 & z_0 & z_1\\
                0 & \ldots & \ldots & 0 & z_0\\
                0 & & \ldots & & 0\\
                \vdots & & & & \vdots\\
                0 & & & & 0
             \end{bmatrix}.
        \]

    \begin{lemma}\label{ref:toeplitz:lem}
        Suppose that $(x_1, \ldots ,x_n)\in \CommMat$ where $x_1$ is nilpotent and in the Jordan canonical form.
        Then each $x_i$, $i\geq 2$ is an upper-triangular Toeplitz matrix.
    \end{lemma}
    \begin{proof}
        \emph{The proof is well-known, we give it to introduce some notation
        and concepts used later.}
        Recall that $x_i$ are linear operators on $V$. Introduce a
        $\kk[t]$-module structure on $V$ by $t\cdot v = x_1(v)$. Since all
        $x_i$ commute with $x_1$, they are endomorphisms of the resulting
        $\kk[t]$-module.
        Since $x_1$ is in Jordan form, this module is isomorphic to
        \[
            \bigoplus_{i=1}^m \frac{\kk[t]}{(t^{a_i})}
        \]
        where $a_1, \ldots , a_m$ are sizes of Jordan blocks of $x_1$. Let
        $M_i := \kk[t]/(t^{a_i})$. Then $V = \bigoplus_i M_i$, so every
        endomorphism of $V$ has the form

        \[
            \left[
                \begin{array}{cccc}
                    y_{a_m,a_m}&y_{a_m,a_{m-1}}&\cdots&y_{a_m,a_1}\\
                    y_{a_{m-1},a_m}&y_{a_{m-1},a_{m-1}}&\cdots&y_{a_{m-1},a_1}\\
                    \vdots&\vdots&\ddots&\vdots\\
                    y_{a_1,a_m}&y_{a_1,a_{m-1}}&\cdots&y_{a_1,a_1}
            \end{array} \right]
        \]
        for some $y_{i,i'}\in \kk[t]$. Moreover if $a_{i} > a_{i'}$ then
        $y_{i, i'}$ is divisible by $t^{a_{i} - a_{i'}}$. It remains to go back to matrices and
        see that $y_{i,j} = z_0+z_1t+z_2t^2 +  \ldots $ with $z_i\in \kk$
        corresponds to the upper-triangular Toeplitz matrix above.
    \end{proof}
    \begin{remark}\label{ref:firstBlockZero:rmk}
        By Corollary~\ref{algebra}, we may add a polynomial in $x_1$ to each
        $x_i$, $i\geq 2$ so that a chosen
        diagonal block $y_{k,k}$ is zero. We will frequently use this to make
        $y_{a_m, a_m}$ a zero block.
    \end{remark}

    \subsection{Concatenation and components}

    \newcommand{\concat}{\operatorname{concat}}%
    For any $d, d'$ the \emph{naive concatenation map} $\CommMatParam{d}\times
    \CommMatParam{d'}\to \CommMatParam{d+d'}$
    maps tuples $\xx = (x_i)$,
        $\xx' = (x'_i)$ to the tuple
    \[
        \left(\begin{bmatrix} x_i & 0\\ 0 &
        x'_i\end{bmatrix} \ |\ i=1,2, \ldots ,n\right).\]
    This map is nowhere dominant since we can conjugate the target tuple by an
    element of $\GL_{d+d'}$. To remedy this, we consider the \emph{concatenation map}
    $\concat\colon\GL_{d+d'}\times \CommMatParam{d}\times
    \CommMatParam{d'}\to \CommMatParam{d+d'}$ that sends a triple $(g'', \xx,
    \xx')$ to the tuple
    \[
        \left(g''\cdot \begin{bmatrix} x_i & 0\\ 0 &
        x'_i\end{bmatrix}\cdot (g'')^{-1}\ |\ {i=1,2, \ldots ,n}\right).
    \]
    In the theorem below, we show that $\concat$ map is usually dominant.

    We say that tuples $\xx,
    \xx'$ have \emph{intersecting supports} if for every $i$ the matrices
    $x_i$ and $x_i'$ have a common eigenvalue. We say that tuples have
    \emph{disjoint supports} if they do not have intersecting supports.  If
    the tuples have disjoint supports then the supports of the modules
    associated to $\xx$ and $\xx'$ are disjoint, but not vice versa,
    see~\S\ref{ssec:dictionary}.

    \begin{proposition}[Concatenation of
        components]\label{ref:productOfComponents:prop}%
        \phantom{AA}
        \begin{enumerate}
            \item\label{it:bundle} Let $\xx\in \CommMatParam{d}$ and $\xx'\in
                \CommMatParam{d'}$ be tuples with disjoint supports. Let $g\in
                \GL_{d+d'}$.
                Then $\concat$ is a smooth map near the point $(g, \xx,
                \xx')$. The fiber $\concat^{-1}(\concat(g, \xx, \xx'))$
                has dimension $d^2 + (d')^2$ at $(g, \xx, \xx')$.
            \item\label{it:components} Let $\mathcal{C}$ be an irreducible component of
                $\CommMatParam{d}$ and let $\mathcal{C}'$ be an irreducible
                component of $\CommMatParam{d'}$. Then the closure of $\GL_{d+d'}\cdot
                (\mathcal{C}\times \mathcal{C}')$ is an irreducible component of
                $\CommMatParam{d+d'}$ which has dimension
                $2d d' + \dim \mathcal{C} + \dim \mathcal{C}'$.
                We call this component the \emph{concatenation} of
                $\mathcal{C}$ and $\mathcal{C}'$. If $\mathcal{C},
                \mathcal{C}'$ have smooth points then their concatenation also
                has a smooth point.
        \end{enumerate}
    \end{proposition}
    \begin{proof}
        \def\mm{\mathfrak{m}}%
        We first discuss how~\eqref{it:components} follows
        from~\eqref{it:bundle}.
        Choose general points $\xx$ and $\xx'$ of $\mathcal{C}, \mathcal{C}'$.
        Thanks to the translation maps~\eqref{eq:translationMap}, we assume
        that for every $i$ the matrices $x_i$ and $x_i'$ have disjoint
        eigenvalues. Then there exists an open neighbourhood $W$ of $(\xx, \xx')$
        in $\mathcal{C}\times \mathcal{C}'$ where this condition holds.
        By~\eqref{it:bundle}, the map $\concat|_{\GL_{d+d'} \times W}\colon
        \GL_{d+d'}\times W\to \CommMatParam{d+d'}$ is smooth and
        its fibers have
        dimension $d^2 + (d')^2$ at every point. Therefore the image of this
        map is open and has
        dimension $(d+d')^2 + \dim\mathcal{C} + \dim \mathcal{C}' - d^2 -
        (d')^2$. Since $W$ is open in $\mathcal{C} \times \mathcal{C}'$ the
        scheme $\GL_{d+d'}\times W$ is irreducible hence its image is an open
        irreducible subset of required dimension.

        \def\bb{\mathbf{b}}%
        Now we prove~\eqref{it:bundle}.
        Let $\xx'' = \concat(g, \xx, \xx')$.
        Let $M$, $M'$ be the modules associated to $\xx$, $\xx'$. By
        Lemma~\ref{ref:matricesaremoduleswithbasis} they come with canonical
        $\kk$-linear bases $\bb$, $\bb'$.
        The module associated to
        point $\xx''$ is $M\oplus M'$ by Lemma~\ref{ref:GLVaction:lem} and
        comes with a canonical basis $\bb'' = g\cdot (\bb, \bb')$.
        The point
        $\xx''\in\CommMatParam{d+d'}$ has an obstruction theory with obstruction group
        $\Ext^2(M\oplus M', M\oplus M')$, see
        Corollary~\ref{ref:obstructionTheoryForCommMat:cor}.
        Similarly, the
        point $(g, \xx, \xx')\in \GL_{d+d'}\times \CommMatParam{d}\times
        \CommMatParam{d'}$ has an obstruction theory with obstruction group
        $\Ext^2(M, M)\oplus \Ext^2(M', M')$. By an examination of
        \rred{the} construction of these obstruction
        theories~\cite[Theorem~6.4.9]{fantechi_et_al_fundamental_ag}, the map
        $\concat$ induces a map of
        those theories which is injective onto the direct summand $\Ext^2(M,
        M)\oplus \Ext^2(M', M')$ of $\Ext^2(M\oplus M', M\oplus
        M')$. If we prove that
        \[
            d\concat\colon T_{(g, \xx, \xx')}\GL_{d+d'}\times
            \CommMatParam{d}\times \CommMatParam{d'} \to T_{\xx''} \CommMatParam{d+d'}
        \]
        is surjective, then by the Fundamental Theorem of obstruction calculus, see
        \S\ref{ssec:obstructions}, the map $\concat$ is smooth at $(g, \xx,
        \xx')$. Then the
        dimension of the fiber of $\concat$ at $(g, \xx, \xx')$ is computed as
        $\dim_{\kk}\ker (d\concat)$.

        The vector space $T_{\xx''} \CommMatParam{d+d'}$ is
        described in Lemma~\ref{ref:tangentspace:lem} as tuples of commuting
        matrices in $\MM_{d+d'}(\kk[\varepsilon]/\varepsilon^2)$ which reduce to $\xx''$
        modulo $\varepsilon$. We fix such a tuple $\widetilde{\xx''}$. Arguing as in
        Lemma~\ref{ref:matricesaremoduleswithbasis}, we obtain an associated module
        $\mathcal{M}''$ over $S[\varepsilon]/\varepsilon^2$ together with a
        fixed $\kk[\varepsilon]/\varepsilon^2$-linear basis $\widetilde{\bb''}$. Since $\widetilde{\xx''}$
        reduces to $\xx''$ modulo $\varepsilon$, we have
        $\mathcal{M}''/\varepsilon\mathcal{M}'' \simeq M\oplus M'$ and the basis
        $\widetilde{\bb''}$ reduces to $\bb''$ modulo $\varepsilon$.

        \def\wI{J}%
        \def\wIp{J'}%
        We claim that the direct sum decomposition $\mathcal{M}''/\varepsilon\mathcal{M}''
        \simeq M\oplus M'$ lifts to a direct sum decomposition of
        $\mathcal{M}''$.
        As in Section~\ref{ssec:dictionary}, for $I = \Ann(M)$ and $I' =
        \Ann(M')$ we have that the set of maximal ideals containing
        $I$ (respectively, $I'$) is the support of $M$ (respectively, of
        $M'$). Since these supports are disjoint, no maximal ideal contains
        $I+I'$, so $I+I' = (1)$. Since $I'M' = 0$ and $IM = 0$, we have
        \[
            I(M\oplus M') = IM' = IM' + I'M' = (I+I')M' = M'
        \]
        and similarly $I'(M\oplus M') = M$, so
        we obtain $II'(M\oplus M') = 0$.
        Since $\mathcal{M}''/\varepsilon \mathcal{M}'' = M\oplus M'$, we have that
        $II'\mathcal{M}''\subset \varepsilon \mathcal{M}''$.
        The multiplication by $\varepsilon$ gives a surjection
        $\mathcal{M}''/\varepsilon \mathcal{M}''\to \varepsilon
        \mathcal{M}''$. The subset $II'\varepsilon \mathcal{M}''$ is the image of
        $II'(\mathcal{M}''/\varepsilon \mathcal{M}'') = II'(M \oplus M') =
        0$, so
        that $(II')^2 \mathcal{M}'' \subset II'\varepsilon \mathcal{M}'' = 0$. Since $I+I' = (1)$ also $I^2 + (I')^2
        = (1)$.
        Let $\wI, \wIp \subset
        S[\varepsilon]/\varepsilon^2$ be the ideals generated by $I$, $I'$
        respectively. They satisfy $\wI^2 + \wIp^2 = (1)$ and
        $\wI^2\wIp^2\mathcal{M}'' = 0$
        so $\wI^2\mathcal{M}''\cap \wIp^2\mathcal{M}'' = (\wI^2+\wIp^2)\cdot(\wI^2\mathcal{M}''\cap \wIp^2\mathcal{M}'') \subset
        \wI^2\wIp^2\mathcal{M}'' = 0$. By Chinese Remainder
        Theorem, we obtain
        \[
            \mathcal{M}'' = \frac{\mathcal{M}''}{\wI^2\mathcal{M}''\cap
                \wIp^2\mathcal{M}''}  \simeq
                \frac{\mathcal{M}''}{\wI^2\mathcal{M}''} \oplus
                \frac{\mathcal{M}''}{\wIp^2\mathcal{M}''}.
        \]
        Let $\mathcal{M} := \mathcal{M}''/\wI^2\mathcal{M}''$ and $\mathcal{M}' :=
        \mathcal{M}''/\wIp^2\mathcal{M}''$. We have
        \[
            \frac{\mathcal{M}}{\varepsilon\mathcal{M}} \simeq
            \frac{\mathcal{M}''}{(\varepsilon)\mathcal{M}''+\wI^2\mathcal{M}''}
             \simeq
            \frac{\frac{\mathcal{M}''}{\varepsilon
                \mathcal{M}''}}{I^2\left(\frac{\mathcal{M}''}{\varepsilon
                \mathcal{M}''}\right)}
            \simeq \frac{M\oplus M'}{I^2(M\oplus M')} = \frac{M\oplus M'}{M'}
            = M,
        \]
        so the submodule $\mathcal{M}$ reduces to $M$ modulo $\varepsilon$ and
        the same goes for $\mathcal{M}'$. Via the direct sum $\mathcal{M}'' =
        \mathcal{M}\oplus \mathcal{M}'$ we view $\mathcal{M}$,
        $\mathcal{M}'$ as submodules of $\mathcal{M}''$. Since $\wI$, $\wIp$
        are coprime, the submodule $\mathcal{M}$ is exactly the set of
        elements of $\mathcal{M}''$ annihilated by $\wI^2$, and similarly for
        $\mathcal{M}'$.

        The $S[\varepsilon]/\varepsilon^2$-modules $\mathcal{M}$, $\mathcal{M}'$ are
        free $\kk[\varepsilon]/\varepsilon^2$-modules as direct
        summands of the free $\kk[\varepsilon]/\varepsilon^2$-module $\mathcal{M}''$.
        Fix any $\kk[\varepsilon]/\varepsilon^2$-linear bases
        $\widetilde{\bb}$, $\widetilde{\bb'}$ of $\mathcal{M}$,
        $\mathcal{M}'$ which reduce to bases
        $\bb$, $\bb'$ respectively. By the argument of
        Lemma~\ref{ref:matricesaremoduleswithbasis} from the pair $(\mathcal{M},
        \widetilde{\bb})$ we obtain a tuple $\widetilde{\xx}$ of commuting matrices with entries
        in $\kk[\varepsilon]/\varepsilon^2$ which reduces to $\xx$ modulo
        $\varepsilon$. By Lemma~\ref{ref:tangentspace:lem} this gives an
        element of $T_{\xx}\CommMatParam{d}$ which we denote also by
        $\widetilde{\xx}$. From $(\mathcal{M}',
        \widetilde{\bb'})$ we analogously get an element
        $\widetilde{\xx'}\in T_{\xx'}\CommMatParam{d'}$. The bases
        $\widetilde{\bb''}$ and $g\cdot(\widetilde{\bb}, \widetilde{\bb'})$
        are two bases that reduce to $\bb'' = g\cdot(\bb, \bb')$ modulo
        $\varepsilon$, so there exists a unique element $\widetilde{g}\in
        T_{g}\GL_{d+d'}$ such that $\widetilde{g} \cdot(\widetilde{\bb},
        \widetilde{\bb'}) = \widetilde{\bb''}$.
        If we view elements of tangent space as morphisms from
        $\Spec(\kk[\varepsilon]/\varepsilon^2)$, then the tangent map
        $d\concat$ is given by composing them with $\concat$. Therefore,
        the triple $(\widetilde{g},
        \widetilde{\xx}, \widetilde{\xx'})$ in the source of $d\concat$ maps
        to $\widetilde{\xx''}$.
        It remains to compute the kernel of $d\concat$. So we assume
        $\widetilde{\xx''} = \xx''$ and $\mathcal{M}'' =
        (M\oplus M')[\varepsilon]/\varepsilon^2$. If an element $(\widetilde{g},
        \widetilde{\xx}, \widetilde{\xx'})$ maps to $\widetilde{\xx''}$, then
        it induces a decomposition of $\mathcal{M}'' = \mathcal{N}\oplus
        \mathcal{N}'$. Since
        $\widetilde{\xx}$ reduces to $\xx$ modulo $\varepsilon$, we have
        $\mathcal{N}/\varepsilon \mathcal{N} = M$. Since $IM = 0$, we have
        $\wI\mathcal{N} \subset \varepsilon \mathcal{N}$ and so
        $\wI^2 \mathcal{N} = 0$. Similarly
        $\wIp^2 \mathcal{N}' =
        0$, so by uniqueness above, we have $\mathcal{N} = \mathcal{M}$ and
        $\mathcal{N}' = \mathcal{M}'$ and hence the triple $(\widetilde{g},
        \widetilde{\xx}, \widetilde{\xx'})$ is uniquely determined by a choice
        of bases $\widetilde{\bb},
        \widetilde{\bb'}$: for these we have respectively a $d^2$ and $(d')^2$ dimensional
        space of choices.
    \end{proof}
    The component part of the previous theorem was proven by different means
    in~\cite{CrawleyBoevey}. However, it seems that the smoothness of the map
    $\concat$ and the dimension of fibers, which are important for us, are not
    present in that paper.

    \subsection{Specific Jordan blocks}\label{ssec:specificJordanBlocks}
    \begin{lemma}[Very big largest Jordan block]\label{big_largest_Jordan}
        Let $(x_1,\ldots ,x_n)\in \CommMat$ be an $n$-tuple of nilpotent
        matrices such that the sizes $a_m>a_{m-1}\ge \cdots \ge a_1$ of Jordan
        blocks of the Jordan form of some linear combination of $x_1,\ldots
        ,x_n$ satisfy $a_m>2a_{m-1}$. Then the $n$-tuple $(x_1,\ldots ,x_n)$
        belongs to a non-elementary component of $\CommMat$.
    \end{lemma}
    \begin{proof}
        The action of $\GL_n \times \GL(V)$ stabilizes the irreducible components of $\CommMat$, therefore we may assume that
         $x_1$ is in the Jordan canonical form:
         \[
             x_1=\left[
                 \begin{array}{ccc}J_{a_m}\\&\ddots\\&&J_{a_1}
                 \end{array}
             \right]\]
        where $J_{a_k}$ denotes the nilpotent Jordan block of size $a_k$ for
        each $k=1,\ldots ,m$. By Lemma~\ref{ref:toeplitz:lem} the matrices $x_2,\ldots ,x_n$ are of the form
        \[
            x_i=\left[
                \begin{array}{cccc}
                    x_{a_m,a_m}^{(i)}&x_{a_m,a_{m-1}}^{(i)}&\cdots&x_{a_m,a_1}^{(i)}\\
                    x_{a_{m-1},a_m}^{(i)}&x_{a_{m-1},a_{m-1}}^{(i)}&\cdots&x_{a_{m-1},a_1}^{(i)}\\
                    \vdots&\vdots&\ddots&\vdots\\
                    x_{a_1,a_m}^{(i)}&x_{a_1,a_{m-1}}^{(i)}&\cdots&x_{a_1,a_1}^{(i)}
                \end{array} \right]
        \]
    where each $x_{a_j,a_k}^{(i)}$ is an $a_j\times a_k$ upper triangular and
    Toeplitz matrix. By Remark~\ref{ref:firstBlockZero:rmk} we assume that
    $x_{a_m, a_m}^{(i)} = 0$ for every $i\geq 2$. In particular, since $a_m>2a_{m-1}$, the $(a_{m-1}+1)$-th
    row and column of the matrix $x_i$ are zero for each $i\ge 2$.
    Consequently, the matrix $E_{a_{m-1}+1,a_{m-1}+1}$ that has 1 at the
    intersection of the $(a_{m-1}+1)$-th row and column and zeros elsewhere
    commutes with $x_i$ for $i\ge 2$.
    For a nonzero $\lambda$, the matrix $x_1 + \lambda
    E_{a_{m-1}+1,a_{m-1}+1}$ is not nilpotent and in fact has more than one
    eigenvalue. Therefore the
    line $\{(x_1+\lambda
        E_{a_{m-1}+1,a_{m-1}+1},x_2,\ldots ,x_n):\lambda \in \kk\}$ intersects
        \rred{in an open subset}
        some component of $\CommMat$ that is
        non-elementary. But a component is
        closed, so it contains the whole
        line, in particular $(x_1, \ldots ,x_{n})$, and the lemma follows.
    \end{proof}

    \begin{lemma}[One Jordan block]\label{1-dim_kernel}
        Let $d$ be arbitrary and let $(x_1,\ldots ,x_n)\in \CommMat$ be an
        $n$-tuple of nilpotent matrices such that some linear combination of
        $x_1,\ldots ,x_n$ has a one-dimensional kernel. Then the $n$-tuple
        $(x_1,\ldots ,x_n)$ belongs to the principal component of $\CommMat$.
    \end{lemma}
    \begin{proof}
        Using the actions of $\GL(V)$ and $\GL_n$ we may assume that
        $x_1 = J_d$ has a one-dimensional kernel, so by
        Lemma~\ref{ref:toeplitz:lem} every other matrix in the tuple is a
        polynomial in $x_1$ so the tuple generates the algebra
        $\kk[x_1]$. By Corollary~\ref{algebra} it follows that this tuple lies
        in the principal component.
    \end{proof}

    \begin{proposition}[Two Jordan blocks
        structure]\label{2-dim_kernel-structure}
        Let $(x_1,\ldots ,x_n)\in \CommMat$ be an
        $n$-tuple of nilpotent matrices such that some linear combination of
        $x_1,\ldots ,x_n$ has a two-dimensional kernel.
        Pick such a combination and let $\Delta$ be the difference of the
        sizes of its two Jordan blocks.
        Then up to
        $\GL(V)$-action and nonlinear change of generators as in
        Corollary~\ref{algebra}, there exists a $j$ such that the $n$-tuple
        $(x_1,\ldots ,x_n)$ is a limit of tuples of the form
        \begin{equation}\label{eq:2blocks_canonical_form}
            x_1 = \begin{bmatrix}
                t&\\&t
                \end{bmatrix},\ \ x_2 = \begin{bmatrix}
                    t^j & 0\\
                    0 & 0
                \end{bmatrix},\ \ x_i = \begin{bmatrix}
                    0 & b_i\\
                    c_i & 0
                \end{bmatrix}\mbox{ for }i\geq 3
            \end{equation}
        where for every $i\geq 3$ we have $c_i\in t^{j}\kk[t]$
        and $b_i\in t^{j+\Delta}\kk[t]$ and $t^jx_i
        = 0$. Here, we use the ``polynomial'' notation from
        Lemma~\ref{ref:toeplitz:lem}.
    \end{proposition}

    \begin{proof}
        \def\coorder{\operatorname{coorder}}%
        \def\ord{\operatorname{ord}}%
        After a coordinate change we may assume that $x_1$ is in the Jordan
        canonical form with two Jordan blocks of sizes $m\ge k$:
        \[
            x_1=\left[
                \begin{array}{cc}J_m\\&J_k
                \end{array}
            \right].\]
        Introduce a $\kk[t]/(t^m)$-module structure on $V$ where $t\cdot v = x_1(v)$
        for all $v\in V$. Then $V = \kk[t]/(t^m) \oplus \kk[t]/(t^k)$, see
        Lemma~\ref{ref:toeplitz:lem}. Recall that the only ideals in
        $\kk[t]/(t^m)$ are generated by powers of $t$.
        For an element
        $f\in \kk[t]$ let the valuation of $f$, denoted $\nu(f)$ be the
        maximal power of $t$ that divides $f$.
        This valuation descends to a function on $\kk[t]/(t^m)$. For any $g,h\in
        \kk[t]/(t^m)$, we have $\nu(g)\leq \nu(h)$ if and only if the ideal
        $(g)$ contains the ideal $(h)$ in $\kk[t]/(t^m)$.

        Below we repeatedly make use of Corollary~\ref{algebra}.
        We have $\Delta = m - k\geq 0$.
        Write each $x_i$ as a matrix
        \begin{equation}\label{eq:xasmatrix}
            x_i =
            \begin{bmatrix}
                a^{(i)} & b^{(i)}\\
                c^{(i)} & d^{(i)}
            \end{bmatrix}
        \end{equation}
        for $a^{(i)}\in \kk[t]/(t^m)$, $b^{(i)}\in t^{\Delta}\kk[t]/(t^m)$ and $c^{(i)}, d^{(i)}\in
        \kk[t]/(t^k)$.
        By definition, the matrix $x_1$ is diagonal with both diagonal entries equal
        to $t$.         Let the \emph{vanishing order} $\ord(x_i)$ of $x_i$ be defined as
        $\min\left(\nu(a^{(i)}-d^{(i)}), \nu(c^{(i)}), \nu(b^{(i)}) - \Delta\right)$.
        Let $j = \min(\ord(x_i)\ |\ i=2,3, \ldots )$. After a coordinate
        change, we may assume $j = \ord(x_2)$. Moreover, if $\Delta =0$ we may use the $\GL_2$-action
        on the ring of the endomorphisms of the $\kk[t]/(t^m)$-module $V$ to assume that the
        matrix of the coefficients of $x_2$ at $t^j$ is lower triangular (for
        example, the transpose of a Jordan normal form), in particular
        $b^{(2)}\in t^{j+1}\kk[t]/(t^m)$.

        We subtract an appropriate polynomial in $x_1$ from each $x_i$, $i\geq
        2$, and hence obtain $d^{(i)} = 0$ for all $i\geq 2$. Note that this does not change the vanishing orders of the matrices $x_i$.
        Therefore we have
        \begin{equation}\label{eq:commutator}
            [x_i, x_j] = x_ix_j - x_jx_i = \begin{bmatrix}
                b^{(i)}c^{(j)} - b^{(j)}c^{(i)} & a^{(i)}b^{(j)} -
                a^{(j)}b^{(i)}\\
                c^{(i)}a^{(j)} - c^{(j)}a^{(i)} & c^{(i)}b^{(j)} -
                c^{(j)}b^{(i)}
            \end{bmatrix}.
        \end{equation}


        We will now divide $x_3, \ldots ,x_n$ by $x_2$ with remainder:
        we claim that each $x_i$ for $i\geq 3$ can be written as $x_i' + r_i$,
        where $x_i'\in x_2\kk[x_1]$ and
        \[
            r_i = \begin{bmatrix}
                a_i' & b_i'\\
                c_i' & 0
            \end{bmatrix}\mbox{ with } t^{j}r_i = \begin{bmatrix}
                * & 0\\
                0 & 0
            \end{bmatrix}\mbox{ and } a_i', c_i'\in t^j \kk[t],\ b_i'\in
            t^{j+\Delta} \kk[t].
        \]
        We have three cases.
        \begin{enumerate}
            \item $j = \nu(a^{(2)})$. Then $a^{(2)} = t^{j}u$ with $u\in
                \kk[t]/(t^m)$ invertible. For all $i\geq 3$ we have $j \leq
                \ord(x_i) \leq \nu(a^{(i)})$, hence $\nu(a^{(2)}) \leq
                \nu(a^{(i)})$, so $a^{(i)}$ is a multiple of $a^{(2)}$ in
                $\kk[t]/(t^m)$, say $a^{(i)} = a^{(2)}\cdot q(t)$ where $q(t)\in
                \kk[t]$.
                Let $x_i' := x_2\cdot q(x_1)$ and let $r_i = x_i - x_i'$. Then $r_i$
                is a matrix
                \[
                    \begin{bmatrix}
                        0 & b_i'\\
                        c_i' & 0
                    \end{bmatrix}.
                \]
                The commutativity $[x_2, x_i] = 0$ yields $[x_2, r_i] = 0$, so
                by~\eqref{eq:commutator} we get
                \[
                    0 = \begin{bmatrix}
                        b^{(2)}c'_i - b'_ic^{(2)} & a^{(2)}b'_i\\
                        - c'_ia^{(2)} & c^{(2)}b'_i -
                        c'_ib^{(2)}
                    \end{bmatrix}
                \]
                so indeed
                \[
                    t^{j}\cdot r_i = t^{j}\cdot \begin{bmatrix}
                        0 & b_i'\\
                        c_i' & 0
                    \end{bmatrix} = \begin{bmatrix}
                        0 & t^{j}b'_i\\
                        t^{j}c'_{i} & 0
                    \end{bmatrix} = u^{-1}\begin{bmatrix}
                        0 & a^{(2)}b'_{i}\\
                        a^{(2)}c'_{i} & 0
                    \end{bmatrix} = 0
                \]
                as claimed.
            \item
                $j = \nu(c^{(2)})$. Analogously as above, we divide every
                $x_i$ to obtain a remainder with $c'_i = 0$ and then
                use~\eqref{eq:commutator} to conclude that $c^{(2)}b'_i = 0$ in
                $\kk[t]/(t^m)$ hence also $t^jb'_i = 0$.
            \item
                $j = \nu(b^{(2)}) - \Delta$. Then $\nu(b^{(2)}) = j + \Delta$. Analogously as above, we divide every
                $x_i$ to obtain a remainder $r_i$ with $b'_i = 0$ and
                use~\eqref{eq:commutator} to conclude that $b^{(2)}c'_i = 0$
                in $\kk[t]/(t^m)$. This implies that $c'_i$ is divisible by
                $t^{m - \nu(b^{(2)})} = t^{m - j - \Delta} = t^{k-j}$, so
                $t^{j}\cdot c'_i = 0$ in $\kk[t]/(t^k)$ and hence the only
                nonzero entry of $t^{j}\cdot
                r_i$ is in the top left corner.
        \end{enumerate}
        Moreover, $a_i', c_i'\in t^j\kk[t]$ since the corresponding entries of
        both $x_i$ and $x_i'$ lie there; same for $b_i'\in t^{j+\Delta}\kk[t]$.
        We replace $x_i$ by $r_i$, by Corollary~\ref{algebra} this does not change
        the components containing our tuple. We keep the notation that
        \[
            x_i =
            \begin{bmatrix}
                a^{(i)} & b^{(i)}\\
                c^{(i)} & d^{(i)}
            \end{bmatrix}.
        \]
        Consider the matrix $A = \begin{bmatrix}
            t^j & 0\\
            0 & 0\\
        \end{bmatrix}$.
        Clearly, $A$ is an endomorphism of the $\kk[t]$-module $V$, so the
        underlying matrix in $\MM_{d}$ commutes with $x_1$.
        From the above considerations we see that $A\cdot x_i = x_i\cdot A =
        \begin{bmatrix}
            t^ja^{(i)} & 0\\
            0 & 0
        \end{bmatrix}$ for $i\geq 3$.
        For each $\lambda\in \kk$ define
        $X_2(\lambda) := \lambda A + x_2$. For each $\lambda\in\kk$, this
        matrix commutes with $x_1, x_3, x_4, \ldots ,x_n$.
        For all but one choices of $\lambda\in \kk$ we have $\nu(\lambda t^j +
        a^{(2)}) = j$. Since we want to prove that our starting tuple is a
        limit, we may replace $x_2$ by suitable $X_2(\lambda)$ and hence
        we can assume that $\nu(a^{(2)}) = j$ and that $x_i$ is as in the case (1) above for $i\ge 3$.

        Now we put $x_2$ in a normal form by using the automorphisms of the
        module $V$.
        For every $\alpha\in \kk[t]$, $\beta\in t^{\Delta}\kk[t]$
        such that $1-\alpha\beta\in \kk[t]/(t^m)$ is an invertible element (which is
        always the case if $\Delta > 0$) we have an automorphism $\Phi(\alpha,
        \beta)$ of
        the module $V$ given by the matrix
        \[
            \begin{bmatrix}
                1 & \beta\\
                \alpha & 1
            \end{bmatrix}
        \]
        whose inverse is the matrix $\frac{1}{1-\alpha\beta}\begin{bmatrix}
            1 & -\beta\\
            -\alpha & 1
        \end{bmatrix}$.
        The coordinate
        change $\Phi(\alpha, \beta)$ maps $x_1$ to itself and each $x_i$, for $i\geq 2$ to
        \[
            \frac{1}{1-\alpha\beta}\begin{bmatrix}
                a^{(i)} - \alpha b^{(i)}+ \beta c^{(i)} &  - \beta a^{(i)} + b^{(i)} - \beta ^2c^{(i)}\\
                \alpha a^{(i)} -\alpha ^2b^{(i)}+ c^{(i)} & -\alpha \beta a^{(i)} + \alpha b^{(i)} - \beta c^{(i)}
            \end{bmatrix}.
        \]
        Now, since $\nu(a^{(2)}) = j \leq \nu(b^{(2)}), \nu(c^{(2)})$
        and $b^{(2)}\in t^{j+1}\kk[t]$ (also for $\Delta=0$), we can
        solve the equation $\alpha a^{(2)} -\alpha ^2b^{(2)}+ c^{(2)} = 0$ in
        $\alpha\in \kk[t]/(t^m)$, hence after a transformation $\Phi(\alpha, 0)$
        we have $c^{(2)} = 0$. The change of coordinates might have not preserved the conditions
        $a^{(i)} = 0$ for $i\ge 3$ and $d^{(i)} = 0$ for $i\geq 2$, but it did preserve the
        conditions $a^{(i)}\in t^j\kk[t]/(t^m)$, $b^{(i)}\in
        t^{j+\Delta}\kk[t]/(t^m)$, $c^{(i)}, d^{(i)}\in t^j\kk[t]/(t^k)$ and $\nu (a^{(2)})=j$.  We now subtract an appropriate
        element of $\kk[x_1]$ from $x_2$ and appropriate elements of $\kk[x_1, x_2]$ from each $x_i$, $i\ge 3$, to guarantee
       $a^{(i)} =0$ for $i\ge 3$ and $d^{(i)} = 0$ for $i\ge 2$.
        Now, $\nu(b^{(2)}) - \Delta \geq
        \nu(a^{(2)})$, so we may also solve the equation $-\beta a^{(2)} +
        b^{(2)} = 0$ in $\beta\in t^{\Delta}\kk[t]/(t^m)$. Applying $\Phi(0,
        \beta)$ we get $b^{(2)} = 0$. Finally we rescale $x_2$ by an
        invertible element $u\in \kk[t]/(t^m)$ such that $ua^{(2)} = t^j$ so
        that
        \[
            x_2 = \begin{bmatrix}
                t^{j} & 0\\
                0 & 0
            \end{bmatrix}.
        \]
        Again, we subtract appropriate elements of $\kk[x_1,x_2]$ from each $x_i$ to get $a^{(i)}=d^{(i)}=0$ for each $i\ge 3$.
        This concludes the proof.
    \end{proof}

    \begin{proposition}[Two Jordan blocks]\label{2-dim_kernel}
        Let $d>4$ be arbitrary and let $(x_1,\ldots ,x_n)\in \CommMat$ be an
        $n$-tuple of nilpotent matrices such that some linear combination of
        $x_1,\ldots ,x_n$ has a two-dimensional kernel. Then the $n$-tuple
        $(x_1,\ldots ,x_n)$ belongs to a non-elementary component of $\CommMat$.
    \end{proposition}

    \begin{proof}
        Using Proposition~\ref{2-dim_kernel-structure} and closedness of
        non-elementary locus, we assume that
        $(x_1, \ldots ,x_n)$ is in the form~\eqref{eq:2blocks_canonical_form}.
        We consider three cases:
        \begin{enumerate}
            \item $j \geq \frac{k+1}{2}$.
                    By adding appropriate powers of $x_1$ to
                    $x_2$, we may assume that $a^{(i)} = 0$ for $i\geq 2$.
                    (This operation makes $d^{(2)}$ nonzero.)
                    Since every block entry of every $x_i$ with $i\geq 2$ is
                    divisible by $t^j$, the matrix
                    $x_i$ has zeros in the $\lceil \frac{k}{2}\rceil$-th row and
                    column.
                    In this case, the matrix unit $E_{\lceil
                    \frac{k}{2}\rceil,\lceil \frac{k}{2}\rceil}$ commutes with $x_i$
                    for each $i\geq 2$. The line $\{(x_1+\lambda E_{\lceil
                        \frac{k}{2}\rceil,\lceil \frac{k}{2}\rceil},x_2,\ldots
                    ,x_n):\lambda \in \kk\}$ intersects in an open subset a component of $\CommMat$
                    that contains $n$-tuples of matrices with more than one
                    eigenvalue, so this component must contain the whole line, and in
                    particular it contains $(x_1,\ldots ,x_n)$.
                \item $j < \frac{k}{2}$.
                    Since $t^j$ annihilates the left-lower and right-upper
                    block entries of all $x_i$ with $i\geq 2$, we see that
                    each of these matrices has zeros in the $(m+\lceil
                    \frac{k}{2}\rceil)$-th row and
                    column.
                    In this case the matrix unit $E_{\lceil
                        m+\frac{k}{2}\rceil,\lceil m+\frac{k}{2}\rceil}$
                        commutes with $x_i$ for each $i\geq 2$ and as in the
                        previous case, the tuple $(x_1,\ldots ,x_n)$ lies on
                        a non-elementary component.
                \item $j = \frac{k}{2}$, $k$ even. In this case every two
                    matrices in the form~\eqref{eq:2blocks_canonical_form} commute by
                    equation~\eqref{eq:commutator}. So we have an affine space
                    of tuples of commuting matrices and it is
                    enough to prove that a \emph{general} element of this
                    space lies on a non-elementary component.
                    We thus assume that the space spanned by
                    $\kk[x_1]x_2$, $\kk[x_1]x_3$, $\kk[x_1]x_4$
                    is equal to the space spanned by
                    \[
                        \kk[t]\cdot\begin{bmatrix}
                            t^j & 0\\
                            0 & 0
                        \end{bmatrix},\ 
                        \kk[t]\cdot\begin{bmatrix}
                            0 & t^{j+\Delta}\\
                            0 & 0
                        \end{bmatrix},\ 
                        \kk[t]\cdot\begin{bmatrix}
                            0 & 0\\
                            t^{j} & 0
                        \end{bmatrix}.
                    \]
                    Then every other matrix $x_i$ is an element of $\langle
                    x_2, x_3, x_4\rangle \kk[x_1]$, so by
                    Corollary~\ref{algebra} we assume $x_i = 0$ for $i\geq 5$.
                    Finally, we replace $x_2$ by $x_1^j-x_2$ and get
                    \[
                        x_2 = \begin{bmatrix}
                            0 & 0\\
                            0 & t^j
                        \end{bmatrix}.
                    \]
                    If $m > k = 2j$, then all matrices $x_i$ for $i\geq 2$ have
                    $(j+1)$-th row
                    and column zero, so $E_{j+1,j+1}$ commutes with them and
                    the tuple lies on a non-elementary component by the same
                    argument as in previous cases.
                    If $m = k$, then the argument is slightly more
                    complicated: the matrix
                    $E := E_{1,1}+E_{j+1,j+1}+E_{m+1,m+1}+E_{m+j+1,m+j+1}$
                    commutes with $x_i$ for all $i\geq 2$. If $d > 4$, then $m
                    > 2$ so $E$ is not the identity matrix and we conclude as
                    before.\qedhere
        \end{enumerate}
    \end{proof}

    \begin{remark}
        If $d=4$, then there are two possible Jordan structures for a matrix
        with 2-dimensional kernel: $(3,1)$ and $(2,2)$. In the first case the
        proposition still holds, by Lemma~\ref{big_largest_Jordan}. On the other hand, in the second case the
        proposition is false, as will be shown in
        Example~\ref{ex:squareZero4x4}.
    \end{remark}

    \section{\BBname{} decompositions and components of $\Quotmain$}\label{sec:BBdecomposition}

    \newcommand{\EL}{\mathcal{E}{le}}%
    In this section, we assume $\chr \kk = 0$ for technical reasons;
    see~\cite{Jelisiejew__Elementary} for details.
    Oversimplifying, the \BBname{} on $\Quotmain$ works as
    follows: consider the locus
    \[
        \EL = \left\{ [F/K]\ |\ F/K\mbox{ is supported on a single point of
        }\mathbb{A}^{\ambM}\right\} \subset \Quotmain.
    \]
    The \BBname{} decomposition of $\Quotmain$ is a certain subdivision of
    $\EL$ into loci $\EL_1, \ldots , \EL_{s}$.  For a point $[F/K]\in \EL$
    there exists a simple linear algebraic condition: having
    \emph{trivial negative tangents}~\eqref{eq:TNTcondition}, which
    implies that $\EL_i\to \Quotmain$ is an open immersion near $[F/K]$, see Proposition~\ref{ref:TNTandElementary:prop}.
    In that case $[F/K]$ is a point of an elementary component of $\Quotmain$.
    So the basic takeaway from
    this section could be that
    \begin{center}
        \emph{To find an elementary component of $\Quotmain$ it is enough to prove
        that a given point $[F/K]$ satisfies the
        condition~\eqref{eq:TNTcondition}.}
    \end{center}
    Needless to say, there are follow-up questions, for example about
    smoothness of $[F/K]$. We will answer some of them below.
    The content of this section parallels the material on Hilbert schemes
    in~\cite{Jelisiejew__Elementary}; we refer the reader there for details.
    \smallskip

    Let $\Gmult = \Spec(\kk[t^{\pm 1}])$ be the one-dimensional torus.
    To even speak about the \BBname{} decomposition, we need a
    $\Gmult$-action on $\Quotmain$, which we now introduce:
    \begin{itemize}
        \item First, we fix a positive $\Gmult$-action on $S$, i.e., we
            fix $\deg(x_i) \in \mathbb{Z}_{>0}$ for $i=1, \ldots ,\ambM$.
        \item Second, we \emph{linearize} the free module $F$. This amounts
            to fixing the degrees on generators $e_1$, \ldots
            , $e_{\genM}$ of the $S$-module $F$, so that $F = S(-\deg e_1) \oplus
            S(-\deg e_2) \oplus  \ldots \oplus S(-\deg e_{\genM})$. Here we
            put no positivity assumptions on the degrees.
    \end{itemize}
    \begin{remark}
        In applications below use only the most natural action, where
        $\deg(x_i) = 1$ for all $i$ and $\deg e_j = 0$ for all $j$.
    \end{remark}
    The above choice of degrees of generators of $F$ determines an action of
    $\Gmult$ on $F$, namely $t\circ f := t^{-\deg f} f$ for a
    homogeneous element $f\in F$ and $t\in \Gmult(\kk)$.
    \begin{example}
        If $S = \kk[x_1, x_2]$ with $\deg x_1 = \deg x_2 = 1$ and $F =
        S(-1)\oplus S$, then its element $e_1 + x_2 e_2$ is homogeneous of
        degree one, hence $t\circ (e_1 + x_2 e_2) = t^{-1}(e_1 + x_2e_2)$.
    \end{example}
    The action on $F$ induces an action on $\Quotmain$, we just have $t\circ
    (F/K) = F/(t\circ K)$, where $t\circ K = \left\{ t\circ k\ |\ k\in K
    \right\}$.
    Let $\Gbar = \Spec\kk[t^{-1}] = \Gmult \cup \left\{ \infty \right\}$.
    The \emph{(negative) \BBname{} decomposition} of $\Quotmain$ is uniquely
    determined by the $\Gmult$-action. Specifically, for any $\kk$-algebra
    $A$, the $A$-points of this decompositions
    are
    \[
        \Quotplus(A) = \left\{ \varphi\colon \Gbar \times \Spec(A) \to
            \Quotmain\ |\ \varphi \mbox{ is $\Gmult$-equivariant} \right\}.
    \]
    The superscript ``$+$'' instead of the more natural ``$-$'' is introduced
    for consistency
    with~\cite{Jelisiejew__Elementary}.
    The above formula describes a functor $\Quotplus$. This functor is
    represented by a scheme by~\cite[Prop~4.5, 5.3]{jelisiejew_sienkiewicz__BB} since
    $\Quotmain$ is covered by $\Gmult$-invariant affine open subschemes given
    in Example~\ref{ex:monomialBasis}. Below we denote
    by $\Quotplus$ both the functor and the representing scheme.
    The $\kk$-points of $\Quotplus$ by definition correspond to
    $\Gmult$-equivariant maps $\varphi\colon \Gbar\to \Quotmain$. Each such
    map can be restricted to $\varphi^{\circ}\colon \Gmult\to \Quotmain$. But
    a $\Gmult$-equivariant map from $\Gmult$ is just the orbit map of
    $\varphi^{\circ}(1) = \varphi(1)$! From this point of view,
    $\varphi(\infty)$ is the limit of the orbit of $\varphi(1)$ and informally
    speaking $\Quotplus$ parameterizes points of $\Quotmain$ together with
    the limits at infinity of their $\Gmult$-orbits.

    The mapping $\varphi\mapsto \varphi(1, -)$ gives a natural forgetful
    morphism
    \newcommand{\thetazero}{\theta_0}%
    \[
        \thetazero\colon \Quotplus \to \Quotmain
    \]
    which is universally injective, or, in other words, injective on
    $L$-points for every field $L$. We identify the $\kk$-points of $\Quotplus$
    with their images in $\Quotmain$. Below we check which
    $\kk$-points of $\Quotmain$ are obtained in this way.
    \begin{lemma}\label{ref:positiveness}
        The $\kk$-points of $\Quotplus$ are exactly the $\kk$-points of
        $\Quotmain$ which correspond to modules $M$ supported only at the
        origin of $\mathbb{A}^n$.
    \end{lemma}
    \begin{proof}
        The argument is very similar
        to~\cite[Proposition~3.3]{Jelisiejew__Elementary}.

        Let $M$ correspond to a point of $\Quotplus$. This means that the
        $\Gmult$-orbit of $[M]$ extends to a morphism $\varphi\colon\Gbar \to \Quotmain$.
        Pick $i\in \{1,2, \ldots ,n\}$ and $\alpha\in \kk$ and suppose that
        $\Supp M$ intersects the hyperplane $V(x_i - \alpha)$. Then the support
        of $t\circ M$ intersects the hyperplane $V(t\circ(x_i - \alpha)) =
        V(x_i - t^{\deg x_i} \alpha)$. If $\alpha\neq 0$ this means that
        $\Supp(t\circ M)$ is divergent, a contradiction since it converges to
        the support of $\varphi(\infty)\in \Quotmain$. This shows that $\Supp M =
        \{0\}$.
        Conversely, suppose that $M$ is supported at zero. We may view $[M]$ as a
        point on the $\OpQuot$ scheme of $\OO^{\oplus r}_{\mathbb{P}^n}$. Such $\OpQuot$ is
        projective, hence every $\Gmult$ orbit extends, moreover the limit is
        still supported at zero so it lies in $\Quotmain$.
    \end{proof}
    Consider a module $M = F/K$ supported at the origin. We say that $M$ has \emph{trivial negative tangents} if
    \begin{equation}\label{eq:TNTcondition}
        \dim_{\kk}\left(\Hom(K, M)/\Hom(K, M)_{\geq 0}\right) = n.
    \end{equation}
    Recall from~\S\ref{ssec:components} that an irreducible component of $\Quotmain$ is \emph{elementary}
    if its geometric points correspond to modules supported only at one
    point. Let
    \[
        \theta\colon \mathbb{A}^{\ambM} \times \Quotplus \to\Quotmain
    \]
    be the morphism defined on points by $\theta(w, [F/K]) = [F/K] + w$. In other words,
    $\theta(w, [F/K])$ is the module $F/K$ translated by vector $w$.

    \begin{proposition}\label{ref:TNTandElementary:prop}
        If $M$ has trivial negative tangents, then $\theta\colon
        \mathbb{A}^{\ambM}\times \Quotplus\to \Quotmain$ is an open immersion
        near $[M]$.
        Conversely, if $\compo \subset \Quotmain$ is a generically reduced
        elementary component, then a general point of $\compo$ has trivial
        negative tangents.
    \end{proposition}
    \begin{proof}
        This follows similarly to~\cite[Theorem~4.5,
        Theorem~4.9]{Jelisiejew__Elementary}.
    \end{proof}

        While we do not employ it significantly in the current article,
        Proposition~\ref{ref:TNTandElementary:prop} is very useful to describe
        new elementary components of Quot schemes, see the introduction
        of~\cite{Jelisiejew__Elementary} for the case of Hilbert schemes.
        Also, all examples in Section~\ref{ssec:examples} below have trivial
        negative tangents.

    With a bit of deformation theory, we can even check that a given point
    with trivial negative tangents is smooth.
    \begin{lemma}\label{ref:obstruction:lem}
        The tangent space to $\Quotplus$ at a point $[M = F/K]$ is equal to
        $\Hom_S(K, M)_{\geq 0}$. Moreover, the point $[M = F/K]\in \Quotplus$ has an
        obstruction theory with obstruction space $\Ext^1_S(K, M)_{\geq 0}$.
    \end{lemma}
    \begin{proof}
        This follows exactly as
        in~\cite[Theorem~4.2]{Jelisiejew__Elementary}.
    \end{proof}

    \begin{theorem}
        Let $M = F/K$ be a module of finite degree supported at the origin and such that $\Ext^1_S(K, M)_{\geq 0} = 0$. Then
        $[M]\in \Quotplus$ is a smooth point. If moreover $M$ has trivial
        negative tangents, then $[M]\in \Quotmain$ is a smooth point on an
        elementary component.
    \end{theorem}
    \begin{proof}
        Follows by combining Lemma~\ref{ref:obstruction:lem}, Example~\ref{example:obs}, and
        Proposition~\ref{ref:TNTandElementary:prop}.
    \end{proof}

    Later we will also need the \emph{positive} \BBname{} decomposition that we
    describe below.
    Let $\Gbar' = \Spec\kk[t] = \Gmult \cup \left\{ 0 \right\}$.
    For any $\kk$-algebra
    $A$, the $A$-points of this decomposition
    are
    \[
        \Quotminus(A) = \left\{ \varphi\colon \Gbar' \times \Spec(A) \to
            \Quotmain\ |\ \varphi \mbox{ is $\Gmult$-equivariant} \right\}.
    \]
    This decomposition too is represented by a scheme, which we denote
    $\Quotminus$ and its point $M = F/K$ has tangent space $\Hom(K, M)_{\leq 0}$ and an obstruction
    theory with obstruction space
    $\Ext^1(K, M)_{\leq 0}$.
    \section{Results specific for degree at most eight}\label{sec:degeight}

    Throughout this section we assume $\chr \kk = 0$. Some of the arguments
    do not use this assumptions, most others can be made for large enough
    characteristics. However, characteristic zero is indispensable for proofs
    of surjectivity of tangent
    maps and for all arguments performed with the help of
    \emph{Macaulay2}.

    \subsection{Examples of elementary components}\label{ssec:examples}

        In this subsection we gather examples of elementary components in
        $\CommMat$. For compactness we use the tensor notation: a tuple
        $(x_1, \ldots ,x_n)\in \CommMat$ is written as $\sum x_i\cdot e_i$
        where $e_1, \ldots ,e_n$ are formal coordinates. For example, the
        triple
        \[
            \begin{bmatrix}
                1 & 0\\
                0 & 1
            \end{bmatrix}, \begin{bmatrix}
                1 & 1\\
                0 & 1
            \end{bmatrix}, \begin{bmatrix}
                0 & 1\\
                0 & 0
            \end{bmatrix}
            \mbox{ is presented as a single matrix }
            \begin{bmatrix}
                e_1 + e_2 & e_2 + e_3\\
                0 & e_1 + e_2
            \end{bmatrix}.
        \]
        For any $n$, $d$ and $m\in \{1,\ldots ,d-1\}$  define
        \[
            \squareZeroComp{m}{d-m}{n}=\left\{g\cdot\left(\left[
                \begin{array}{cc}0&A_1\\0&0
                \end{array}
            \right],\ldots ,\left[
                \begin{array}{cc}0&A_n\\0&0
                \end{array}
            \right]\right)\cdot g^{-1};A_1,\ldots ,A_n\in \MM_{m\times (d-m)},
        g\in \GL_d\right\}
            \subset \CommMat.
        \]
        For a general tuple in $\squareZeroComp{m}{d-m}{n}$, the set of $g\in \GL_d$ that conjugate this
        tuple into a block upper triangular tuple
        has codimension $m(d-m)$ in $\GL_d$.
        Therefore, $\dim  \squareZeroComp{m}{d-m}{n}  =
        (n+1)m(d-m)$. Adding some multiple of identity matrices to each matrix,
        we get a locus of dimension $(n+1)m(d-m)+n$, that we call the
        $m$-th \emph{square-zero locus} of $\CommMat$.
        Below we say that a tuple $\xx$ \emph{witnesses} that the square-zero locus
        is a component if the tangent space $T_{\xx} \CommMat$ has dimension
        $(n+1)m(d-m)+n$; equal to the dimension of the square-zero locus.

        \begin{example}[$4\times 4$ square-zero
            quadruple, $m=2$]\label{ex:squareZero4x4}
            The quadruple
            \[
                \begin{bmatrix}
                    0 & 0 & e_1 & e_2\\
                    0 & 0 & e_3 & e_4\\
                    0 & 0 & 0 & 0\\
                    0 & 0 & 0 & 0
                \end{bmatrix}
            \]
            witnesses that the square-zero locus is a component for $(n,d,m) =
            (4,4,2)$. This tuple
            corresponds to a module $M$ with Hilbert function $(2, 2)$ that
            is generated by two elements and has trivial negative tangents.
            It was already known in~\cite[p.~72]{Gur} that this tuple
            does not lie on the principal component, while seemingly it was
            unknown which other components \rred{it lies} on.
        \end{example}
        \begin{example}[$5\times 5$ square zero
            quintuple, $m=2$]\label{ex:squareZero5x5}
            The tuple
            \[
                \begin{bmatrix}0& 0& {e}_{1}& {e}_{2}& {e}_{3}\\ 0& 0& {e}_{4}&
                    {e}_{1}+{e}_{5}& {e}_{2}\\ 0& 0& 0& 0& 0\\ 0& 0& 0& 0& 0\\ 0&
                    0& 0& 0& 0\\ \end{bmatrix}
            \]
            witnesses that the square-zero locus is a component for $(n,d,m) =
            (5,5,2)$. This tuple
            corresponds to a module with Hilbert function $(3, 2)$ that has
            trivial negative tangents and is a smooth point of an elementary
            $31$-dimensional component of Quot.
        \end{example}
        \begin{example}[$6\times 6$ square zero
            sextuple, $m=3$]\label{ex:squareZero6x6m3}
            The tuple $x_1=E_{14}+E_{25}+E_{36}, x_2=E_{15}+E_{26},
            x_3=E_{16}, x_4=E_{24}+E_{35}, x_5=E_{25}+E_{36}, x_6=E_{34}$
            witnesses that the square-zero locus is a component for $(n,d,m) =
            (6,6,3)$. This tuple
            corresponds to a module with Hilbert function $(3,3)$ that has
            trivial negative tangents and is a smooth point of  an elementary $51$-dimensional component of Quot.
        \end{example}
        \begin{example}[$6\times 6$ square zero
            sextuple II, $m=2$]\label{ex:squareZero6x6m2}
            The tuple $x_1=E_{13}+E_{24}, x_2=E_{14}+E_{25}, x_3=E_{15}+E_{26}, x_4=E_{16}, x_5=E_{23}, x_6=E_{24}$
            witnesses that the square-zero locus is a component for $(n,d,m) =
            (6,6,2)$. This tuple
            corresponds to a module with Hilbert function $(4,2)$ that has
            trivial negative tangents and is a smooth point of an elementary $50$-dimensional component of Quot.
        \end{example}
        \begin{example}[$7\times 7$ square zero
            quintuple, $m=3$]\label{ex:squareZero7x7m3}
            The tuple $x_1=E_{14}+E_{25}+E_{36}, x_2=E_{15}+E_{26}+E_{37}, x_3=E_{16}+E_{27}, x_4=E_{24}+E_{35}, x_5=E_{26}+E_{37}$
            witnesses that the square-zero locus is a component for $(n,d,m) =
            (5,7,3)$. This tuple
            corresponds to a module with Hilbert function $(4,3)$ that has
            trivial negative tangents and is a smooth point of an elementary $56$-dimensional component of Quot.
        \end{example}
        \begin{example}[$7\times 7$ square zero
            septuple, $m=2$]\label{ex:squareZero7x7m2}
            The tuple $x_1=E_{13}+E_{24}, x_2=E_{14}+E_{25},
            x_3=E_{15}+E_{26}, x_4=E_{16}+E_{27}, x_5=E_{17}, x_6=E_{23},
            x_7=E_{24}$
            witnesses that the square-zero locus is a component for $(n,d,m) =
            (7,7,2)$. This tuple
            corresponds to a module with Hilbert function $(5,2)$
            that has
            trivial negative tangents and is a smooth point of an elementary
            $73$-dimensional component of Quot.
        \end{example}

        \newcommand{\cubeZeroCompII}[5]{\mathcal{W}_{#1,#2,#3;#4}^{#5}}
        \newcommand{\cubeZeroComp}[4]{\mathcal{W}_{#1,#2,#3}^{#4}}
        \newcommand{\cubeZeroCompIIstandard}{\cubeZeroCompII{a}{b}{d-a-b}{c}{n}}
        \begin{example}[$7\times 7$ non-square zero quintuples]\label{ex:332}
            This is the only example that does not parameterize square-zero
            matrices.  Consider the locus of quintuples of $7\times 7$
            matrices that have the form
            \begin{equation}\label{eq:ex322matrices}
                x_i = \begin{bmatrix}
                    \mu_i & 0 & 0 & \lambda_{1i}u_{11}+\lambda_{2i}u_{21} & \lambda_{1i}u_{21}+\lambda_{2i}u_{31} & * & *\\
                    0 & \mu_i & 0 & \lambda_{1i}u_{12}+\lambda_{2i}u_{22} & \lambda_{1i}u_{22}+\lambda_{2i}u_{32} & * & *\\
                    0 & 0 & \mu_i & \lambda_{1i}u_{13}+\lambda_{2i}u_{23} & \lambda_{1i}u_{23}+\lambda_{2i}u_{33} & * & *\\
                    0 & 0 & 0 & \mu_i & 0 & 0 & \lambda_{1i}\\
                    0 & 0 & 0 & 0 & \mu_i & 0 & \lambda_{2i}\\
                    0 & 0 & 0 & 0 & 0 & \mu_i & 0\\
                    0 & 0 & 0 & 0 & 0 & 0 & \mu_i
                \end{bmatrix}\qquad i=1,2, \ldots 5,
            \end{equation}
            where we take arbitrary $\mu_i, \lambda_{1i}, \lambda_{2i}, u_{jk}\in \kk$ for
            $i=1,2, \ldots ,5$ and $j,k = 1,2,3$ and stars
            denote arbitrary entries. The matrices in each quintuple
            commute. There are $5\cdot (3+6)+ 9 = 54$
            parameters, so we obtain a morphism $\mathbb{A}^{54}\to C_5(\MM_7)$.
            For $[\lambda_{1i}\ \lambda_{2i}]_{i=1}^2$
            linearly independent we can recover the $u_{jk}$ from the
            matrices, so this map is generically one-to-one and so its
            image is a rational locus $\mathcal{L}$ of dimension $54$.
            Consider the locus $\GL_7\cdot \mathcal{L}$. To
            obtain its dimension we
            pick a
            general point $\xx = (x_1, \ldots ,x_5)\in \mathcal{L}$ and
            compute the dimension of $G := \left\{ g\in \GL_7 |\ g\xx g^{-1}\in \mathcal{L}
        \right\}$. This
            subgroup does not change when we subtract identity matrices,
            so we assume $\mu_i = 0$ for $i=1,2, \ldots ,5$.
            We have $\bigcap_i\ker(x_i) = (*, *, *, 0, 0, 0, 0)$,
            $(\bigcap_i\ker(x_i)) + \sum\im(x_i) = (*, *, *, *, *, 0, 0)$,
            $\bigcap_{i,j} \ker(x_i\cdot x_j) = (*, *, *, *, *, *, 0)$ so
            any element of $G$ stabilizes those
            spaces. Therefore, $G\subset \GL_7$ has codimension at least
            $17$ and so $\dim(\GL_7\cdot \mathcal{L})\geq 71$.
            The tangent space to $\CommMat$ at the quintuple
            \[
                \begin{bmatrix}0& 0& 0& {e}_{1}& 0& {e}_{3}& 0\\
                        0& 0& 0& {e}_{2}& {e}_{1}& {e}_{4}& 0\\
                        0& 0& 0& 0& {e}_{2}& {e}_{5}& 0\\
                        0& 0& 0& 0& 0& 0& {e}_{1}\\
                        0& 0& 0& 0& 0& 0& {e}_{2}\\
                        0& 0& 0& 0& 0& 0& 0\\
                        0& 0& 0& 0& 0& 0& 0\\ \end{bmatrix}
            \]
            is $71$-dimensional, so indeed we obtain a component.
            The tuple above corresponds to an $\kk[y_1, \ldots
            ,y_5]$-module $M$ with Hilbert function $(2, 2, 3)$ that is
            generated by two elements.
            In the language of Section~\ref{ssec:cubeZero} below, this
            component is a part of $\cubeZeroCompIIstandard$ for $(a,b,c) =
            (3,3,2)$.
            Moreover, if $\mu_i = 0$ then $x_i^2$ has rank at most one, which
            implies that the associated module does not have the strong
            Lefschetz property. Note that this happens for any graded
            module in the open locus of this component.
        \end{example}

        \begin{remark}\label{ref:addingMatrices:rmk}
            All of the above components give rise to elementary components of
            $\CommMat$ for $n$ greater than in the examples. More precisely,
            consider a witness point $\xx\in \CommMat$ as in each of the
            examples above; this point is smooth. Consider the point
            $\xx'\in C_{n+e}(\MM_d)$ obtained by padding the tuple $\xx$ with $e$
            zero matrices. We claim that $\xx'$ is smooth as well. By
            Lemma~\ref{ref:tangentspace:lem} the difference
            $\dim T_{\xx'}C_{n+e}(\MM_d) - \dim T_{\xx}\CommMat$ is equal to $e\cdot
            \dim C$, where $C\subset \MM_d$ is the space of matrices commuting
            with every matrix in $\xx$. Now, in each case a direct check shows
            that this commutator is as small as possible: in the square-zero
            cases it is given by square-zero matrices (and scalar matrices)
            while in the case of Example~\ref{ex:332} it is given by matrices
            of the shape~\eqref{eq:ex322matrices} with the same $(u_{ij})$ as
            in $\xx$. Knowing this, we see that the tangent space dimension is equal to
            the dimension of the locus defined analogously as in \rred{the} examples,
            thus $\xx'$ is smooth.
        \end{remark}\goodbreak

    \subsection{Cube nonzero cases}
    \begin{proposition}\label{cube_not_zero}
        Let $d\le 7$ and let $(x_1,\ldots ,x_n)\in \CommMat$ be a tuple of
        nilpotent matrices such that some linear combination \rred{of them} has nonzero
        cube. Then the $n$-tuple belongs to a non-elementary component of
        $\CommMat$.
    \end{proposition}
    \begin{proof}
        Using the action of $\GL_n$ we assume that $x_1^3\ne 0$. We put $x_1$ in a Jordan form.
        Let $a\geq a'$ be the sizes of the two largest Jordan blocks of $x_1$,
        so $a\geq 4$.
        If the kernel of $x_1$ is at most two-dimensional,
        we conclude using Proposition~\ref{2-dim_kernel} or
        Lemma~\ref{1-dim_kernel}. Otherwise, $x_1$ has at least three
        Jordan blocks, so $a+a'\leq 6$. If $a\geq 2a'+1$, then we conclude
        using Lemma~\ref{big_largest_Jordan}. If $a\leq 2a'$ then $a = 4$, $a'
        = 2$, so the Jordan type of $x_1$ is $(4, 2, 1)$.
        Using Lemma~\ref{ref:toeplitz:lem} and Remark~\ref{ref:firstBlockZero:rmk}
        we put the matrices $x_i$ for $i\geq 2$ in the form
        \[
            x_i=\left[
                \begin{array}{ccccccc}
                    0 & 0 & 0 & 0 & b_0^{(i)} & b_1^{(i)} & c^{(i)}\\
                    0 & 0 & 0 & 0 & 0 & b_0^{(i)} & 0\\
                    0 & 0 & 0 & 0 & 0 & 0 & 0\\
                    0 & 0 & 0 & 0 & 0 & 0 & 0\\
                    0 & 0 & d_0^{(i)} & d_1^{(i)} & e_0^{(i)} & e_1^{(i)} & f^{(i)}\\
                    0 & 0 & 0 & d_0^{(i)} & 0 & e_0^{(i)} & 0\\
                    0 & 0 & 0 & g^{(i)} & 0 & h^{(i)} & k^{(i)}
                \end{array}
            \right].
        \]
        Since $x_i$ are nilpotent, we have $k^{(i)} = 0 = e_0^{(i)}$ for every
        $i\geq 2$.
        The $(1,3)$-entry of the commutator of $x_i$ and $x_j$ is $b_0^{(i)}d_0^{(j)} -
        b_0^{(j)}d_0^{(i)}$, so any two pairs $(b_0^{(i)},d_0^{(i)})$ and
        $(b_0^{(j)},d_0^{(j)})$ with $2\le i<j$ are linearly dependent. We
        assume by Corollary~\ref{algebra} that $b_0^{(i)}=0$ and $d_0^{(i)}=0$ for
        each $i\ge 3$. Let $y$ be the matrix with the $(3,6)$-th entry equal
        to $-b_0^{(2)}$ and all other entries zero.
        The tuple $(x_1+\lambda E_{22},x_2+\lambda y,x_3,\ldots ,x_n)$
        commutes for each $\lambda \in \Bbbk$. For $\lambda \ne 0$ the matrix
        $x_1+\lambda E_{22}$ has two distinct eigenvalues, so such an
        $n$-tuple belongs to a non-elementary component of $\CommMat$.
        Since components are closed, also $(x_1, \ldots ,x_n)$ belongs to this component.
    \end{proof}

    \subsection{Cube zero, square nonzero cases}\label{ssec:cubeZero}
        Proposition~\ref{cube_not_zero} takes care of the case when
        $x_ix_jx_k\neq 0$ for some $i$, $j$, $k$. In \rred{this} subsection we
        consider the next case: where $x_ix_jx_k = 0$ for all $i,j,k$ but
        there exist some $i$, $j$ such that $x_ix_j\neq 0$.

        Now we decompose the cube-zero locus into subloci.
        For a tuple $\mathbf{x} = (x_1, \ldots ,x_n)$ in $\CommMat$ let
        $K_1(\xx):=\cap_{i=1}^n\ker x_i$ and $K_2(\xx):=\cap _{i,j=1}^n\ker
        x_ix_j$ and $\im(\xx) = \sum_{i=1}^n \im x_i$. From cube-zero we
        get $\im(\xx) \subset K_2(\xx)$. Consider the locus
        \[
            \cubeZeroComp{a}{b}{d-a-b}{n}=\Big\{\xx\in \CommMat\ |\ \dim
                K_1(\xx)=a,\ \dim K_2(\xx)=a+b,\ \forall{i,j,k} : x_ix_jx_k = 0\}.
        \]
        and its subloci, for $1\leq c$, defined by
        \[
            \cubeZeroCompIIstandard =
                \left\{(x_1,\ldots ,x_n)\in\cubeZeroComp{a}{b}{d-a-b}{n}\ |\
            \dim (\im(\xx) + K_1(\xx))=a+c\right\}.
        \]
        Observe that for $c>b$ or $c>n(d-a-b)$ the locus
        $\cubeZeroCompIIstandard$ is empty.
        For any tuple $\xx\in\cubeZeroComp{a}{b}{d-a-b}{n}$, we have $K_1(\xx) \subset K_2(\xx)$ and if we
        choose bases of these spaces compatibly, then the matrices $x_i$ can
        be written as
        \begin{equation}\label{x_i}
            x_i=\left[
                \begin{array}{ccc}0&A_i&B_i\\0&0&C_i\\0&0&0
                \end{array}
            \right]
        \end{equation}
        for some matrices $A_i\in \MM_{a\times b}$, $B_i\in \MM_{a\times (d-a-b)}$
        and $C_i\in \MM_{b\times (d-a-b)}$. By construction, the common kernel
        of $(A_i)_i$ is zero and the common kernel of $(C_i)_i$ is zero.
        The introduction of parameters $a$, $b$, $c$ is arbitrary, but it
        decomposes the cube zero locus into more approachable subloci. Below
        we prove that some of them are irreducible.

        \begin{lemma}\label{dim-cube_zero}
            Let $a,b,c,d$ be positive integers, $d-a-b=1$
            and $c\le b, n$. Then the locally closed locus
            $\cubeZeroCompIIstandard$ is irreducible of dimension
            $n(a(b-c+1)+c)+\frac{ac(c+1)}{2}+ab+bc-c^2+a+b$.
        \end{lemma}
        \begin{proof}
            In essence, this is the same parameter count as in Example~\ref{ex:332}.
            By definition, a tuple $\xx\in \cubeZeroCompIIstandard$ determines
            a flag $K_1(\xx) \subset K_1(\xx)+\im(\xx) \subset K_2(\xx)$ so
            the locus $\cubeZeroCompIIstandard$ is fibered over the
            flag variety of subspaces $V_1\subset V_2\subset V_3\subset
            V \simeq \kk^{d}$ where $(\dim V_1, \dim V_2, \dim V_3) = (a, a+c, a+b)$.
            This flag variety has dimension $ab+bc-c^2+a+b$. It remains to
            compute the fiber, so we assume that $K_1(\xx) = \spann{e_1,
            \ldots ,e_a}$,
            $K_1(\xx)+\im(\xx) = \spann{e_1, \ldots ,e_{a+c}}$, $K_2(\xx) =
            \spann{e_1, \ldots , e_{a+b}}$. In this case, the matrices $x_i$ have
            the form~\eqref{x_i} with $C_i$ having nonzero entries only in the
            first $c$ rows. Up to a linear change of coordinates, which
            amounts to $nc$ parameters, we assume $C_i = e_i$ for $i\leq c$
            and $C_i = 0$ for $i > c$.
            The commutativity condition then reduces to saying that for $1\leq i <
            j\leq c$ the $i$-th column of $A_j$ is the $j$-th column of $A_i$
            and additionally, the first $c$ columns of $A_i$ are zero for $i >
            c$.
            Therefore we have $na(b-c+1)+ac^2$ parameters and $a\binom{c}{2}$
            linear and independent conditions, so be obtain
            an affine space, in particular the whole fiber is irreducible.
            The flag variety is homogeneous under the action of
            $\GL(V)$, the whole locus $\cubeZeroCompIIstandard$ is an image of the
            product of the fiber and $\GL(V)$ so this locus is irreducible.
            The dimension count follows.
        \end{proof}
        The following lemma will be frequently used in the proof of the main theorem of this section. It is a generalization of the case $a_m=3$ of Lemma~\ref{big_largest_Jordan}.
        \begin{lemma}\label{small_image_of_square}
            Let $\xx\in \CommMat$ be a cube-zero tuple such that
            $\dim \sum_{i,j} \im(x_ix_j) = 1$ and $\dim K_2(\xx) = d-1$. Then $\xx$ belongs to a
            non-elementary component of $\CommMat$.
        \end{lemma}
        \begin{proof}
            Let $f\in V \setminus K_2(\xx)$, then $V = K_2(\xx) + \kk f$.
            Let $e$ span $\sum_{i,j} \im(x_ix_j)$. Then every
            $x_ix_j$ is a multiple of the unique rank one matrix that sends $f$ to $e$.
            Replacing matrices with linear combinations, we may assume
            $x_1^2\neq 0$ and then $x_1x_i = 0$ for all $i\geq 2$. Rearrange our fixed basis of $V$ such that the
            first basis element is $e$. In this basis, our tuple
            becomes
            \[
                \left(\left[
                    \begin{array}{ccccc}0&0&A_1'& * & * \\0&0&0& * & *
                        \\0&0&0&0&C_1'\\0&0&0&0&0\\0&0&0&0&0
                    \end{array}
                \right],\ldots ,\left[
                    \begin{array}{ccccc}0&0&A_n'&*&*\\0&0&0& *&*\\0&0&0&0&C_n'\\0&0&0&0&0\\0&0&0&0&0
                    \end{array}
                \right]\right)
            \]
            where $*$ denotes some irrelevant yet possibly nonzero parts.
            Consider the
            matrix $y$ with the middle block
            $C_1'A_1'$ and all other blocks zero. The relations
            $x_1x_i = x_ix_1 = 0$ for all $i\geq 2$ imply that $A_1'C_i' = 0$ and
            $A_i'C_1' = 0$ for
            these $i$ and this proves that the matrix $y$ commutes with $x_i$ for each $i\ge 2$.
            The trace of $y$ is nonzero, so
            for $\lambda \ne 0$ the
            matrix $x_1+\lambda y$ has nonzero trace so it has a nonzero eigenvalue, so the family
            $(x_1+\lambda y,x_2,\ldots ,x_n)$ proves that $\xx$ lies on a non-elementary
            component.
        \end{proof}

        \begin{theorem}\label{cube_zero}
            Let $d\le 7$ and let $a$, $b$, $c$, $n$ be positive
            integers with $c\le b$, $c\leq n(d-a-b)$ and $a+b<d$ and let
                $\mathcal{V} = \mathcal{V}_{a,b,d-a-b;c}^n := (\kk
                    I)^n+\overline{\cubeZeroCompII{a}{b}{d-a-b}{c}{n}}$. Then the following holds:
            \begin{enumerate}[label=(\arabic*)]
                \item\label{it:caseOne}
                    For $(a,b,c,d) = (3,3,2,7)$ and $n\geq 5$ the variety
                    $\mathcal{V}$
                    is an irreducible component of $C_n(\MM_{d})$.
                \item\label{it:caseTwo} For $(a,b,c,d) = (2, 2, 2, 7)$ and
                        $n\geq 5$ the variety $\mathcal{V}$ has a
                        component $\mathcal{Z}$ which is the transpose of the locus from
                        case~\ref{it:caseOne} and $\mathcal{Z}$ is a
                        component of $\CommMat$,
                        while all other components of $\mathcal{V}$ belong to non-elementary
                    components of $\CommMat$.
                \item\label{it:caseNonelementary}
                    In all other cases the variety $\mathcal{V}$
                    belongs to the union of non-elementary components of $\CommMat$.
            \end{enumerate}
        \end{theorem}
        \begin{proof}
            Fix a tuple $\xx\in \cubeZeroCompII{a}{b}{d-a-b}{c}{n}$.
            First we discard several easy cases.
            If $n\leq 2$ then it is classically known that $\CommMat$ is
            irreducible. If $n = 3$ then it is also known that $\CommMat$ is
            irreducible for $d\leq 10$, see~\cite{Han,
            Sivic__Varieties_of_commuting_matrices}.
            If $a = 1$ and $d-a-b=1$ then $\dim
            \sum_{i,j} \im(x_ix_j) = 1$ so the tuple lies on a non-elementary
            component by Lemma~\ref{small_image_of_square}.
            So, after possibly
            transposing the matrices, we assume
            $\mathbf{a\geq2}$.

            Assume now $c=1$.
            We will check that $\dim K_2(\xx) = d-1$ and that $\dim
            \sum_{i,j} \im(x_ix_j)= 1$.
            Since $\im(\xx)\not\subset \ker(\xx)$, up to linear coordinate
            change we assume $\im(x_1) \not \subset \ker(\xx)$. Let $w\in V$
            be such that $x_1(w)\not\in \ker(\xx)$ and denote this element by
            $v$. Since $x_i(w)\in \im(\xx) \subset \kk v + \ker(\xx)$, we have $x_i(w) = \alpha_i v +
            k_i$ where $\alpha_i\in \kk$ and $k_i\in \ker(\xx)$.
            Then $x_i(v) = x_i(x_1(w)) = x_1(x_iw) = x_1(\alpha_i v + k_i) =
            \alpha_i x_1(v)$ so after linear change of
            coordinates we assume that $x_i(v) =
            0$ for $i=2,3, \ldots ,n$. In particular, $\dim \sum_{i,j}
            \im(x_ix_j)= 1$. Since $x_2$, \ldots ,$x_n$
            annihilate $\im(\xx)$ we have that $(x_2, \ldots ,x_n)\cdot
            (x_1, \ldots ,x_n)$ annihilates $V$ and so
            $\bigcap\ker(x_ix_j) = \ker(x_1^2)$ is of codimension one, as
            claimed. By Lemma~\ref{small_image_of_square}, also this tuple
            lies on a non-elementary component.
            In the following we assume $\mathbf{c\geq 2}$.
            Then $\mathbf{b\leq 4}$.
            We will subdivide the remaining cases with respect to $d-a-b$.

            \textbf{Case 1, $\mathbf{d-a-b =1}$.}
            Example~\ref{ex:332} proves the part~\ref{it:caseOne} for $n = 5$.
            To obtain examples with bigger $n$ simply add zero
            matrices, see Remark~\ref{ref:addingMatrices:rmk}. It remains
            to prove that the
            remaining cases are non-elementary.
            The locus
            $\cubeZeroCompIIstandard$ is irreducible by
            Lemma~\ref{dim-cube_zero}.

            Suppose first $\mathbf{b=c}$. We claim that $\cubeZeroCompIIstandard$ contains \emph{cyclic
            tuples}: the tuples $\xx$ for which there exists a vector $v\in V$
            such that $V = \kk[x_1, \ldots ,x_n]\cdot v$. If $a\leq
                \binom{b+1}{2}$ then the claim follows from
                Example~\ref{example:moduleForAB}. If $a >
                \binom{b+1}{2}$ then let $\delta = a - \binom{b+1}{2}$ and
                consider the tuple corresponding to the algebra
                \[
                    \frac{\kk[y_1, \ldots ,y_b,y'_1, \ldots , y'_\delta]}{(y_1,
                        \ldots ,y_b)^3 + (y_1, \ldots ,y_b,y_1', \ldots
                        ,y_{\delta}')(y'_1, \ldots
                ,y'_\delta)}
                \]
                in its monomial basis. In our case $d\leq 7$ so the only option is
                $a = 4$, $b = 2$, then $b+\delta = 3\leq n$ so indeed the above
                algebra
                gives rise to a tuple in $\cubeZeroCompII{a}{b}{d-a-b}{c}{n}$;
                if $n > b+\delta$ we pad the tuple with zero matrices.
            The set of cyclic tuples is open
            and corresponds to the locus of algebras of degree $1 + a + b$
            in the ADHM construction, see~\S\ref{ssec:quot_and_commuting}.
            The Hilbert scheme of up to $7$ points is
            irreducible~\cite[Theorem~1.1]{cartwright_erman_velasco_viray_Hilb8},
            so these tuples belong to the closure of the principal component.

            In the following we can thus assume $\mathbf{b > c}$. We already have
            $c\geq 2$, so $b\geq 3$ and so $a \leq 3$.
            Since $a\leq 3$ and $c\geq 2$, we have $a\leq
            \binom{c+1}{2}$.
            Recall that $n\geq 4\geq b$.
            Using Example~\ref{example:independentQuadrics} and irreducibility
            of $\cubeZeroCompIIstandard$ we
            assume that the linear span of $(A_iC_j)_{1\leq i,j\leq n}$ is
            $a$-dimensional.
            We now make a series of reductions to relate the present case to
            the previous one.
            By a version of \cite[Lemma 2.7]{Han} we may and
            will assume that $B_i = 0$.
            Using the $\GL_n$-action, we assume that $C_i = 0$ for $i > c$.
            Consider the subtuple $(x_1, \ldots ,x_c)$. The matrices $C_1,
            \ldots ,C_c$ have nonzero common cokernel, so up to linear
            transformation some of the rows are
            identically zero in each $x_i$, similarly for $A_1, \ldots ,A_c$
            if those have nonzero common cokernel. Erase those rows and
            the corresponding columns.
            We obtain a commuting tuple which
            falls in the cube zero, square nonzero case; in fact it belongs
            to the case $b=c$ just considered and is even cyclic, thus corresponds
            to an algebra. By slight abuse of notation,
            we refer to it as $(x_1, \ldots ,x_c)$. We know already that
            it lies on a non-elementary component, however, we need a bit
            more, so we make a provisional definition. We say that the tuple
            $(x_1, \ldots ,x_c)$
            is \emph{deformable in the middle} if there exists a tuple
            \[
                z_i = \begin{bmatrix}
                    0 & 0 & 0\\
                    0 & D_i & 0\\
                    0 & 0 & 0
                \end{bmatrix}
            \]
            such that $(x_i + \lambda z_i)_{i=1}^c$ is a commuting tuple for
            every $\lambda\in \kk$ and moreover at least one $z_i$ has a
            nonzero eigenvalue. We now show that for $a = 2$ the tuple is
            deformable in the middle:
            \begin{itemize}
                \item If $c\geq 3$ then the tuple corresponds to an algebra with
                    Hilbert function $(1, c, 2)$.
                    The deformation
                    from~\cite[Proposition~4.10]{cartwright_erman_velasco_viray_Hilb8} gives
                    a deformation in the middle for the monomial basis
                    $y_{c-1}^2, y_c^2, y_c,  \ldots , y_1, 1$. (The referenced result
                    requires a generality assumption on the tuple, but we may
                    impose arbitrary such assumptions.) Specifically, it gives
                    a deformation
                    \[
                        (y_iy_j\ |\ i\neq j) + (y_1^2 - \lambda y_1 - a_1y_{c-1}^2 -
                        b_1y_c^{2}) + (y_i^2 - a_{i}y_{c-1}^2 - b_iy_c^2\ |\
                        i\geq 2),
                    \]
                    parameterized by $\lambda$ where $a_i, b_i\in \kk$ are
                    constants.
                \item In the special case $c = 2$, the tuple corresponds to an
                    algebra with Hilbert function $(1, 2, 2)$.
                    This algebra is the quotient of a polynomial ring by an
                    ideal
                    generated by a single quadric --- which we may assume by
                    genericity has
                    full rank --- and all cubics, so it is isomorphic
                    to $(y_1 y_2, y_1^3, y_2^3)$. For $\lambda\in \kk$ the deformation
                    \[
                        \frac{\kk[y_1, y_2]}{(y_1y_2, y_1^3, y_2^3 -
                        \lambda y_2^2)}
                    \]
                    in the basis $y_2^2 - \lambda y_2, y_1^2, y_2, y_1, 1$ gives a
                    deformation in the middle.
            \end{itemize}
            Going back to our original tuple, we see that \rred{it}
            lies on a non-elementary component whenever $(x_1, \ldots ,x_c)$
            is deformable in the middle; this is
            because the first $c$ columns of $A_i$ matrices, for $i > c$, are
            zero. Therefore, we automatically get that the tuple lies on a
            non-elementary component whenever $a = 2$.
            We have already
            reduced to the case $a\leq 3$, so it remains to consider
            $\mathbf{a = 3}$. Since $b \geq 3$ and $a+b\leq 7-1$, we get
            $\mathbf{b = 3}$. Since $b > c \geq 2$, we get $\mathbf{c = 2}$.
            In this case for $n\geq 5$ we have a component, see
            Example~\ref{ex:332}, so we assume
            $n\leq 4$, hence $n = 4$.
            Consider the tuple corresponding to the
            module whose dual generators are $Q = z_1^2e_1^* + z_1z_2e_2^* +
            z_2^2e_3^*, L = z_3e_1^* + z_4e_2^*$. In tensor notation, in the
            basis $e_1^*, e_2^*, e_3^*, L, z_1e_1^* + z_2e_2^*, z_1e_2^* +
            z_2e_3^*, Q$ it is given
            by the following matrix with $\lambda = 0$:
            \[
                \begin{bmatrix}
                    0 & 0 & 0 & e_3 & e_1 & 0   & 0\\
                    0 & 0 & 0 & e_4 & e_2 & e_1 & 0\\
                    0 & 0 & \lambda e_1 & 0   & 0 & e_2   & 0\\
                    0 & 0 & 0 & 0 & 0 & 0 & 0\\
                    0 & 0 & 0 & 0 & 0 & 0 & e_1\\
                    0 & 0 & 0 & 0 & 0 & \lambda e_1 & e_2\\
                    0 & 0 & 0 & 0 & 0 & 0 & \lambda e_1
                \end{bmatrix}
            \]
            The above tuple commutes for arbitrary $\lambda$. For
            $\lambda\neq 0$ we get a tuple $(x_1, \ldots ,x_4)$ with $x_1$
            having eigenvalues $0$ and $\lambda$. The tuple is then a product
            of
            \begin{enumerate}
                \item a tuple of $3\times 3$ matrices acting on the
                    $3$-dimensional generalized eigenspace of $x_1$ for eigenvalue $\lambda$. This
                    tuple lies in the principal component.
                \item a tuple of $4\times 4$ matrices acting on the
                    $4$-dimensional generalized eigenspace of $x_1$ for eigenvalue zero. This tuple
                    is square-zero.
            \end{enumerate}
            By Proposition~\ref{ref:productOfComponents:prop} this shows that our
            original tuple lies on the concatenation of $4\times 4$ square-zero component and
            the $3\times 3$ principal component.
            Those components have dimensions $24$ and $18$ respectively
            by~\S\ref{ssec:examples} and~\S\ref{ssec:components} so by
            Proposition~\ref{ref:productOfComponents:prop} the
            concatenation has dimension $24 + 18 + 2\cdot 3\cdot 4 = 66$.
            A direct check shows that the
            tangent space to $C_4(\MM_7)$ at our tuple is $66$-dimensional, so our tuple is a
            smooth point of this component and so the whole locus
            $\cubeZeroCompII{3}{3}{1}{2}{4}$, being
            irreducible, is contained in this component.

            \textbf{Case 2, $\mathbf{d-a-b \geq 2}$.}
                The spaces $V/\left(\sum_{i,j} \im(x_ix_j)\right)$ and
                $\bigcap\ker(x_j^Tx_i^T)$ are dual and they have dimension $d
                - \dim \left(\sum_{i,j}\im A_iC_j\right)$. If they have codimension one,
                then the tuple $\xx^T$ falls into Case~1 (this concerns in
                particular the component $\mathcal{Z}$ from
                part~\ref{it:caseTwo}). Since we already considered Case~1,
                below we assume $\dim \sum_{i,j}\im
                A_iC_j \geq 2$.
                In particular, $\mathbf{a\geq 2}$ so $\mathbf{b\leq
                3}$.

            It also suffices to consider
            $n$-tuples $(x_1,\ldots ,x_n)\in \cubeZeroComp{a}{b}{d-a-b}{n}$
            with $x_i$ of the form~\eqref{x_i} where some linear combination
            of the matrices $A_1,\ldots ,A_n$ has rank at least $2$.
            Indeed, if each linear combination of the matrices $A_i$ has rank
            at most $1$, then the matrices $A_i$ have either a common kernel
            of codimension 1 or a common 1-dimensional image, see for
            example~\cite[Lemma~2]{AL}. The common kernel of the matrices
            $A_i$ is trivial, as discussed below~\eqref{x_i}, and $b\ge 2$, so
            these matrices do not have a codimension one common kernel.
            If the matrices $A_i$ have common $1$-dimensional
                image, then $\dim \sum_{i,j} A_iC_j = 1$, a contradiction.

            Assume first that $\mathbf{b=2}$. As above, consider an arbitrary $n$-tuple
            $(x_1,\ldots ,x_n)\in \cubeZeroComp{a}{2}{d-a-2}{n}$ with $x_i$ of
            the form~\eqref{x_i}. By the argument above, some linear combination of the
            matrices $A_i$ has rank $2$, and using the actions of $\GL_n$ and
            $\GL_d$ we may assume that $A_1=\left[
                \begin{array}{c}I\\0
                \end{array}
            \right]$. Let $A_i=\left[
                \begin{array}{c}A_i'\\A_i''
                \end{array}
            \right]$ for $i\ge 2$. The commutativity then implies
            $C_i=A_i'C_1$ for each $i\ge 2$, and hence $0=\cap_{i=1}^n\ker
            C_i=\ker C_1$. It follows that $C_1$ is injective and consequently
            $d-a-2=2$. We may assume that $C_1=I$, and then $C_i=A_i'$ for
            $i\ge 2$. Commutativity then reduces to
            $[A_i',A_j']=0$ for $i,j\ge 2$ and $A_i''=0$ for each $i\ge 2$.
            The maximal dimension
            of a commutative vector space of $2\times 2$ matrices is $2$,
            therefore we use $\GL_n$-action to assume that $A_i'=0$ for $i\ge 3$ and
            that $A_2'$ is singular. Since $A_2'$ is
            singular, there exist nonzero vectors $u,v\in \Bbbk^2$ such that
            $A_2'u=0$ and $v^TA_2'=0$. The matrix
            \[
                y = \begin{bmatrix}0&0&0\\0&uv^T&0\\0&0&0
                    \end{bmatrix}
            \]
            then commutes with $x_i$ for each $i\ge 2$. If $v^Tu\ne 0$, then
            $y$ has two distinct eigenvalues, and if $v^Tu=0$, then for
            $\lambda \neq 0$ the matrix $x_1+\lambda y$ is nilpotent, but with
            nonzero cube, so the $n$-tuple $(x_1,\ldots ,x_n)$ in both cases
            belongs to a non-elementary component, in the second case, by Proposition~\ref{cube_not_zero}.
            This in particular concludes the proof of part~\ref{it:caseTwo}.
            In the rest of the proof we assume that $\mathbf{b = 3}$. This
            forces $\mathbf{a=2}$ and
            $\mathbf{d-a-b=2}$.

        If $\mathbf{c=2}$, let $(x_1,\ldots ,x_n)\in
        \cubeZeroCompII{2}{3}{2}{2}{n}$ be an $n$-tuple of commuting matrices
        of the form~\eqref{x_i}. As $c=2$ and $b=3$, there exists a nonzero
        vector $u\in \kk^b$ such that $u^TC_i=0$ for each $i=1,\ldots ,n$.
        Consider the transposed tuple $(x_1^T, \ldots ,
            x_n^T)$. The vector $u$ forces it to have at least three-dimensional
            common kernel, so either this tuple falls into the Case~1 or into the case
        $b=2$ above.

            We thus assume that $\mathbf{c = 3}$.
            In the remaining case our strategy is to find deformations in the
            middle. The matrices $B_i$ are irrelevant here, so we assume
            $B_i=0$, so the associated module is naturally graded.
            We begin by proving that the matrices $(x_1, \ldots ,x_n)$
            span an at most three dimensional space.  We use
            apolarity~\S\ref{ssec:apolarity}. Our matrices in the standard
            basis correspond to a module $M$, where $M\subset
            \Fdual$ is generated by two (homogeneous) quadrics $Q_1, Q_2$.
            Using Example~\ref{example:twoQuadricsThreeLinears}, we get our
            claim. So we reduce to the case $n=3$ and $S = \kk[y_1, y_2, y_3]$
            and below we assume $\mathbf{n=3}$.

            We will use downward induction on $d$ and for that reason we will
            consider also $d=8,9$ and $a=2,3$ even though such large $d$ fall
            outside the scope of the current
            theorem. For a tuple $\xx$ we also introduce the number $a^T = a^T(\xx)
            := \dim \coker \xx = \dim \bigcap \ker(x_i^T)$. This number can be
            computed from the associated module and by abuse of notation we
            will view $a^T$ also as a function of a module. In this setting it
            is simply the minimal number of generators of this module,
            see~\S\ref{ssec:transposes}.

            We first tackle the case $\mathbf{(a, a^T, b, c, d) = (3, 3, 3, 3,
            9)}$. Since $d-a-b=3 = a^T$, we have $\im(\xx)
                \supset \ker(\xx)$.
            By assumptions, the matrices $A_i$ have zero common kernel and
            cokernel. Assume that no linear combination of those matrices has
            full rank. By~\cite[Theorem~1.1]{Eisenbud__Harris__vector_spaces}
            this implies that either all $A_i$ are
            skew-symmetric
            or they form
            a \emph{compression space}, which means that, up to base change,
            they jointly have the shape
            \begin{equation}\label{eq:Amats}
                \begin{bmatrix}
                    0 & 0 & *\\
                    0 & 0 & *\\
                    * & * & *
                \end{bmatrix}.
            \end{equation}
            We first show that the compression space case is in fact
            impossible. Suppose that the matrices have the
            shape~\eqref{eq:Amats}. The common kernel of $A_1$, $A_2$, $A_3$
            is zero, so up to linear change of coordinates we have
            \[
                A_1 =
                \begin{bmatrix}
                    0 & 0 & 1\\
                    0 & 0 & 0\\
                    * & * & *
                \end{bmatrix},
                \quad
                A_2 =
                \begin{bmatrix}
                    0 & 0 & 0\\
                    0 & 0 & 1\\
                    * & * & *
                \end{bmatrix}
                \quad \mbox{and}\quad
                A_3 =
                \begin{bmatrix}
                    0 & 0 & 0\\
                    0 & 0 & 0\\
                    * & * & *
                \end{bmatrix}.
            \]
            The condition $A_iC_j = A_jC_i$ for all $i,j$ then forces all
            $C_i$ matrices to have zero last row, which implies that
            $\dim(\im(\xx) + K_1(\xx))\leq 5$ and thus contradicts $a+c = 6$.
            Therefore, we assume that $A_i$ are skew-symmetric.
            Since they have
            zero common kernel, in fact the matrices $A_i$ form a basis of the
            space of skew-symmetric matrices. But then commutativity implies
            that $C_i = 0$ for all $i$, a contradiction. Summing up, there
            exists a linear combination of $A_i$ that has full rank. Repeating
            the argument for $C_i$ and performing linear operations, we assume
            $A_1 = C_1 = I$. Commutativity with $x_1$ implies $A_i = C_i$
            for all $i$ and then $A_i$, $A_j$ commute for any $i$, $j$. We
            have now an obvious deformation in the middle that
            adds a block $A_i$ in the middle of each $x_i$. Since $A_1$ has
            three nonzero eigenvalues, this deformation
            splits the module into a degree three and six modules. (This
            case did not use the assumption $n=3$.)

            Now we consider the case $\mathbf{(a, a^T, b, c, d) = (3, 2, 3, 3,
            8)}$ and additionally make the following
                assumption: if $Q_1, Q_2\in \Fdual_2$ are the homogeneous dual
                generators for the module $M$ associated
                to our tuple, then no linear combination of $Q_1$, $Q_2$ is
                ``rank one'', i.e., annihilated by a two-dimensional subspace of $S_1$.
            This is the most challenging part. Consider the dual module $M^{\vee}$.

            Since $a=3$, the
            dual module is minimally generated by three
            elements~\S\ref{ssec:transposes} so we present it
            as a quotient of $S^{\oplus 3}$. By
            apolarity~\S\ref{ssec:apolarity} our module $M$ is a submodule of
            $(S^{\oplus 3})^*$. Since $a^T=2$,  the module $M$ is generated by two elements
            corresponding to $e_7$, $e_8$ vectors. Since $B_i = 0$ for all
            $i$, the module $M$ is naturally graded and the two above generators are
            of degree two.

            We will show that $M^{\perp}$, the annihilator of
            $M$ in $S^{\oplus 3}$, has a degree two
            element among minimal homogeneous generators. If it is so, then replacing this
            generator $g\in M^{\perp}$ by $S_1g$ we get an
            inclusion $N \subset M^{\perp}$ with $\dim_{\kk} M^{\perp}/N = 1$.
            Therefore, we have
            $M^{\vee}  \simeq  S^{\oplus 3}/M^{\perp}$ is
            a quotient of a module $S^{\oplus 3}/N$ of degree $9$.
            Since $B_i$ are assumed to be zero, the modules $M$, $M^{\vee}$,
            $S^{\oplus 3}/N$ are all graded and moreover
            $H_{S^{\oplus 3}/N} = (3, 3, 3)$. This shows that the module
            $S^{\oplus 3}/N$ corresponds to a tuple of matrices with
            invariants $(a, a^T, b, c, d) = (3, 3, 3, 3,
            9)$. In terms of matrices,
            this means that the transpose of our tuple arises from a $9\times 9$ tuple by
            erasing the first row and column. The deformation in the middle in the case
            $9\times 9$, constructed above, gives a deformation in the middle in our case.

            To look for generators of $M^{\perp}$ we will look at the
            syzygies of $M$.
            The module $M$ is generated by $Q_1$, $Q_2$ so we have $M =
            S^{\oplus 2}/K$. Since $\dim M_1 = 3$, we have $\dim K_1 = 3$. Let
            $k_1, k_2, k_3\in S^{\oplus 2}$ be a $\kk$-basis of $K_1$.

            Assume first that there are no linear
            forms $l_1, l_2, l_3\in S$ not all equal to zero such that $\sum
            k_i l_i = 0$. Write $k_i = k_{1i}e_1 + k_{2i}e_2$ so that the
            vector $(k_1, k_2, k_3)$ becomes a matrix
            \[
                \begin{bmatrix}
                    k_{11} & k_{12} & k_{13}\\
                    k_{21} & k_{22} & k_{23}
                \end{bmatrix}.
            \]
            Let $\Delta_i$ be the minor of this matrix obtained by removing
            $i$-th column. By a general fact, the vector $(\Delta_1,
            -\Delta_2, \Delta_3)^T$ lies in the kernel of this matrix. Let us
            translate this to the resolution language. This vector then
            becomes a quadratic syzygy for the module
            $M$.
            Our assumption on the nonexistence of linear forms $l_1,
            \ldots l_3$ such that $\sum k_il_i = 0$ implies that there are no
            linear syzygies between $k_1,k_2,k_3$ so the above syzygy is minimal and the minimal resolution
            of $M$ looks like
            \[
                \begin{tikzcd}
                    0 & \arrow[l] M & \arrow[l] S^{\oplus 2}(-2) & \arrow[l]
                    S^{\oplus 3}(-3)\oplus G_1 & \arrow[l] S^{\oplus 1}(-5) \oplus G_2 &
                    \arrow[l] G_3 &\arrow[l] 0
                \end{tikzcd}
            \]
            for some graded free $S$-modules $G_1, G_2, G_3$.
            Now by Theorem~\ref{ref:duality_for_resolutions:thm} the resolution of $S^{\oplus
            3}/M^{\perp}$ is dual to the resolution of $M$ twisted by $n = 3$ so the minimal
            syzygy $S(-5)$ above corresponds to a minimal generator
            of degree $5-3 = 2$ in $M^{\perp}$ and we win.

            It remains to consider the case when there are linear forms $l_1,
            l_2, l_3\in S$ not all equal to zero such that $\sum k_i l_i = 0$.
            In other words assume that $M$ has some linear syzygy.
            Green's Linear Syzygy
            Theorem~\cite[Theorem~7.1]{eisenbud:syzygies} implies that there
            are some linear forms annihilating a generator of our
            module. Below we prove this directly for our case, without using
            the theorem.

            If $l_1, l_2, l_3$ are linearly independent, then up to coordinate
            change $(l_1, l_2, l_3) = (y_1, y_2, y_3)$. By a direct
            computation, the triples $(k_1', k_2', k_3')$ of linear forms such
            that $\sum k_i'y_i = 0$ are linear combinations of the rows of
            \[
                B = \begin{bmatrix}
                    0 & -y_3 & y_2\\
                    y_3 & 0 & -y_1\\
                    -y_2 & y_1 & 0
                \end{bmatrix}
            \]
            so $(k_{i1}, k_{i2}, k_{i3})$ are linear combinations of the
            rows of $B$ for $i=1,2$ so
            \[
                \begin{bmatrix}
                    k_{11} & k_{12} & k_{13}\\
                    k_{21} & k_{22} & k_{23}
                \end{bmatrix} = A\cdot B
            \]
            for some matrix $A\in \MM_{2\times 3}(\kk)$ of rank two. The
            matrix $A B A^{T}$ is nonzero and antisymmetric so it has the form
            \[
                \begin{bmatrix}
                    0 & \ell\\
                    -\ell & 0
                \end{bmatrix}
            \]
            for a non-zero linear form $\ell$. This matrix is equal to
            $[k_{ij}]\cdot A^T$, so $-\ell e_2, \ell
            e_1\in K_1$.
            After a coordinate change we assume $\ell = y_3$ and view $S^{\oplus
            3}/M^{\perp}$ as a module over $\kk[y_1, y_2]$. Since
            $M^{\perp}_1$ is $3$-dimensional, it generates at most a
            $6$-dimensional subspace of the $7$-dimensional space
            $M^{\perp}_2$. So there is a minimal quadric generator and we
            conclude.

            If $l_1, l_2, l_3$ are linearly dependent, then up to coordinate
            change we have $(l_1, l_2, l_3) = (y_2, y_1, 0)$, so $k_1 =
            \lambda y_1 e_1 + \mu y_1 e_2$ and $k_2 = -\lambda y_2 e_1 - \mu
            y_2 e_2$ so both $y_2$ and $y_1$ annihilate $\lambda Q_1 + \mu
            Q_2$, which is thus a rank one quadric which condradicts our
            assumption that no such quadric exists in $\spann{Q_1, Q_2}$. This concludes the case $(a, b, c, d)
            = (3, 3, 3, 8)$.

            The last remaining case for $b = 3$ is $\mathbf{(a, b, c, d) = (2,
            3, 3, 7)}$.
            As explained at the beginning of the Case~2, we have a matrix
            $A_i$ of rank two, in particular the joint image of the matrices $A_i$
            spans the first two coordinates. Using $c=3$ we deduce $a^T = 2$.

            From the point of view of modules, in this case we consider the module $S^{\oplus
            2}/M^{\perp}$ with Hilbert function $(2, 3, 2)$. We have $\dim M^{\perp}_1 = 3$, so that
            $\dim S_1M^{\perp}_1 \leq 9$ while $\dim M^{\perp}_2 = 10$.
            Therefore $M^{\perp}$ has a minimal generator of degree two and so
            arguing as in the $d=8$ case, we get that
            the dual of $M$ is a quotient of a graded module $N$ with Hilbert function $(2,
            3, 3)$. 
            The module $S^{\oplus 2}/M^{\perp}$ is dual to
            $M$ and
            $a(M^{\vee}) = a^T(M) = 2$ and $a^T(M^{\vee}) = 2$. As in
            the previous case, we get $a(N) = 3$, $a^T(N) = 2$, $b(N) = c(N) = 3$.
                Consider the degree zero part of $S^{\oplus 2}/M^{\perp}$.
                Suppose first that no element of this part is annihilated by a
                two-dimensional subspace of $S_1$, then the same holds for $N$. In this case we reduce to
                the case $d=8$: the deformation in the middle for $N$ has been constructed
            above.

                It remains to consider the case where there is a degree zero
                element of $S^{\oplus 2}/M^{\perp}$
                annihilated by a two-dimensional subspace of $S_1$.
                Consider homogeneous dual generators for this
                module. Since $a = 2$ and $S^{\oplus 2}/M^{\perp}$ has
                Hilbert function $(2, 3, 2)$, these are two quadrics. So we
                consider the locus $\mathcal{L}$ of pairs of quadrics $Q_1, Q_2\in
                (S^{\oplus 2})^*_2$ such that two conditions hold:
                \begin{enumerate}
                    \item some non-zero linear combination of $Q_1$, $Q_2$ has
                        ``rank one'', i.e., it is annihilated by a codimension one space
                        of $S_1$,
                    \item $\dim_{\kk}(S_1Q_1 + S_1Q_2) \leq 3$.
                \end{enumerate}
                Up to linear transformations we may assume $Q_1 = z_1^2 e_1^*$
                and either $y_1Q_1 = y_2Q_2$ or
                $y_1(Q_1 - Q_2) = 0$.
                The locus with $y_1(Q_1 - Q_2) = 0$ lies in the closure of the one
                where $y_1Q_1 = y_2Q_2$. The locus where $y_1Q_1 = y_2Q_2$ is just
                an affine space. This shows that $\mathcal{L}$ is
            irreducible.

                To prove that the modules obtained from $\mathcal{L}$ lie on a non-elementary component
                we will consider a deformation that is \textbf{not} in the
                middle, so below we also take into account the $B_i$ matrices
                and so we consider all $n$, not necessarily $n = 3$.
                By Corollary~\ref{ref:truncation:cor} we may assume $n\geq 4$.
                The addition of arbitrary $B_i$ matrices does not break the irreducibility of
                our locus (on the level of dual generators it amounts to adding arbitrary linear
                forms to each $Q_i$).
                For every $\lambda\in \kk$ consider the commuting $n$-tuple
                written in tensor notation as
                \[
                    \xx(\lambda) := \begin{bmatrix}
                        \lambda e_1 & 0 & e_2 & e_1 & 0 & 0 & 0\\
                        0 & 0 & e_3 & 0 & e_1 & 0 & e_4\\
                        0 & 0 & \lambda e_1 & 0 & 0 & e_1 & 0\\
                        0 & 0 & 0 & 0 & 0 & e_2 & e_1\\
                        0 & 0 & \lambda e_3 & 0 & \lambda e_1 & e_3 & 0 \\
                        0 & 0 & 0 & 0 & 0 & 0 & 0\\
                        0 & 0 & 0 & 0 & 0 & 0 & 0
                    \end{bmatrix}
                \]
                so that the matrices $x_i$ are zero for $5\leq i\leq n$.
                The tuple $\xx(0)$ corresponds to a module from $\mathcal{L}$
                with $B_i$ matrices added.
                For $\lambda\neq 0$ the tuple $\xx(\lambda)$ corresponds to a
                module in the concatenation of the $4\times 4$ square zero
                component in $C_n(\MM_4)$ and the principal component in
                $C_n(\MM_3)$. By
                Proposition~\ref{ref:productOfComponents:prop} and using the
                formulas for dimensions given in~\S\ref{ssec:components}
                and~\S\ref{ssec:examples} this
                component has dimension $34 + 8n$.
                For $n = 4$ we directly check that $T_{\xx(0)}C_4(\MM_7)$ is
                $66$-dimensional. We also directly compute that a matrix
                commuting with all elements of $\xx(0)$ is an element of $\kk[\xx(0)]$ and
                $\dim_{\kk} \kk[\xx(0)] = 8$. As in
                Remark~\ref{ref:addingMatrices:rmk} this shows that for $n\geq
                4$ the tangent space $T_{\xx(0)}C_n(\MM_7)$ has dimension
                $66+8(n-4) = 34 + 8n$. Thus $\xx(0)$ is a smooth point of the
                concatenated component. By irreducibility the whole locus
                lies in the corresponding non-elementary
                component of Quot.
                This concludes the whole proof.
        \end{proof}

        \subsection{Square-zero cases}

            In this section we consider commuting matrices such that the
            square of any
            linear combination of them is zero. Since $\chr \kk\neq 2$, we equivalently assume $x_ix_j = 0$
            for all $i, j$. This implies that the image of each $x_i$ is
            contained in the common kernel. Extending a basis of $\sum \im
            x_i$ to a basis of $V$ we get matrices that have nonzero entries
            only in the top-right $m\times (d-m)$ corner.

            For any $d$, any $m\in \{1,\ldots ,d-1\}$ and any $n\geq1$ define
            \[
                \squareZeroCompUpper{m}{d-m}{n}=\left\{\left(\left[
                    \begin{array}{cc}0&A_1\\0&0
                    \end{array}
                \right],\ldots ,\left[
                    \begin{array}{cc}0&A_n\\0&0
                    \end{array}
                \right]\right);A_1,\ldots ,A_n\in \MM_{m\times (d-m)}\right\}
                \subset \CommMat.
            \]
            and let $\squareZeroComp{m}{d-m}{n} = \GL_d \cdot
            \squareZeroCompUpper{m}{d-m}{n}$.
            The discussion above shows that a square-zero tuple $(x_i)$ with
            $\dim \sum \im(x_i) = m$ corresponds to a point in
            $\squareZeroComp{m}{d-m}{n}$. The locus
            $\squareZeroComp{m}{d-m}{n}$ is irreducible for all values of $n$,
            $m$, $d$ and its dimension is $(n+1)m(d-m)$ by~\S\ref{ssec:examples}.

            Let $\mathcal{V}_{m, d-m}^n \subset\CommMat$ consist of tuples of
            commuting matrices with lower left $(d-m)\times m$ corner equal to zero.
            We have an inclusion $\squareZeroCompUpper{m}{d-m}{n} \subset \mathcal{V}_{m,
            d-m}^n$. Moreover, for every $t\in \kk$ and a commuting tuple
            of matrices
            $
                \begin{bmatrix}
                    B_i & A_i\\
                    0 & C_i
                \end{bmatrix}$
                 in $\mathcal{V}_{m, d-m}^n$ the tuple
                $\begin{bmatrix}
                    tB_i & A_i\\
                    0 & tC_i
                \end{bmatrix}$
            commutes as well. Letting $t$ go to zero, we get that
            $\squareZeroCompUpper{m}{d-m}{n} \subset \mathcal{V}_{m,
            d-m}^n$ is a retract. Let $\mathcal{W}_{m,d-m}^n$ be the
            intersection of the principal component of $\CommMat$ with
            $\mathcal{V}_{m, d-m}^n$. The proof in~\S\ref{ssec:components} that the
            principal component is irreducible and of dimension $d^2+(n-1)d$,
            adapts immediately and shows that $\mathcal{W}_{m,d-m}^n$ is irreducible and of
            dimension $d^2-(d-m)m+(n-1)d$, since a polynomial in a matrix with
            zero lower left $(d-m)\times m$ corner is a matrix of the same
            shape. The retraction $\mathcal{V}_{m,d-m}^n\to
            \squareZeroCompUpper{m}{d-m}{n}$ restricts to a retraction
            $\mathcal{W}_{m, d-m}^n\to \squareZeroCompUpper{m}{d-m}{n}\cap
            \mathcal{W}_{m, d-m}^n$. We define
            \[
                \pi_{m, d-m}^n\colon \mathcal{W}_{m, d-m}^n\to
                \squareZeroCompUpper{m}{d-m}{n}
            \]
            as the composition of the above retraction and the inclusion in
            $\squareZeroCompUpper{m}{d-m}{n}$.
            The image of $\pi_{m, d-m}^n$ is $\squareZeroCompUpper{m}{d-m}{n}\cap
            \mathcal{W}_{m, d-m}^n$ so it is closed, but not necessarily
            the whole $\squareZeroCompUpper{m}{d-m}{n}$.

            \noindent In the following theorem we restrict to $n\geq
            4$, since for $n\leq 3$ the variety $C_n(\MM_d)$ is irreducible for all $d\leq
            10$ by~\cite{MT, Sivic__Varieties_of_commuting_matrices}.
            \begin{theorem}\label{square_zero}
                Let $n\ge 4$ and $d\le 7$ be arbitrary and $1\le m\le d-1$. Then the following holds:
                \begin{enumerate}
                    \item[(a)]
                        If $m=1$ or $m=d-1$, then $\squareZeroComp{m}{d-m}{n}$ belongs to the principal component.
                    \item[(b)]
                        If $m\in \{2,d-2\}$ and $n<d$, then $\squareZeroComp{m}{d-m}{n}$ belongs to the principal component.
                    \item[(c)]
                        The variety $\overline{\squareZeroComp{3}{3}{n}}$ belongs to the
                        principal component for $n<6$.
                    \item[(d)]
                        If $d=7$ and $m\in \{3,4\}$, then
                        $\squareZeroComp{m}{d-m}{4}$ belongs to the
                        concatenation of $(\Bbbk
                        I)^4+\overline{\squareZeroComp{2}{2}{4}}$ and
                        the principal component of $C_4(\MM_{3})$.
                    \item[(e)]
                        In all other cases the set $(\Bbbk
                        I)^n+\overline{\squareZeroComp{m}{d-m}{n}}$ is an
                        irreducible component of $\CommMat$.
                \end{enumerate}
            \end{theorem}
            \begin{proof}
                Case \textbf{(a)}. Up to transposition we may assume $m=1$. Up to $\GL_d$ and $\GL_n$ action we have
                $x_i = E_{1i+1}$ for $i=1,2, \ldots ,s$ and $x_{i} = 0$ for
                $i>s$. Then our tuple is a limit of commuting tuples of the form
                $(E_{1, i+1} + \lambda E_{i+1,i+1}\ :\ i=1,2 \ldots ,s)$ which
                belong to the principal component.

                Case \textbf{(b)}. Recall that $d > n\geq 4$, so $d\geq 5$. Up to transposition we may assume $m=2$.
                Additionally, by the same argument as in Corollary~\ref{ref:truncation:cor} we assume $n=d-1$. We
                will show that the image of the map $\pi =
                \pi_{2,d-2}^{d-1}$ contains an open subset of
                $\squareZeroCompUpper{2}{d-2}{d-1}$. We define $x_j=x^j$ for $j=1,\ldots ,n=d-1$ where
                \[
                    x=\left[\begin{array}{c c | c c c c}
                        &1&1&\\-1&&&1&\\\hline
                        &&&1\\&&&&\ddots\\&&&&&1\\&&-1
                    \end{array}\right].
                \]
                The matrix $x_1$ has all eigenvalues distinct.
                We directly verify that
                the point
                $(x_1, \ldots ,x_{d-1})\in \mathcal{W}_{m, d-m}^n$ is smooth.
                We directly
                compute that the tangent map $d\pi$ is
                surjective at $(x_1, \ldots ,x_{d-1})$
                so the image of $\pi$
                contains an open neighbourhood of $\pi( (x_1, \ldots
                ,x_{d-1}))$ in $\squareZeroCompUpper{2}{d-2}{d-1}$. On the other
                hand, the image of the map $\pi$ is a closed subset, so
                $\pi$ is surjective.

                Case \textbf{(c)}. We follow the argument from Case~(b). It is
                enough to make the proof for $n=5$ and
                we take the tuple
                $(x, x^2, x^3, x^4, x^5)$ where
                \[
                    x=\left[\begin{array}{ccc|ccc}
                        &1&&1\\&&1&&1\\1&&&&&1\\\hline&&&&1\\&&&-1\\&&&&&-1
                \end{array}\right].
                \]
                Case \textbf{(d)}.
                Let $\mathcal{C}$ denote the concatenated
                component. We employ an argument analogous to the previous
                cases, but consider a
                retraction from
                $\mathcal{C}\cap \mathcal{V}_{m, d-m}^n$ to $\mathcal{C}\cap
                \squareZeroCompUpper{m}{d-m}{n}$ and a map $\pi'$ analogous to
                $\pi$. Transposing if necessary, we can
                assume $m=3$. Consider a tuple $\xx_0$ of four
                matrices, written in tensor notation
                (see~\S\ref{ssec:examples}) and another invertible matrix $g$
                as below.
                \[
                    \xx_0 = \begin{bmatrix}
                        e_3 + e_4 & 0 & 0 & 0 & 0 & 0 & 0\\
                        0 & e_2 + e_4 & 0 & 0 & 0 & 0 & 0\\
                        0 & 0 & e_1 + e_4 & 0 & 0 & 0 & 0\\
                        0 & 0 & 0 & 0 & 0 & e_2 & e_1\\
                        0 & 0 & 0 & 0 & 0 & e_4 & e_3\\
                        0 & 0 & 0 & 0 & 0 & 0 & 0\\
                        0 & 0 & 0 & 0 & 0 & 0 & 0
                    \end{bmatrix}\qquad g = \begin{bmatrix}
                        1 & 0 & 0 & 0 & 1 & 0 & 0\\
                        0 & 1 & 0 & 1 & 1 & 0 & 0\\
                        0 & 0 & 1 & 1 & 0 & 0 & 0\\
                        0 & 0 & 0 & 1 & 0 & 0 & 0\\
                        0 & 0 & 0 & 0 & 1 & 0 & 0\\
                        0 & 0 & 0 & 0 & 0 & 1 & 0\\
                        0 & 0 & 0 & 0 & 0 & 0 & 1
                    \end{bmatrix}.
                \]
                We take the tuple $\xx = (gx_ig^{-1})_{x_i\in \xx_0}$.
                Directly by construction of $\xx_0$, the module corresponding
                to both tuples lies in the concatenated component.
                By Proposition~\ref{ref:productOfComponents:prop} the component
                $\mathcal{C}$ is $66$-dimensional. By the construction from
                that proposition, the intersection of $\mathcal{C}$ with
                the space of tuples of matrices with zero lower $4\times 3$ corner
                contains the locus obtained by naively concatenating the
                components and then acting with the parabolic subgroup of $\GL_7$
                consisting of matrices with zero lower $4\times 3$ corner.
                In particular, this intersection contains a locus of
                dimension at least $66 - 12$ and our point lies on that locus.
                We directly
                compute that the tangent space at $\xx$ to $\mathcal{V}_{3,
                4}^4$ has dimension $54 =
                66 - 12$, so $\xx$ is a smooth point of $\mathcal{C}\cap
                \mathcal{V}_{3, 4}^4$. We also directly compute that $d\pi'$
                is surjective at this point.

                It remains to prove that the other cases correspond to
                elementary components. For $m=2$ and $n\geq d$ it is
                enough (Remark~\ref{ref:addingMatrices:rmk}) to take $n=d$ and the smooth points with
                trivial negative tangents are exhibited in
                Examples~\ref{ex:squareZero4x4},~\ref{ex:squareZero5x5},~\ref{ex:squareZero6x6m2},~\ref{ex:squareZero7x7m2}.
                The transposes of those give the cases for $m=d-2$.
                For $(m,d)=(3, 6)$ see Example~\ref{ex:squareZero6x6m3},
                while for $(m, d) = (3, 7)$ see
                Example~\ref{ex:squareZero7x7m3}.
            \end{proof}

            \begin{remark}
                In the case \textbf{(d)}, we have strong computational
                evidence that the locus does \emph{not} lie in the principal
                component, however this remains open. Also, this locus
                contains no smooth points of $C_4(\MM_7)$.
            \end{remark}

            \begin{proof}[Proof of
                Theorem~\ref{ref:numberOfComponentsIntro:thm} and correctness of Table~\ref{tab:componentTableForQuot}]
                First consider elementary components of $\CommMat$. By
                Proposition~\ref{cube_not_zero}, Theorem~\ref{cube_zero} and
                Theorem~\ref{square_zero} apart from the principal component
                of $C_n(\MM_1)$, all other elementary components are the
                ones listed in~\S\ref{ssec:examples} and their transposes (as
                well as components formed by increasing $n$, see
                Remark~\ref{ref:addingMatrices:rmk},
                but we ignore these, counting them only at the very end).
                Explicitly, these are the square-zero loci for
                \[
                    (d, m, n) = (4,
                    2, 4),\ (5, 2, 5),\ (5, 3, 5),\ (6, 2, 6),\ (6, 3, 6),\
                    (6, 4, 6),\ (7,2,7),\
                    (7, 3, 5),\ (7,
                4, 5),\ (7, 5, 7)
                \]
                and two other components: the cube-zero, square non-zero component from
                Example~\ref{ex:332} and its transpose (both with $n=5$).
                Clearly this last component could only coincide with its
                transpose. But this does not happen, as the general tuple from
                this component has (after adding multiples of the identity
                matrix to make the
                matrices nilpotent) a three-dimensional common kernel, while its
                transpose has a two-dimensional common kernel.
                The square zero loci mentioned above are also pairwise
                distinct since $m$ can be recovered from a general element of
                the square-zero locus as the dimension of the common kernel.
                We have the following number of those elementary components:
                \[
                    \begin{tabular}{c c c c c c c c c}
                        &           $d=4$ & $d=5$  & $d=6$ & $d=7$\\\midrule
                        $n = 4$ &  $1$  & $ $    & $ $   & $ $ \\
                        $n = 5$ &  $ $  & $2$    & $ $   & $4$ \\
                        $n = 6$ & $ $  & $ $    & $3$   & $  $\\
                        $n = 7$  & $ $  & $ $    & $ $   & $2$\\
                    \end{tabular}
                \]
                By increasing $n$ as in Remark~\ref{ref:addingMatrices:rmk} and/or concatenating with the principal
                component of $C_n(\MM_1)$ to increase $d$ we get the numbers of components as
                in Table~\ref{tab:componentTable}. Arguing as in
                Proposition~\ref{ref:genericallyReduced:prop} below, we conclude that
                these are all the components. Finally, to obtain the number of
                components for $\Quotmain$ we compute the number of generators
                for the modules corresponding to general points of the elementary components in
                Section~\ref{ssec:examples} and, for the non-elementary
                components, note that if modules $M_1$ and $M_2$ have disjoint
                supports, then a surjection from $S^{\oplus r}$ onto
                $M_1\oplus M_2$ exists if and
                only if such surjections exist for both $M_1$ and $M_2$.

                Every component contains a smooth point as proven in
                Proposition~\ref{ref:genericallyReduced:prop}.
            \end{proof}

        \subsection{A generically nonreduced component of
            $\Quot{4}{8}$}\label{ssec:nonreducedComponent}

            \begin{proposition}\label{ref:genericallyReduced:prop}
                For $d\leq 7$ and any $n$, $r$ the schemes $\CommMat$ and $\Quotmain$ are
                generically reduced (their singular loci are nowhere dense).
            \end{proposition}
            \begin{proof}
                It is enough to prove this for $\Quotmain$, since for $r$
                large enough $\CommMat$ is an image of a smooth map from a
                scheme $\prodfamilystable$ smooth over $\Quotmain$,
                see~\S\ref{ssec:components} and the target of a smooth map is
                reduced if and only if the source
                is~\cite[Tag~039Q, Tag~025O]{stacks_project}.
                The classification of elementary components for degrees at
                most $7$ shows that each of them has a smooth point, so is
                generically reduced.
                Now we argue that for $d\leq 7$ \emph{every} component of $\Quot{n}{d}$ is
                generically reduced. In one sentence, this is because each
                non-elementary component locally in \'etale topology (over
                $\mathbb{C}$ one can take analytic topology) is a symmetrized product of
                elementary components.
                To make this more precise, consider an arbitrary component
                $\mathcal{Z}$ and a general point $[M]$ on it.  Suppose that $M =
                \bigoplus M_i$, where each $M_i$ is supported at a single point.
                Since $[M]$ is general, the number of summands is maximal, so that
                $M_i$ is not a limit of reducible modules. By generality again, we
                may assume $[M_i]$ lies on a single elementary component
                $\mathcal{Z}_i$ and is a smooth point there. It
                follows that $\mathcal{Z}$ is a concatenation of $\mathcal{Z}_i$.
                By Proposition~\ref{ref:productOfComponents:prop} and since $[M_i]$ are
                smooth, the point $[M]$ is smooth as well.
            \end{proof}
            The aim of this section is to show that
            Proposition~\ref{ref:genericallyReduced:prop} is no longer valid
            for $d = 8$ and $n\geq 4$.  Apart from theoretical importance,
            this is relevant for the search for elementary components, since
            one method for such a search is that a locus is constructed and
            then certified to be a
            component by exhibiting a smooth point. The result below shows
            that for larger $d$, $n$ this method is in general insufficient.

            We first discuss the general assumptions necessary for our example and then
            provide a concrete example satisfying them. For a graded module
            $N$ its \emph{Hilbert series} is the formal series $\sum_i
            (\dim_{\kk} N_i) T^i$.
            Consider the locus
            $\mathcal{L}_0 = \squareZeroCompUpper{4}{4}{4}\subset
            C_4(\MM_8)$ of $4$-tuples of $8\times 8$ matrices that
            have nonzero entries only in the top right quadrant.
            From the module point of view we consider $S = \kk[y_1, \ldots
            ,y_4]$, a free module $F = S^{\oplus 4}$ and quotient modules of the
            form $F/K$ where $K$ is generated by all quadratic forms and
            $4\cdot 4 - 4 = 12$ linear forms, so that the quotient $F/K$ is
            graded with Hilbert series $4+4T$. Denote their locus (with
            reduced scheme structure) by
            $\mathcal{L}^{\OpQuot}_0$. Denote by $\mathcal{L}^{\OpQuot}$ the
            locus consisting of modules in $\mathcal{L}^{\OpQuot}_0$
            translated by an arbitrary vector in $\mathbb{A}^4$.
            The above construction gives morphisms $\Gr(12, 16)\to
            \mathcal{L}^{\OpQuot}_0$ and $\Gr(12, 16) \times \mathbb{A}^4\to
            \mathcal{L}^{\OpQuot}$ that are bijective on closed points, in
            particular $\mathcal{L}^{\OpQuot}$ is $52$-dimensional.

            Consider a point $[M]$ corresponding to a quotient $F/K$ for $K$
            as above.
            We will now put several constrains on $K$ that allow us to deduce
            nonreducedness and finally we exhibit an example of $K$ satisfying
            those assumptions.
            \begin{assumption}\label{assu:first}
                Assume that $K$ is generated by linear forms.
            \end{assumption}
            Under the above assumption, the multiplication map
            $S_1\otimes_{\kk} K_1 \to K_2$ is onto, so its kernel is
            $8$-dimensional and so the minimal graded free resolution of $F/K$
            has the shape
            \begin{equation}\label{eq:minresDegree8}
                \begin{tikzcd}
                    0 & \arrow[l] F/K & \arrow[l] F & \arrow[l] S^{12}(-1) &
                    \arrow[l] S^{\oplus 8}(-2) \oplus F_2'\arrow[l]
                        \ldots 
                \end{tikzcd}
            \end{equation}
            where the free module $F_2'$ is a direct sum of $S(-j)$ for
            varying $j\geq 3$.
            \begin{lemma}\label{ref:tangentToQuotAtLocus:lem}
                The tangent space $T$ to $\Quot{4}{8}$ at $[F/K]$ as above is
                homogeneous and satisfies $\dim T_{-1} \geq 16$ as well as
                $\dim T_0 = 48$.
            \end{lemma}
            \begin{proof}
                By Lemma~\ref{ref:tangentSpaceToQuot:lem}, the tangent space $T$ is isomorphic to $\Hom_S(K, F/K)$
                so it is homogeneous. From~\eqref{eq:minresDegree8} we see
                that the minimal graded free resolution of $K$ begins with
                \[
                    \begin{tikzcd}
                        0 & \arrow[l] K & \arrow[l] S^{12}(-1) &
                        \arrow[l] S^{\oplus 8}(-2) \oplus F_2' & \arrow[l]
                        \ldots
                    \end{tikzcd}
                \]
                Applying $\Hom_S(-, F/K)$ we get an exact sequence
                \[
                    \begin{tikzcd}
                        0 \arrow[r] & T \arrow[r] & (F/K)^{\oplus 12}(1)
                        \arrow[r] & (F/K)^{\oplus 8}(2) \oplus \Hom_S(F_2',
                        F/K).
                    \end{tikzcd}
                \]
                The Hilbert series of $\Hom_S(F_2', F/K)$ is concentrated in
                degrees $\leq -2$.
                The Hilbert series of $(F/K)^{\oplus 12}(1)$ is
                $48T^{-1} + 48$ while the Hilbert series of $(F/K)^{\oplus
                8}(2)$ is $32T^{-2} + 32T^{-1}$. Thus the kernel has Hilbert
                series $\alpha T^{-1} + 48$, where $\alpha\geq 16$, which is exactly
                the claim.
            \end{proof}
            One can observe that the $48$ tangent directions in degree zero
            arise from varying the $4\cdot 12$ coefficients of the linear
            forms in $K_1$.
            \begin{assumption}\label{assu:tg}
                Assume that the tangent space at $[F/K]\in \Quot{4}{8}$ is
                $64$-dimensional.
            \end{assumption}
            \def\locring{\hat{\OO}_{[M]}}%
            Consider now the complete local ring $\locring$ of $[M]\in
            \Quot{4}{8}$.
            By Example~\ref{ex:monomialBasis}, there is an
            affine open neighbourhood $\Spec(A)$ of $[M]$ preserved by the
            usual $\Gmult$-action, see Section~\ref{sec:BBdecomposition} for
            details on the $\Gmult$-action. This action gives a $\mathbb{Z}$-grading on
            the algebra $A$. The point $[M]$ is $\Gmult$-fixed since $K$ is
            homogeneous, so that the corresponding maximal ideal $\mm \subset
            A$ is also graded. Therefore also $A/\mm^k$ for every $k$ and in
            particular $\mm/\mm^2$ are graded.
            Let $\hatS= \kk[[t_1, \ldots
            ,t_{16}, u_1, \ldots , u_{48}]]$ where $t_i$ and $u_j$ forms a
            basis of the cotangent space, where $\deg t_i = 1$ and $\deg u_j =
            0$; note that the degrees on the cotangent space are the negatives
            of the degrees on the tangent space. Let $\mms= (t_1, \ldots ,
            t_{16}, u_1, \ldots ,u_{48})$.
            By Assumption~\ref{assu:tg} the local ring $\locring$ has the form
            $\locring = \hatS/I$, where $I \subset \mms^2$.

            We now come to the main part of the
            argument. Recall the primary obstruction map from
            Theorem~\ref{ref:primaryForQuot:thm}. Its dual is a homogeneous map
            \[
                o\colon \Ext^1(K, M)^{\vee} \to \Sym_2(\mms/\mms^2).
            \]
            We restrict it to the degree $2$ part of $\Ext^1(K, M)^{\vee}$ and
            obtain the
            map
            \begin{equation}\label{eq:baro}
                \bar{o}\colon \left(\Ext^1(K, M)^{\vee}\right)_{2} \to
                \Sym_2(\spann{t_1, \ldots ,t_{16}}).
            \end{equation}
            We now make our third and final restriction about the point.
            \begin{assumption}\label{assu:dimension}
                The quotient ring of $\kk[t_1, \ldots ,t_{16}]$ by the ideal generated by $\im \bar{o}$ has dimension
                at most $4$.
            \end{assumption}
            \begin{theorem}\label{ref:genericallyNonreduced:thm}
                Let $[M] = [F/K]$ be a point of $\mathcal{L}^{\OpQuot}_0\subset \Quot{4}{8}$ satisfying
                Assumptions~\ref{assu:first}-\ref{assu:dimension}. Then
                the image of
                $\mathcal{L}^{\OpQuot}\to \Quot{4}{8}$ is $|\mathcal{Z}|$ for
                an
                irreducible component $\mathcal{Z}$ of $\Quot{4}{8}$. The component
                $\mathcal{Z}$ is
                generically nonreduced.
            \end{theorem}
            \begin{proof}
                For every $k\geq 3$ consider the natural map
                \[
                    \varphi_k\colon\kk[t_1, \ldots ,t_{16}]\to
                    \frac{\locring}{\mm^k+(u_1,  \ldots , u_{48})}.
                \]
                This is a homomorphism of graded rings so its kernel is a
                homogeneous ideal. This map is also bijective on tangent
                spaces and the maximal ideal of the ring on the right is nilpotent, so the induced map
                \[
                    \kk[t_1, \ldots ,t_{16}]/(t_1,\ldots ,t_{16})^k\to
                    \frac{\locring}{\mm^k+(u_1,  \ldots , u_{48})}
                \]
                is surjective by the nilpotent Nakayama lemma and hence $\varphi_k$ is surjective.
                The canonical projection
               \[
                   \frac{\locring}{(u_1,\ldots ,u_{48})}\to \frac{\hatS}{\mms^2+(u_1,\ldots ,u_{48})}=\frac{\locring}{\mm^2+(u_1,\ldots ,u_{48})}
               \]
                can be clearly lifted to a local ring homomorphism
               \[
                   \frac{\locring}{(u_1,\ldots ,u_{48})}\to \frac{\hatS}{\mms^3+I+(u_1,\ldots ,u_{48})}=\frac{\locring}{\mm^3+(u_1,\ldots ,u_{48})},
               \]
               so the corresponding obstruction map 
               $ob_I\colon (\mms^2/(\mms^3+I))^{\vee}\to \Ext^1(K,M)$ is zero. By the discussion before Theorem~\ref{ref:primaryForQuot:thm} we therefore get that the image of the map
               $\bar{o}$ from~\eqref{eq:baro} lies in $(\mms^3+I)/\mms^3$, so in the kernel of $\varphi_3$.
                The degree two parts of $\ker \varphi_k$ for all $k\geq 3$ are
                equal. In particular, the intersection of all $\ker \varphi_k$
                contains $\im \bar{o}$. Passing to the limit of the
                maps from $\kk[t_1, \ldots ,t_{16}]/((t_1, \ldots ,t_{16})^k + \im
                \bar{o})$ induced by $\varphi_k$, we get a
                surjective ring homomorphism
                \[
                    \varphi\colon \kk[[t_1, \ldots ,t_{16}]]/(\im
                    \bar{o})\onto \frac{\locring}{(u_1, \ldots ,u_{48})}.
                \]
                By Assumption~\ref{assu:dimension} the ring $\kk[t_1, \ldots
                ,t_{16}]/(\im \bar{o})$ localized at $(t_1, \ldots
                ,t_{16})$ has dimension at most four. Completion of a
                Noetherian local ring preserves dimension,
                see~\cite[Corollary~11.19]{Atiyah_Macdonald} so also
                the dimension of $\kk[[t_1, \ldots ,t_{16}]]/(\im \bar{o})$ is at
                most four. Since $\varphi$ is surjective, the dimension of
                $\locring/(u_1, \ldots ,u_{48})$ is at most four. Since
                dividing a local ring by an element of the maximal ideal
                lowers the dimension by at most one,
                see~\cite[Corollary~11.18]{Atiyah_Macdonald}, the dimension of
                $\locring$ is at most $48+4 = 52$, which is the dimension of
                $\mathcal{L}^{\OpQuot}$. Therefore $\mathcal{L}^{\OpQuot}$
                corresponds to a minimal prime of $\OO_{[M]}$ and hence to a
                component $\mathcal{Z}$ passing through $[M]$. By
                Lemma~\ref{ref:tangentToQuotAtLocus:lem} the tangent space at
                every $\kk$-point of this component is at least $64$, so that
                $\mathcal{Z}$ is generically nonreduced.
            \end{proof}

            \begin{example}\label{ex:explicitNonreduced}
        Let $\kk$ be a field of characteristic zero and consider $K$
        generated by the columns of the following matrix
        \setcounter{MaxMatrixCols}{20}
        \[
            \begin{pmatrix}y_1& y_2& y_3& y_4& 0& 0& 0& 0& 0& 0& 0& 0\\ 0&
                0& 0& 0& y_1& y_2& y_3& y_4& 0& 0& 0& 0\\ 0& 0& 0& 0& 0&
                0& 0& 0& y_1& y_2& y_3& y_4\\ 0& 0& y_4& y_1& y_2& 0& 0&
              2y_1+3y_2+y_3+2y_4& y_2+y_3& 2y_1+3y_2+5y_3+5y_4& 0&0
          \end{pmatrix}
        \]
        Using \texttt{Macaulay2} and formula~\eqref{eq:primaryOb} we verify
        Assumptions~\ref{assu:first}-\ref{assu:dimension} and additionally we
        check that the obstruction group $\Ext^1(K, M)$ is concentrated in degree
        $-2$.
        The code for computing explicitly the dimension in
            Assumption~\ref{assu:dimension} is included as an ancillary file
            \emph{VerifyComputationsInArticle.m2} in~\cite{selfReference}.
    \end{example}

    \begin{corollary}\label{ref:genericallynonreducedmatrices:cor}
        The component $(\Bbbk I)^4+\overline{\squareZeroComp{4}{4}{4}}$ of
        $C_4(\MM_8)$ is generically
        nonreduced.
    \end{corollary}
    \begin{proof}
        This component corresponds to the
        generically nonreduced component $\mathcal{Z} \subset \Quotmain$
        constructed in Theorem~\ref{ref:genericallyNonreduced:thm} using
        Example~\ref{ex:explicitNonreduced}.
        Being generically nonreduced, the component $\mathcal{Z}$ has no
        smooth points.
        By the discussion of Section~\ref{ssec:components} also
        $(\Bbbk I)^4+\overline{\squareZeroComp{4}{4}{4}}$ has no smooth points, so is generically
        nonreduced as well.
    \end{proof}

    \begin{proof}[Proof of Main
        Theorem~\ref{ref:mainthm:genericallynonreduced:intro}]
        It is enough to prove that the generically nonreduced irreducible
        component $\mathcal{Z}$
        from Corollary~\ref{ref:genericallynonreducedmatrices:cor}
        induces a generically nonreduced irreducible component of $\CommMat$
        for $n\geq 4$ and $d\geq 8$.
        To
        pass from $n$ to $n+1$, consider the map $i\colon C_n(\MM_d)\to
        C_{n+1}(\MM_d)$ that adds the zero matrix. Forgetting about the last
        matrix gives a map $s\colon C_{n+1}(\MM_d)\to C_n(\MM_d)$ such that
        $si = \id$. Let $\xx$ be a point of $\mathcal{Z}$ that does not lie on
        other components and let
        $\mathcal{Z}'$ be any component containing $i(\xx)$. A deformation
        of $i(\xx)$ induces a deformation of $\xx$ by $s$ so the component
        $\mathcal{Z}'$ consists of square-zero matrices of the same shape as
        in $\mathcal{Z}$. In particular, for any $n\geq 4$ the locus $(\Bbbk
        I)^n+\overline{\squareZeroComp{4}{4}{n}}$ is a component. Now a
        tangent space argument shows that it is nonreduced.
        To pass from $d$ to $d+1$ take the
        concatenation with the principal component of
        $C_n(\MM_1)$ and use Proposition~\ref{ref:productOfComponents:prop}.

        To pass from $\CommMat$ to Quot schemes, fix $r\geq 4$.
        The tuples in the above components of $\CommMat$ correspond to modules
        that have a four-element generating set, hence the components
        correspond to components of $\Quotmain$ which are generically
        nonreduced by the comparison of tangent spaces (or arguing using the
    smooth maps from $\prodfamilystable$).
    \end{proof}

    \begin{remark}
        We decided to avoid mentioning explicitly the \BBname{} decomposition
        in our argument. However, it does underlie the whole construction and
        we discuss here the argument from its perspective. Assumption~\ref{assu:first} together
        with the equivalent of Proposition~\ref{ref:TNTandElementary:prop}
        imply that injection of the \BBname{} cell of the \emph{positive}
        decomposition (constructed at the end of~\S\ref{sec:BBdecomposition})
        is an open immersion near $[F/K]$. Therefore, locally
        near this point, the Quot scheme retracts to its $\Gmult$-fixed
        locus. By Assumption~\ref{assu:tg}, the $\Gmult$-fixed locus locally near our point is just the Grassmannian $\Gr(12, 16)$
        and is smooth. By Assumption~\ref{assu:tg} again, the fiber of the
        retraction has $16$-dimensional tangent space, while by
        Assumption~\ref{assu:dimension} the fiber itself is $4$-dimensional,
        hence the contradiction with being reduced. The details of this line
        of argument will appear in~\cite{Szachniewicz}.
    \end{remark}

    \appendix
    \section{Functorial approach to comparison between $\CommMat$ and
    $\Quotmain$}\label{ssec:functorial}

        The bijection on points obtained in
        Lemma~\ref{ref:descriptionOfCommMat:lem} is not enough to compare
        singularities of $\CommMat$ and $\Quotmain$ or to prove that the map
        $\prodfamilystable\to \Quotmain$ is a morphism. However, the same idea
        can be extended to give such a comparison, using the language of the
        functor of points~\cite[Section~VI.1]{eisenbud_harris}.

        For a $\kk$-algebra $A$, an $A$-point of a scheme $X$ is just a
        morphism $\Spec(A)\to X$ of schemes over $\Spec(\kk)$. In particular,
        when $\kk$ is algebraically closed and $X$ is reasonable (for example,
        locally of finite type), the
        $\kk$-points of $X$ and closed points of $X$ agree.
        The set of $A$-points is denoted by $X(A)$.
        For a morphism of affine schemes $\Spec(A)\to \Spec(B)$ we get a map
        of sets $X(B)\to X(A)$. In this way $X(-)$ becomes a functor from
        $\kk$-algebras to sets.

        The idea of the functor of points is that for every $A$ the set of
        $X(A)$ resembles the set $X(\kk)$ of closed points of $X$, so it is easier
        to deal with sets of $A$-points for every $A$ than with
        $X$ viewed as a locally ringed topological space.
        \begin{example}
            What is an $A$-point of $\Matrices$? Since $\Matrices = \Spec
            \kk[z_{ij}\ |\ 1\leq i,j\leq d]$, an $A$-point of $\Matrices$ is a
            homomorphism $\varphi\colon\kk[z_{ij}]\to A$, i.e., a matrix
            $[\varphi(z_{ij})]_{i,j}$ with entries in $A$. Conversely, having
            such a matrix $[a_{ij}]$ we can uniquely define $\varphi$ by
            $\varphi(z_{ij}) = a_{ij}$. The conclusion
            is that an $A$-point of $\Matrices$ is a $d\times d$ matrix with entries in
            $A$.  Similarly, since $\GL_{\degM} = \Spec\kk[z_{ij},
                \rred{\Delta}\ |\ 1\leq
            i,j\leq d\rred{]/(\Delta \cdot \det [z_{ij}] = 1)}$, an $A$-point of $\GL_{\degM}$ is a matrix
            with entries in $A$ and with invertible determinant which is
            exactly an element of $\GL_{\degM}(A)$.
        \end{example}
        \begin{example}
            The scheme $\CommMat$ is closed in $\Matrices^{n} = \Spec
            \kk[z_{ij}^e\ |\ 1\leq i,j\leq d,1\leq e\leq n]$ given by the
            quadratic equations as explained in~\S\ref{ssec:commMat}. An $A$-point
            of $\CommMat$ is a homomorphism $\varphi\colon \kk[z_{ij}^e\ |\ 1\leq i, j\leq
            d,1\leq e\leq n]/I(\CommMat)\to A$. This gives an $n$-tuple
            $[\varphi(z_{ij}^e)]_{1\leq i,j\leq d}$ for $e=1,2, \ldots ,n$ of
            $d\times d$ matrices and the fact that $\varphi$ factors through
            the quotient by $I(\CommMat)$
            implies exactly that those matrices commute.
            This shows that an $A$-point of $\CommMat$ is just a commuting
            $n$-tuple of $d\times d$ matrices with entries in $A$.
        \end{example}
        For simplicity of notation, let $S_A:=S\tensor_{\kk} A$, $F_A := F\tensor_{\kk} A$ and $V_A :=
        V\tensor_{\kk} A$ where $F := S^{\oplus r}$ is a free module. The Quot scheme is \emph{defined} as a functor of
        points.
        \begin{example}[{\cite[Chapter~5]{fantechi_et_al_fundamental_ag}}]\label{ex:Quotdef}
            We define the scheme $\Quotmain$ by declaring that for every
            $\kk$-algebra $A$ the set of $A$-points of
            $\Quotmain$ is
            \begin{align*}
                \Quotmain(A) := \Big\{ \frac{F_A}{K}\ |\ &K \subset
                F_A\mbox{ is an $S_A$-submodule, the
                $A$-module }\frac{F_A}{K}\mbox{ is locally free
                and}\\
                & \mbox{for every maximal $\mm\triangleleft A$ the $(A/\mm)$-vector
                space }
                \frac{F\tensor_{\kk} (A/\mm)}{\overline{K}}\mbox{ has dimension
                }d\Big\}.
            \end{align*}
            For example, $\Quotmain(\kk) = \left\{ F/K\ |\ K\subset F\mbox{ an
            $S$-submodule, }\dim_{\kk}(F/K) = d \right\}$. An important
            observation is that $K$ is a kernel of a surjection of locally
            free $A$-modules, so it is a locally free $A$-module as well.
            Warning: the $S_A$-module $F_A/K$ is not locally free: it is not
            even torsion-free.
            For a map $A\to A'$ of algebras we declare that the map
            $\Quotmain(A)\to \Quotmain(A')$ sends $\frac{F_A}{K}$ to
            $\frac{F_A}{K}\tensor_{A} A' \simeq
            \frac{F_{A'}}{K\tensor_{A} A'}$. This defines a functor
            $\Quotmain$. It is nontrivial theorem that this functor is
            represented by a scheme,
            see~\cite[Chapter~5]{fantechi_et_al_fundamental_ag}.
        \end{example}

        A technically less demanding way of proving that $\Quotmain$ is a
        scheme is to
        exhibit an open cover by affine
        schemes and use \cite[Theorem~VI.14]{eisenbud_harris}. We present this
        approach,
        without much detail, below.
        \begin{example}\label{ex:monomialBasis}
            Fix $\degM$ elements of the module $F$ that are ``monomials'',
            i.e., that have the form $y_1^{a_1}\cdot  \ldots y_n^{a_n}
            e_j$ for some $a_i\in \mathbb{Z}_{\geq 0}$, $j\in \{1,2, \ldots
            ,r\}$ and denote their
            set by $\lambda$. Inside $\Quotmain$ we consider the locus
            $U_{\lambda}$
            of quotients $F/K$ such that the image of $\lambda$ is a
            $\kk$-linear basis of $F/K$. More precisely, for each $A$, we define
            $U_{\lambda}(A)$ as the set of $S_A$-submodules $K \subset F_A$ such that the
            images of $\lambda$ form an $A$-linear base of $F_A/K$. This is an open condition and the
            resulting open subscheme $U_{\lambda} \subset \Quotmain$ is affine by the argument
            repeating the one done for the Hilbert scheme
            in~\cite[Section~18.1]{Miller_Sturmfels}. The subschemes
            $\left\{U_{\lambda} \right\}_{\lambda}$ form an open cover of the
            Quot scheme. To prove representability, apply
            \cite[Theorem~VI.14]{eisenbud_harris}.
        \end{example}

        Having discussed the existence of $\Quotmain$, we discuss the
        analogues of the maps defined in~\S\ref{ssec:quot_and_commuting}.
    \newcommand{\FA}{F_{A}}%
    \newcommand{\VA}{V_{A}}%
    \begin{lemma}\label{ref:functorialDescriptionOfProdFam:lem}
            Let $A$ be a $\kk$-algebra. The map $(x_1, \ldots ,x_{n}, v_1, \ldots , v_{\genM})\mapsto
        (\frac{\FA}{\ker \pi_M}, \pibarM)$ is a bijection between the $A$-points of
        $\prodfamilystable$ and the set
        \[
            \left\{ \left(\frac{\FA}{K},
            \varphi\right)\ \Big|\ [\FA/K]\mbox{ is an $A$-point of
            }\Quotmain,\ \varphi\colon \FA/K\to \VA\mbox{ is an
            isomorphism of $A$-modules} \right\}.
        \]
        This bijection gives rise to an isomorphism of functors.
    \end{lemma}
    \begin{proof}
        The proof works exactly as in the case $A=\kk$ proven in
        Lemma~\ref{ref:descriptionOfCommMat:lem}.
    \end{proof}
    Arguing as in Lemma~\ref{ref:GLVaction:lem} but for $A$-points, we get a
    map of functors $\GL(V)\times \prodfamilystable\to \prodfamilystable$ so
    a $\GL(V)$-action on $\prodfamilystable$.
    For an algebraic group $G$, a \emph{(Zariski local) principal $G$-bundle} $f\colon P\to T$ is
    a morphism of schemes with $G$-action fiberwise that locally
    trivialises: for every
    point $t\in T$ there exists an open neighbourhood $U$ of $t$ and a $G$-equivariant
    isomorphism $f^{-1}(U)
    \simeq G \times U$ of schemes over $U$.
    \begin{corollary}\label{ref:prodfamilybundleOverQuotmain:cor}
        There is a morphism of schemes $p\colon\prodfamilystable\to
        \Quotmain$
        defined on $A$-points by the formula $(x_1, \ldots ,x_n, v_1, \ldots ,
        v_{\genM})\mapsto \left[\frac{\FA}{\ker \pi_M}\right]$.
        This map makes $\prodfamilystable$ a principal $\GL(V)$-bundle over
        $\Quotmain$.
    \end{corollary}
    \begin{proof}
        By Lemma~\ref{ref:functorialDescriptionOfProdFam:lem}, the map above
        is a map of functors, so by Yoneda's
        Lemma~\cite[Lemma~VI.1]{eisenbud_harris} it gives a morphism of
        schemes $p\colon \prodfamilystable\to \Quotmain$.
        To prove that $p$ is a principal $\GL(V)$-bundle, we can argue
        locally on $\Quotmain$. Choose a point of this scheme and its open
        neighbourhood $Z = \Spec(B)$. The corresponding submodule
        $\mathcal{K} \subset F_B$ has a
        quotient
        \[
            \mathcal{Q} = \frac{F_B}{\mathcal{K}},
        \]
        which is a locally free $B$-module of rank $\degM$. Shrink $Z$ so that
        $\mathcal{Q}$ becomes a free $B$-module and choose an isomorphism
        $\varphi_0\colon \mathcal{Q} \to V_{B}$ of $B$-modules.
        The preimage $p^{-1}(Z)$ is the fiber product $Z \times_{\Quotmain}
        \prodfamilystable$, so an $A$-point of this preimage is
        a morphism $j\colon \Spec(A)\to \Spec(B)$ and an $A$-point of
        $\prodfamilystable$. By
        Lemma~\ref{ref:functorialDescriptionOfProdFam:lem}, this $A$-point
        gives an $S_A$-submodule $K\subset
        F_{A}$ together with an isomorphism of $A$-modules
        $\varphi\colon F_A/K \simeq V_A$. The product is fibered over
        $\Quotmain$ which means that the submodules $K$ and
        $\mathcal{K}\tensor_{B} A$ of $\FA$ are equal.

        Summing up, an $A$-point of $p^{-1}(Z)$ is an isomorphism of
        $A$-modules $\varphi\colon F_{A}/(\mathcal{K}\tensor_{B} A)\to V_A$.
        Let $\overline{\varphi}_0\colon \mathcal{Q}\tensor_B A\to
        V_{B}\tensor_B A = V_{A}$ be obtained from $\varphi_0$, then we get an automorphism
        $\varphi\circ\overline{\varphi}_0^{-1}\colon \VA\to
        \VA$ of the $A$-module $\VA$, hence an $A$-point of $\GL(V)$. Conversely, such an
        $A$-point gives an isomorphism of $\mathcal{Q}\tensor_{B} A$ with
        $\VA$. This shows that the functor of points of $p^{-1}(Z)$ is
        isomorphic to the functor of points of $\GL(V) \times Z$, so by
        Yoneda's lemma we get the claim.
    \end{proof}
    We can repeat the argument of
    Lemma~\ref{ref:functorialDescriptionOfProdFam:lem} for $\CommMat$.
    \begin{lemma}\label{ref:functorialDescriptionOfCommMat:lem}
        Let $A$ be a $\kk$-algebra. The map $(x_1, \ldots ,x_n)\mapsto
            (M, \id)$ is a bijection between the $A$-points of
            $\CommMat$ and the set
            \[
                \left\{ \left(M,
                    \varphi\right)\ \Big|\ M\ \mbox{ a locally free
                    $A$-module, }\varphi\colon M\to \VA\mbox{ is an $A$-linear
                    isomorphism} \right\}\big/ \mathrm{iso}.
            \]
    \end{lemma}
    There is no scheme $X$ whose $A$-points correspond to locally
    free $A$-modules.
    However, there is such an algebraic stack (see~\cite{Olsson} for
    introduction to stacks) and it is called $\Modfin$.
    \begin{corollary}\label{ref:CommMatAsBundleOverMod:cor}
        The variety $\CommMat$ is an $\GL(V)$-bundle over $\Modfin$.
    \end{corollary}
    \begin{proof}
        This follows from Lemma~\ref{ref:functorialDescriptionOfCommMat:lem}
        along the same lines as Corollary~\ref{ref:prodfamilybundleOverQuotmain:cor}.
    \end{proof}
    \begin{corollary}\label{ref:obstructionTheoryForCommMat:cor}
        The variety $\CommMat$ has an obstruction theory, where a point
        $(x_1, \ldots ,x_{n})$ with corresponding module $M$ has
        obstruction group $\Ext^2(M, M)$.
    \end{corollary}
    \begin{proof}
        The map $\CommMat\to \Modfin$ is smooth by
        Corollary~\ref{ref:CommMatAsBundleOverMod:cor}, so the obstruction
        theory for $\Modfin$,
        see~\cite[Proposition~6.5.1]{fantechi_et_al_fundamental_ag}, lifts to an obstruction
        theory for $\CommMat$.
    \end{proof}

    \small

\newcommand{\etalchar}[1]{$^{#1}$}

    \end{document}